\documentclass[10pt]{amsart}

\usepackage{comment, amsmath, amssymb, amsgen, amsthm, amscd, xspace, color, epsfig, float, epic}
\usepackage{genyoungtabtikz}
\usepackage{pifont}
\usepackage{multirow}
\usepackage{threeparttable}
\usepackage[colorlinks,linkcolor=blue,urlcolor=cyan,citecolor=green,pagebackref]{hyperref}
\makeatletter

\newcommand{\Rmnum}[1]{\expandafter\@slowromancap\romannumeral #1@}
\newcommand{\Ann}{\mathrm{Ann}}

\newcommand{\gln}{\mathfrak{gl}(n,\mathbb{C})}

\newcommand{\spn}{\mathfrak{sp}(n,\mathbb{C})}

\newcommand{\sob}{\mathfrak{so}(2n+1,\mathbb{C})}

\newcommand{\sod}{\mathfrak{so}(2n,\mathbb{C})}

\newcommand{\hobox}[3]{\draw (0+#1,0-#2) rectangle (1+#1,-1-#2)++(-0.5,+0.5) node {$ #3$};}

\newcommand{\gkd}{\operatorname{GKdim}}

\newcommand{\domscale}{0.5}
\newcommand{\GK}{\mathrm{GKdim}}

\newcommand{\pf}{\begin{proof}}
\newcommand{\epf}{\end{proof}}
\newcommand{\eq}{\begin{equation}}
\newcommand{\eeq}{\end{equation}}
\newcommand{\eqn}{\begin{equation*}}
\newcommand{\eeqn}{\end{equation*}}

\newcommand{\frg}{\mathfrak{g}}
\newcommand{\frh}{\mathfrak{h}}

\newcommand{\frl}{\mathfrak{l}}

\newcommand{\frn}{\mathfrak{n}}

\newcommand{\frq}{\mathfrak{q}}

\newcommand{\fru}{\mathfrak{u}}

\newcommand{\frsl}{\mathfrak{sl}}

\newcommand{\frso}{\mathfrak{so}}

\newtheorem{Thm}[equation]{Theorem}%[section]

\newtheorem{prop}[equation]{Proposition}
\newtheorem{lem}[equation]{Lemma}

\newtheorem{Rem}[equation]{Remark}

\theoremstyle{definition}
\newtheorem{definition}[equation]{Definition}

\newtheorem{example}[equation]{Example}

\numberwithin{equation}{section}
\begin{document}

\title[ GKdim and Reducibility]{Gelfand-Kirillov dimensions and Reducibility of  scalar type generalized Verma modules for classical Lie algebras }

\author{Zhanqiang Bai and Jing Jiang*}
\address[Bai]{School of Mathematical Sciences, Soochow University, Suzhou 215006, P. R. China}
\email{zqbai@suda.edu.cn}

\address[Jiang]{School of Mathematical Sciences, Soochow University, Suzhou 215006, P. R. China}
\email{jjsd6514@163.com}

\thanks{*Corresponding author}

\bigskip

\begin{abstract}
Let $\mathfrak{g}$ be a classial Lie algebra and $\mathfrak{p}$ be a maximal parabolic subalgebra. Let $M$ be a generalized Verma module induced from a one dimensional representation of $\mathfrak{p}$. Such $M$ is called a  scalar type generalized Verma module. Its simple quotient $L$ is a highest weight moudle. In this paper, we will
determine the reducibility of such scalar type generalized Verma modules by computing the Gelfand-Kirillov dimension of $L$.

\end{abstract}

\maketitle

\section{ Introduction}

Let $\mathfrak{g}$ be a finite-dimensional complex semisimple Lie
algebra and $U(\frg)$ be its universal enveloping algebra.
Fix a Cartan subalgebra $\mathfrak{h}$ and denote by $\Delta$ the root system associated with $(\frg, \frh)$. Choose a positive root system
$\Delta^+\subset\Delta$ and a simple system $\Pi\subset\Delta^+$. Let $\rho$ be the half sum of roots in $\Delta^+$. Let
$\mathfrak{g}=\bar\fru\oplus\mathfrak{h}\oplus
\mathfrak{u}$ be the Cartan decomposition of $\frg$ with nilpotent radical $\fru$ and its dual space $\bar\fru$.   Choose a subset $I\subset\Pi$ and  it generates a subsystem
$\Delta_I\subset\Delta$.
Let $\frq_I$ be the standard parabolic subalgebra corresponding to $I$ with Levi decomposition $\frq_I=\frl_I\oplus\fru_I$. We frequently drop the
subscript $I$ if there is no confusion.

Let $F(\lambda)$ be a finite-dimensional irreducible $\mathfrak{l}$-module with highest weight $\lambda\in\frh^*$. It can also be viewed as a
$\mathfrak{q}$-module with trivial $\mathfrak{u}$-action. The {\it generalized Verma modules} $M_I(\lambda)$ is defined by
\[
M_I(\lambda):=U(\frg)\otimes_{U(\frq)}F(\lambda).
\]
The simple quotient of $M_I(\lambda)$ is denoted by $L(\lambda)$. In the case when $\dim(F(\lambda))=1$, $M_I(\lambda)$ is called a generalized Verma
module of {\it scalar type}.

The reducibility of generalized Verma modules is very importan in representation theory. It is closely related with many other problems, see for example \cite{BXiao, EHW, Ma}. The most important tool is Jantzen's criterion  \cite{JC}. 
But in general, this criterion  is very complicated in application. Kubo \cite{KT} found some practical reducibility criteria  for all the scalar generalized Verma modules associated with maximal parabolic subalgebras. By using Kubo's result, He \cite{He} gave the reducibily for all scalar-type generalized Verma modules of Hermitian symmetric pairs. Then He-Kubo-Zierau \cite{HKZ} found the  reducibily for all scalar-type generalized Verma modules associated with  maximal parabolic subalgebras. Bai-Xiao \cite{BXiao} found the reducibily for all  generalized Verma modules of Hermitian symmetric pairs.

 Gelfand-Kirillov dimension 
is an important invariant for infinite-dimensional algebraic structures. Since Joseph \cite{Jo78}, Gelfand-Kirillov dimension is used to measure the size of representations of Lie algebras and Lie groups.
Recently Bai-Xiao \cite{BX} proved that a scalar-type generalized Verma module $M$ is reducible if and only if the Gelfand-Kirillov dimension of its simple quotient $L$ is strictly less than $\dim(\fru_I)$. By using this result and  their algorithms \cite{BXX} of computing GK dimensions, we can give a new proof for 
the complete list of parameters $\lambda+\rho$ for which the scalar-type generalized Verma modules associated with   maximal parabolic subalgebras are reducible. Then by using same method, we can give the GK dimensions of all scalar type highest weight moudles. Those modules with minimal positive GK dimensions can be easily found from our formula.
%In proposition, p means the subscript of the subsets $\Pi\setminus\{\alpha_p\}$.

The paper is organized as follows. The necessary priliminaries for maximal parabolic subalgebras and Gelfand-Kirillov dimension are given in Section $2$ and $3$. In Section $4$, we give the reducibility of scalar generalized Verma modules for type $A$. In Section $5$, we give the reducibility of scalar generalized Verma modules for types $B$ and $C$. In Section $6$, we give the reducibility of scalar generalized Verma modules for type $D$. In Section $7$, we give the explicit formula of Gelfand-Kirillov dimensions for  scalar type highest weight modules.

\textbf{Acknowledgments.}
%We would like to thank the anonymous referees for valuable comments and suggestions.
This project is supported by the National Science Foundation of China (Grant No. 12171344).

\section{Maximal parabolic subalgebras}
Let $\mathfrak{g}$ be a finite-dimensional complex semisimple Lie
algebra and let $\mathfrak{b}=\mathfrak{h}\oplus\mathop\oplus\limits_{\alpha\in\Delta^+}\frg_\alpha$
be a fixed borel subalgebra of $\frg$. For a maximal parabolic subalgebra $\frq=\frl\oplus\fru$, it corresponds to the subsets  $I=\Pi\setminus\{\alpha_i\}$
for some $\alpha_i\in\Pi$.
In this section, we give the explicit values of $\dim(\fru)$. The results are as follows.
\begin{enumerate}
\item If $\mathfrak{g}$ is of type $A_{n-1}$, for  $I=\Pi\setminus\{\alpha_p\}$, we have
\begin{align*}
	\Delta^+(\frl)=\{e_i-e_j|1\le i< j\le p\quad \text{or}\quad p+1\le i< j\le n\}.
\end{align*}
\begin{align*}
	\Delta(\fru)=\{e_i-e_j|1\le i< j\le p,\: p+1\le i< j\le n\}.
\end{align*}
So	$\dim(\fru)=p(n-p)$.	
\item If $\mathfrak{g}$ is of type $B_n$,
\begin{enumerate}
	\item for  $I=\Pi\setminus\{\alpha_1\}$, we have
	\begin{align*}
	\Delta^+(\frl)=\{e_i\pm e_j|2\le i< j\le n\}\cup\{e_j|2\le j\le n\},
	\end{align*}
	\begin{align*}
	\Delta(\fru)=\Delta^+-\Delta^+(\frl)=\{e_1\pm e_j|1\le j\le n\}\cup\{e_1\}.
	\end{align*}
So	$\dim(\fru)=2(n-1)+1=2n-1$.	
	\item  for  $I=\Pi\setminus\{\alpha_p\}$, we have
	\begin{align*}
	\Delta^+(\frl)&=\{e_i-e_j|1\le i<j\le p\}\cup\{e_i\pm e_j|p+1\le i<j\le n\}\\
	&\cup\{e_j|p+1\le i<j\le n\}.
	\end{align*}	
So	\begin{align*}
\dim(\fru)&=|\Delta^+-\Delta^+(\frl)|\\
&=n^2-C^2_p-(n-p)-2C^2_{n-p}\\
&=2np-\frac{3}{2}p^2+\frac{1}{2}p.
\end{align*}
	
	\item for  $I=\Pi\setminus\{\alpha_n\}$, we have
	\begin{align*}
	\Delta^+(\frl)=\{e_i- e_j|1\le i< j\le n\},
	\end{align*}
	\begin{align*}
	\Delta(\fru)=\Delta^+-\Delta^+(\frl)=\{e_i+ e_j|1\le i<j\le n\}\cup\{e_j|1\le j\le n\}.
	\end{align*}
\end{enumerate}
So $\dim(\fru)=C^2_n+n=\frac{n(n+1)}{2}$.
\item If $\mathfrak{g}$ is of type $C_n$, in the case of $B_n$ and $C_n$, $\dim(\mathfrak{u})$ is equal.
\item If $\mathfrak{g}$ is of type $D_n$
\begin{enumerate} 
	\item for  $I=\Pi\setminus\{\alpha_1\}$, we have
	\begin{align*}
	\Delta^+(\frl)=\{e_i\pm e_j|2\le i< j\le n\},
	\end{align*}
	\begin{align*}
	\Delta(\fru)=\Delta^+-\Delta^+(\frl)=\{e_1\pm e_j|2\le i<j\le n\}.
	\end{align*}
So	$\dim(\fru)=2(n-1)$.	
	\item for  $I=\Pi\setminus\{\alpha_p\}$, we have
	$$
	\Delta^+(\frl)=\{e_i- e_j|1\le i< j\le p\}\cup\{e_i\pm e_j|p+1\le i< j\le n\}.
	$$

So	$\dim(\fru)=|\Delta^+-\Delta^+(\frl)|=n^2-n-C^2_p-2C^2_{n-p}=2pn-\frac{3}{2}p^2-\frac{1}{2}p$.
		\item for  $I=\Pi\setminus\{\alpha_n\}$, we have
	\begin{align*}
	\Delta^+(\frl)=\{e_i- e_j|1\le i< j\le n\},
	\end{align*}
	\begin{align*}
	\Delta(\fru)=\Delta^+-\Delta^+(\frl)=\{e_i+ e_j|1\le i<j\le n\}.
	\end{align*}
So	$\dim(\fru)=C^2_n=\frac{n(n-1)}{2}$.
\end{enumerate}
\end{enumerate}

For scalar generalized Verma modules associated with a maximal parabolic subalgebra, we have the following two lemmas.
\begin{lem}[{\cite[Lemma 1.4]{HKZ}}]\label{scalar}
	When $M_I(\lambda)$ is of scalar type, we will have  $\lambda=z\xi_i $ for some $z\in\mathbb{R}$,
	and $\xi_i$ is the fundamental weight corresponding to the simple root $\alpha_{i}$ associated with the given maximal parabolic subalgebra.
\end{lem}

\begin{lem}[{\cite[Lemma 4.5]{BX}}]\label{GKdown}
	For any $z\in\mathbb{R}$, we have
	\[
	\GK(L((z+1)\xi))\leq\GK(L(z\xi)).
	\]
	In particular, if $M_I(z\xi)$ is reducible, then $M_I((z+1)\xi)$ is also reducible.
\end{lem}

\section{Gelfand-Kirillov dimension}

Let $M$ be a finite generated $U(\mathfrak{g})$-module. Fix a finite dimensional generating subspace $M_0$ of $M$. Let $U_{n}(\mathfrak{g})$ be the standard filtration of $U(\mathfrak{g})$. Set $M_n=U_n(\mathfrak{g})\cdot M_0$ and
\(
\text{gr} (M)=\bigoplus\limits_{n=0}^{\infty} \text{gr}_n M,
\)
where $\text{gr}_n M=M_n/{M_{n-1}}$. Thus $\text{gr}(M)$ is a graded module of $\text{gr}(U(\mathfrak{g}))\simeq S(\mathfrak{g})$.

\begin{definition} The \textit{Gelfand-Kirillov dimension} of $M$  is defined by
	\begin{equation*}
	\operatorname{GKdim} M = \varlimsup\limits_{n\rightarrow \infty}\frac{\log\dim( U_n(\mathfrak{g})M_{0} )}{\log n}.
	\end{equation*}
\end{definition}

%The definition of GK dimensio$n$ is independent of  the choice of $M_0$.

\begin{definition}
	The  \textit{associated variety} of $M$ is defined by
	\begin{equation*}
	V(M):=\{X\in \mathfrak{g}^* \mid f(X)=0 \text{ for all~} f\in \operatorname{Ann}_{S(\mathfrak{g})}(\operatorname{gr} M)\}.
	\end{equation*}
\end{definition}

The above two definitions are independent of the choice of $M_0$, and $\dim V(M)=\gkd M$ (e.g., \cite{NOT}). 

\begin{definition} Let $\mathfrak{g}$ be a finite-dimensional semisimple Lie algebra. Let $I$ be a two-sided idel in $U(\mathfrak{g})$. Then $\text{gr}(U(\mathfrak{g})/I)\simeq S(\mathfrak{g})/\text{gr}I$ is a graded $S(\mathfrak{g})$-module. Its annihilator is $\text{gr}I$. We define its associated variety by
	$$V(I):=V(U(\mathfrak{g})/I)=\{X\in \mathfrak{g}^* \mid p(X)=0\ \mbox{for all $p\in {\text{gr}}I$}\}.
	$$
\end{definition}

We have the following theorem.

\begin{Thm}[ \cite{Jo85}]
	Let $\mathfrak{g}$ be a reductive Lie algebra and $I$ be a primitive ideal in $U(\mathfrak{g})$.Then $V(I)$ is the closure of a single nilpotent coadjoint orbit $\mathcal{O}_I$ in $\mathfrak{g}^*$. In particular, for a highest weight module $L(\lambda)$, we have $V(AnnL(\lambda))=\overline{\mathcal{O}}_{Ann(L(\lambda))}$.
\end{Thm}

Let $G$ be a connected semisimple finite dimensional complex algebraic group with Lie algebra $\mathfrak{g}$. Let $W$ be the Weyl group of $\mathfrak{g}$.  We use $L_w$ to denote the simple highest weight $\mathfrak{g}$-module of highest weight $-w\rho-\rho$ with  $w\in W$.  We denote $I_w=\Ann(L_w)$. Then we have the following theorem.

\begin{Thm}[ \cite{BoB3}]
	$V(I_w)=\overline{\mathcal{O}}_w$.

\end{Thm}

From Vogan \cite{Vo78}, we have $\dim V(I_w)=2\mathrm{GKdim}(L_w)$. 
In \cite{BX1} and  \cite{BXX}, Bai-Xiao-Xie have found an algorithm to compute the  Gelfand-Kirillov dimensions of highest weight modules of classical Lie algebras. We recall their algorithms here.

For  a totally ordered set $ \Gamma $, we  denote by $ \mathrm{Seq}_n (\Gamma)$ the set of sequences $ x=(x_1,x_2,\cdots, x_n) $   of length $ n $ with $ x_i\in\Gamma $.
We say $q=(q_1, \cdots, q_N)$ is the \textit{dual partition} of a partition $p=(p_1, \cdots, p_N)$ and write $q=p^t$ if $q_i$ is the length of $i$-th column of the Young diagram $p$. 
Let $p(x)$ be the shape of the Young tableau $Y(x)$ obtained by applying Robinson-Schensted algorithm to $x\in \mathrm{Seq}_n (\Gamma)$. For convenience, we set $q(x)=p(x)^t$.

For a Young diagram $p$, use $ (k,l) $ to denote the box in the $ k $-th row and the $ l $-th column.
We say the box $ (k,l) $ is \textit{even} (resp. \textit{odd}) if $ k+l $ is even (resp. odd). Let $ p_i ^{ev}$ (resp. $ p_i^{odd} $) be the number of even (resp. odd) boxes in the $ i $-th row of the Young diagram $ p $. 
One can easily check that \begin{equation}\label{eq:ev-od}
p_i^{ev}=\begin{cases}
\left\lceil \frac{p_i}{2} \right\rceil&\text{ if } i \text{ is odd},\\
\left\lfloor \frac{p_i}{2} \right\rfloor&\text{ if } i \text{ is even},
\end{cases}
\quad p_i^{odd}=\begin{cases}
\left\lfloor \frac{p_i}{2} \right\rfloor&\text{ if } i \text{ is odd},\\
\left\lceil \frac{p_i}{2} \right\rceil&\text{ if } i \text{ is even}.
\end{cases}
\end{equation}
Here for $ a\in \mathbb{R} $, $ \lfloor a \rfloor $ is the largest integer $ n $ such that $ n\leq a $, and $ \lceil a \rceil$ is the smallest integer $n$ such that $ n\geq a $. For convenience, we set
\begin{equation*}
p^{ev}=(p_1^{ev},p_2^{ev},\cdots)\quad\mbox{and}\quad p^{odd}=(p_1^{odd},p_2^{odd},\cdots).
\end{equation*}

For $ x=(x_1,x_2,\cdots,x_n)\in \mathrm{Seq}_n (\Gamma) $, set
\begin{equation*}
\begin{aligned}
{x}^-=&(x_1,x_2,\cdots,x_{n-1}, x_n,-x_n,-x_{n-1},\cdots,-x_2,-x_1).
%{}^-{x}=&(-x_n,-x_{n-1},\cdots, -x_2,-x_1,x_1,x_2,\cdots, x_{n-1}, x_n).
\end{aligned}
\end{equation*}

\begin{Thm}[{\cite[Theorem 1.5]{BXX}}]\label{integral}
	Let $\lambda+\rho=(\lambda_1, \lambda_2, \cdots, \lambda_n)\in \mathfrak{h}^*$ be  integral. Then
	
	\[	\emph{GKdim}\:L(\lambda)=\left\{
	\begin{array}{ll}
	\frac{n(n-1)}{2}-\sum\limits_{i\ge 1}(i-1)p(\lambda+\rho)_i &if\:\Delta =A_{n-1}\\	     	  	   
	n^2-\sum\limits_{i\ge 1}(i-1)p((\lambda+\rho)^-)_i^{odd} &if\:\Delta =B_{n}/C_{n}\\
	n^2-n-\sum\limits_{i\ge 1}(i-1)p((\lambda+\rho)^-)_i^{ev} &if\:\Delta =D_{n}
	\end{array}	
	\right.
	\]	
	
\end{Thm}

%As above, we define $ p(x) ^{ev}$ (resp. $ p(x)^{odd} $) by counting even (resp. odd) boxes in $ Y(x) $.
We define three functions $ F_A $, $ F_B $, $ F_D $ as 
\begin{align*}
F_A(x)&=\sum_{k\geq 1} (k-1) p_k,\\
F_B(x)&=\sum_{k\geq 1} (k-1) p_k^{ev},\\
F_D(x)&=\sum_{k\geq 1} (k-1) p_k^{odd},
\end{align*}
where $ p=p(x)=(p_1,p_2,\cdots) $.

\begin{definition}
	Fix $ \lambda+\rho=(\lambda_1,\cdots,\lambda_n) \in \mathfrak{h}^*$.
	
	For $ \mathfrak g= \gln$,  we define $[\lambda]$ to be the set of  maximal subsequences $ x $ of  $ \lambda+\rho $ such that any two entries of $ x $ has an integral difference. 
	
	For $\mathfrak{g}=\spn, \sob $ or $ \sod $, we define $[\lambda] $ to be the set of  maximal subsequences $ x $ of  $ \lambda+\rho $ such that any two entries of $ x $ have an integral  difference or sum. In this case, we set $ [\lambda]_1 $ (resp. $ [\lambda]_2 $) be to the subset of $ [\lambda] $ consisting of sequences with  all entries belonging to $ \mathbb{Z} $ (resp. $ \frac12+\mathbb{Z} $).
	Since there is at most one element in $[\lambda]_1 $ and $[\lambda]_2 $, we denote them by  $(\lambda+\rho)_{(0)}$ and $(\lambda+\rho)_{(\frac{1}{2})}$.
	 We set $[\lambda]_{1,2}=[\lambda]_1\cup [\lambda]_2, \quad [\lambda]_3=[\lambda]\setminus[\lambda]_{1,2}$.
	
%	and set $$ [\lambda]_{1,2}=[\lambda]_1\cup [\lambda]_2, \quad [\lambda]_3=[\lambda]\setminus[\lambda]_{1,2}  .$$
\end{definition}

\begin{definition}
		Let  $ x=(\lambda_{i_1}, \lambda_{i_2},\cdots \lambda_{i_r})\in[\lambda]_3 $. Let $  y=(\lambda_{j_1}, \lambda_{j_2},\cdots, \lambda_{j_p}) $ be the maximal subsequence of $ x $ such that $ j_1=i_1 $ and the difference of any two entries of $ y$ is an integer. Let $ z= (\lambda_{k_1}, \lambda_{k_2},\cdots, \lambda_{k_q}) $ be the subsequence obtained by deleting $ y$ from $ x $, which is possible empty. 
	Define
	$$  \tilde{x}=(\lambda_{j_1}, \lambda_{j_2},\cdots, \lambda_{j_p}, -\lambda_{k_q}, -\lambda_{k_{q-1}},\cdots ,-\lambda_{k_1}).  $$
\end{definition}

%By using the R-S algorithm, we can get a Young tableau $P(\tilde{x})$.  We define $ F_A(\tilde{x}) :=\sum\limits_{i\geq 1}\frac{c_i(c_i-1)}2$ where $ c_i $ is the number of entries in the $ i $-th column of $P(\tilde{x})$. Similarly with the algorithm of types $BCD$, we can define $F_B(x)$ and  $F_D(x)$ for any $x\in [\lambda]$.

%Note we can change the role of $ x_1 $ and $ x_2 $, but $ \tilde{F}_A(x) $ is invariant.

%Now we can state the algorithm for GK dimensions of $ L(\lambda) $ of Lie algebras  $  \gln$, 
%$ \spn $, $ \sob $, $ \sod $ with $ n\geq 1 $.

\begin{Thm}[{\cite[Theorem 5.8]{BXX}}]\label{GKdim}
	The GK dimension of  $ L(\lambda) $  can be computed as follows.
	\begin{enumerate}
		\item If $  \mathfrak{g}= \gln$,
		\[
		\gkd L(\lambda)=\frac{n(n-1)}{2}-\sum _{x\in [\lambda]} F_A(x).
		\]
		
		\item If $  \mathfrak{g}= \mathfrak{sp}(n,\mathbb{C})$, 
		\[
		\gkd L(\lambda)=n^2- F_B((\lambda+\rho)_{(0)}^-)- F_D((\lambda+\rho)_{(\frac{1}{2})}^-)-\sum _{x\in [\lambda]_3} F_A(\tilde{x}).
		\]
		\item  If $  \mathfrak{g} = \mathfrak{so}(2n+1,\mathbb{C}) $,
		\[
		\gkd L(\lambda)=	n^2-F_B((\lambda+\rho)_{(0)}^-)-F_B((\lambda+\rho)_{(\frac{1}{2})}^-)-\sum _{x\in [\lambda]_3} F_A(\tilde{x}).
		\]
		\item  If $  \mathfrak{g} = \mathfrak{so}(2n,\mathbb{C}) $,
		\[
		\gkd L(\lambda)=n^2-n-F_D((\lambda+\rho)_{(0)}^-)-F_D((\lambda+\rho)_{(\frac{1}{2})}^-)-\sum _{x\in [\lambda]_3} F_A(\tilde{x}).
		\]
	\end{enumerate}
\end{Thm}

The follow lemma is very useful in our proof.
\begin{lem}[{\cite[Theorem 1.1]{BX}}]\label{reducible}
	A scalar generalized Verma module $M_I(\lambda)$ is irreducible if and only if $\emph{GKdim}\:(L(\lambda))=\dim(\mathfrak{u})$.
\end{lem}

\section{Reducibility of scalar  generalized Verma modules for type $A_n$}

%Let $\mathfrak{g}$ be a finite-dimensional complex semisimple Lie
%algebra and let $\mathfrak{b}=\mathfrak{h}\oplus\mathop\oplus\limits_{\alpha\in\Delta^+}\frg_\alpha$
%be a fixed borel subalgebra of $\frg$. For a maximal parabolic subalgebra $\frq (\frq=\frb$  or$\ \frq=\frg)$, we appreciate that $\frq$ correspond to the subsets  $\Pi\setminus\{\alpha_i\}$
%for some $\alpha_i\in\Pi$.

From Lemma \ref{scalar} and Lemma \ref{GKdown}, we only need to find the first reducible point of the scalar generalized Verma module $M_{I}(z\xi)$.

First we recall some notation and results from Bai-Xie \cite{BX1}.

\begin{definition}

	We say $\lambda+\rho\in h^{\ast} $ is $(p, q)$-dominant if $\lambda_i-\lambda_j\in\mathbb{Z}_{>0}$ for all $i, j$ such that $1\le i < j\le p $ or $ p + 1\le i < j\le p + q$, where $\lambda+\rho= (\lambda_1,\lambda_2, \dots ,\lambda_n)$. In particular, $\lambda_1 >\lambda_2 >\dots >\lambda_p$ and$ \lambda_{p+1} >\lambda_{p+2} > \dots >\lambda_{p+q}$.
\end{definition}

\begin{lem}[{\cite[Theorem 5.2]{BX}}]\label{dominant}
	Assume that $\lambda\in h^{\ast} $ is $(p, q)$-dominant.
	\begin{enumerate}
		\item If $\lambda_1-\lambda_{p+1}\in\mathbb{Z}$,
		that is, $\lambda$ is an integral weight, then $P(\lambda)$ is a Young tableau
		with at most two columns. And in this case $\emph{GKdim}\:(L(\lambda))= m(n-m)$ where  m is the number of entries in the second column of P($\lambda$).	
		
		\item  If $\lambda_1-\lambda_{p+1}\notin\mathbb{Z}$,
		then $P(\lambda)$ consists of two Young tableaux with single column, and in this case $\emph{GKdim}\:(L(\lambda))= pq$.
		
	\end{enumerate}
\end{lem}
\begin{Rem}
	In this paper, the element $i$ in a Young tableau represents the i-th component of $\lambda+\rho$ $\:or \:(\lambda+\rho)^-$. For example, the Young tableau for $\lambda+\rho=(2,0,3,1)$ is 	\[Y=
		\tiny{\begin{tikzpicture}[scale=\domscale+0.1,baseline=-19pt]
	\hobox{0}{0}{0}
	\hobox{1}{0}{1}
	\hobox{0}{1}{2}
	\hobox{1}{1}{3}
	\end{tikzpicture}}.
	\]
	We use \[\tiny{\begin{tikzpicture}[scale=\domscale+0.1,baseline=-19pt]
	\hobox{0}{0}{2}
	\hobox{1}{0}{4}
	\hobox{0}{1}{1}
	\hobox{1}{1}{3}
	\end{tikzpicture}}\]
	to represent our Young tableau $Y$.
\end{Rem}
\begin{prop}
		Let $ \mathfrak{g}=\mathfrak{sl}(n,\mathbb{C}) $. $M_{I}(\lambda)$ is a scalar generalized Verma module with highest weight $\lambda=z\xi$, where $\xi=\xi_p$ is the fundamental weight corresponding to $\alpha_{p}=e_p-e_{p+1}$. Then $M_{I}(\lambda)$ is reducible if and only if $z\in 1-\min\{p,n-p\}+\mathbb{Z}_{\ge 0}$.
\end{prop}

\begin{proof}
	Take $\mathfrak{g}=\mathfrak{sl}(\frn,\mathbb{C})$, and $\Delta^+(\frl) = \{\alpha_1,\dots,\alpha_{p-1},\alpha_{p+1},\dots,\alpha_n\}$, $\alpha_p=e_p-e_{p+1}$.
	
	When $M_I(\lambda)$ is of scalar type, we know that $\lambda=z\xi $ for some $z\in\mathbb{R}$ by Lemma \ref{scalar},
	and $\xi=(\underbrace{a,\dots,a}_{p},\underbrace{a-1,\dots,a-1}_{n-p})$ for $a=\frac{n-p}{n}$.
	
	In \cite{EHW}, we know that $2\rho=(n-1,n-3,\dots,-n+3,-n+1)$, thus
	$$\rho=(\frac{n-1}{2},\frac{n-3}{2},\dots,\frac{-n+3}{2},\frac{-n+1}{2}),$$
	$$\lambda+\rho=(\underbrace{za+\frac{n-1}{2}\dots,za+\frac{n-2p+1}{2}}_{p},\underbrace{(a-1)z+\frac{n-2p-1}{2},\dots,(a-1)z+\frac{-n+1}{2}}_{n-p}).$$

 When $za+\frac{n-2p+1}{2}-(a-1)z\frac{n-2p-1}{2}=z+1\notin \mathbb{Z}$, we will have $z\notin \mathbb{Z}$, and
	\begin{align*}
	\rm{GKdim}\:(L(\lambda))&=\frac{n(n-1)}{2}-\frac{p(p-1)}{2}-\frac{(n-p)(n-p-1)}{2}\\&=-\frac{p^2}{2}+np-\frac{p^2}{2}=np-p^2=p(n-p)=\dim(\mathfrak{u}).
	\end{align*}
	By Lemma \ref{reducible} we obtain that $M_{I}(\lambda)$ is irreducible.
	
 When $za+\frac{n-2p+1}{2}-(a-1)z\frac{n-2p-1}{2}=z+1\in \mathbb{Z}$, in other words, $z\in \mathbb{Z}$,
	$z\xi+\rho=\lambda +\rho$ is $(p,q)$-dominant. By Lemma \ref{dominant}, we know that ${\rm GKdim}\:(L(\lambda)) = m(n-m)$, where $ m $ is the number of entries in the second column of $P(\lambda)$.
	
 When $m=\min\{p,n-p\}$, ${\rm GKdim}\:(L(\lambda)) = m(n-m)=p(n-p)=\dim(\mathfrak{u})$, in this case $M_{I}(\lambda)$ is irreducible.
	
 When $m(n-m)<p(n-p)=\dim(\mathfrak{u})\Rightarrow (m-p)n<(m+p)(m-p)$,
	in this case $M_{I}(\lambda)$ is reducible. In other words, when $m<\min\{p,n-p\}$, $M_{I}(\lambda)$ is reducible. Suppose that $p\le n-p$.
	\begin{enumerate}
		\item When $z=-1$, we will have
		\[
		\tiny{\begin{tikzpicture}[scale=\domscale+0.25,baseline=-45pt]
		\hobox{0}{0}{p}
		\hobox{0}{1}{p-1}
		\hobox{0}{2}{\vdots}
		\hobox{0}{3}{1} 
		\end{tikzpicture}}\to 
		\tiny{\begin{tikzpicture}[scale=\domscale+0.25,baseline=-45pt]
		\hobox{0}{0}{p}
		\hobox{1}{0}{p+1}
		\hobox{0}{1}{p-1}
		\hobox{0}{2}{\vdots}
		\hobox{0}{3}{1} 
		\end{tikzpicture}}\dashrightarrow 
		\tiny{\begin{tikzpicture}[scale=\domscale+0.25,baseline=-45pt]
		\hobox{0}{0}{n}
		\hobox{1}{0}{p+1}
		\hobox{0}{1}{\vdots}
		\hobox{0}{2}{\vdots}
		\hobox{0}{3}{1} 
		\end{tikzpicture}}=P(\lambda).
		\]
		
		So $m=1<p$, in this case $M_{I}(\lambda)$ is reducible.
		
		\item When  $z=-2$, we will have
		\[
		\tiny{\begin{tikzpicture}[scale=\domscale+0.25,baseline=-45pt]
		\hobox{0}{0}{p}
		\hobox{0}{1}{p-1}
		\hobox{0}{2}{\vdots}
		\hobox{0}{3}{1} 
		\end{tikzpicture}}\to 
		\tiny{\begin{tikzpicture}[scale=\domscale+0.25,baseline=-45pt]
		\hobox{0}{0}{p}
		\hobox{1}{0}{p+1}
		\hobox{0}{1}{p-1}
		\hobox{0}{2}{\vdots}
		\hobox{0}{3}{1}
		\end{tikzpicture}}\to 
		\tiny{\begin{tikzpicture}[scale=\domscale+0.25,baseline=-45pt]
		\hobox{0}{0}{p}
		\hobox{1}{0}{p+2}
		\hobox{0}{1}{p-1}
		\hobox{1}{1}{p+1}
		\hobox{0}{2}{\vdots}
		\hobox{0}{3}{1}  
		\end{tikzpicture}}\dashrightarrow 
		\tiny{\begin{tikzpicture}[scale=\domscale+0.25,baseline=-45pt]
		\hobox{0}{0}{n}
		\hobox{1}{0}{p+2}
		\hobox{0}{1}{n-1}
		\hobox{1}{1}{p+1}
		\hobox{0}{2}{\vdots}
		\hobox{0}{3}{1}  
		\end{tikzpicture}}=P(\lambda).
		\]
		
		Hence $m=2<p$, in this case $M_{I}(\lambda)$ is reducible.
		
		\item When $z=1-p$, we will have
		\[
		\tiny{\begin{tikzpicture}[scale=\domscale+0.25,baseline=-45pt]
		\hobox{0}{0}{p}
		\hobox{0}{1}{p-1}
		\hobox{0}{2}{\vdots}
		\hobox{0}{3}{1} 
		\end{tikzpicture}}\to 
		\tiny{\begin{tikzpicture}[scale=\domscale+0.25,baseline=-45pt]
		\hobox{0}{0}{p}
		\hobox{1}{0}{p+1}
		\hobox{0}{1}{p-1}
		\hobox{0}{2}{\vdots}
		\hobox{0}{3}{1}
		\end{tikzpicture}}\to 
		\tiny{\begin{tikzpicture}[scale=\domscale+0.25,baseline=-45pt]
		\hobox{0}{0}{p}
		\hobox{1}{0}{p+2}
		\hobox{0}{1}{p-1}
		\hobox{1}{1}{p+1}
		\hobox{0}{2}{\vdots}
		\hobox{0}{3}{1}  
		\end{tikzpicture}}\to 
		\tiny{\begin{tikzpicture}[scale=\domscale+0.25,baseline=-45pt]
		\hobox{0}{0}{p}
		\hobox{1}{0}{p+3}
		\hobox{0}{1}{p-1}
		\hobox{1}{1}{p+2}
		\hobox{0}{2}{\vdots}
		\hobox{0}{3}{1} 
		\end{tikzpicture}}\dashrightarrow 
		\tiny{\begin{tikzpicture}[scale=\domscale+0.25,baseline=-45pt]
		\hobox{0}{0}{p}
		\hobox{1}{0}{2p-1}
		\hobox{0}{1}{p-1}
		\hobox{1}{1}{2p-2}
		\hobox{0}{2}{\vdots}
		\hobox{1}{2}{\vdots}
		\hobox{0}{3}{2} 
		\hobox{1}{3}{p+1} 
		\hobox{0}{4}{1}
		\end{tikzpicture}}=P(\lambda).
		\]
		
		Thus $m=(2p-1-(p+1)+1)=p-1<p$, in this case $M_{I}(\lambda)$ is reducible.
		
		\item When $z=-p$, we will have
		\[
	\tiny{\begin{tikzpicture}[scale=\domscale+0.25,baseline=-45pt]
		\hobox{0}{0}{p}
		\hobox{0}{1}{p-1}
		\hobox{0}{2}{\vdots}
		\hobox{0}{3}{1} 
		\end{tikzpicture}}\to 
		\tiny{\begin{tikzpicture}[scale=\domscale+0.25,baseline=-45pt]
		\hobox{0}{0}{p}
		\hobox{1}{0}{p+1}
		\hobox{0}{1}{p-1}
		\hobox{0}{2}{\vdots}
		\hobox{0}{3}{1}
		\end{tikzpicture}}\to 
		\tiny{\begin{tikzpicture}[scale=\domscale+0.25,baseline=-45pt]
		\hobox{0}{0}{p}
		\hobox{1}{0}{p+2}
		\hobox{0}{1}{p-1}
		\hobox{1}{1}{p+1}
		\hobox{0}{2}{\vdots}
		\hobox{0}{3}{1}  
		\end{tikzpicture}}\to 
		\tiny{\begin{tikzpicture}[scale=\domscale+0.25,baseline=-45pt]
		\hobox{0}{0}{p}
		\hobox{1}{0}{p+3}
		\hobox{0}{1}{p-1}
		\hobox{1}{1}{p+2}
		\hobox{0}{2}{\vdots}
		\hobox{0}{3}{1} 
		\end{tikzpicture}}\dashrightarrow 
		\tiny{\begin{tikzpicture}[scale=\domscale+0.25,baseline=-45pt]
		\hobox{0}{0}{p}
		\hobox{1}{0}{2p}
		\hobox{0}{1}{p-1}
		\hobox{1}{1}{2p-1}
		\hobox{0}{2}{\vdots}
		\hobox{1}{2}{\vdots}
		\hobox{0}{3}{2} 
		\hobox{1}{3}{p+1} 
		\hobox{0}{4}{1}
		\end{tikzpicture}}=	P(\lambda).
		\]
	\end{enumerate}
	Thus $m=(2p-(p+1)+1)=p$, in this case $M_{I}(\lambda)$ is irreducible.
	
	So, $z=1-p$ is the first reducible point. Similarly if $p\ge n-p$, $z=p-n+1$ is the first reducible point. 
	Hence for  $ \mathfrak{g}=\mathfrak{sl}(n, \mathbb{C}) $, $M_{I}(\lambda)$ is reducible if and only if $z\in 1-\min\{p,n-p\}+\mathbb{Z}_{\ge 0}$.
\end{proof}

\begin{example}
	Let $\frg=\frsl(3, \mathbb{C})$, We consider the parabolic subalgebra $\frq$ corresponding to the subset $\Pi\setminus\{\alpha_1\}$.
	In this case, we have $p=1$ and $n=3$. So the scalar generalized Verma module $M_{I}(\lambda)$ (with highest weight $\lambda=z\xi_1$) is reducible if and only if 
	$z\in 1-\min\{1,2\}+\mathbb{Z}_{\ge 0}=\mathbb{Z}_{\ge 0}$. This result is the same with \cite{He} and \cite{KT}.

%	From \cite{KT}, we know that $M_q(\Lambda_t)$ is reducible if and only if $t\in\frac{3}{2}+\mathbb{Z}_{\ge 0}$. In this case, the hightest wight of $M_q(\Lambda_t)$ is $\lambda-\rho$, and
%	we should minus $c_k=\frac{n}{2}$. (see \cite{HK} for more details) Hence, if the hightest wight of $M_q(\Lambda_t)$ is $\lambda$, $M_q(\Lambda_t)$ is reducible if and only if $t\in \frac{3}{2}-\frac{n}{2}+\mathbb{Z}_{\ge 0}=\mathbb{Z}_{\ge 0}$.
%	
%	And $t\in 1-min\{1,2\}+\mathbb{Z}_{\ge 0}=\mathbb{Z}_{\ge 0}$.
%	Both ways, we get the same result.
\end{example}

\section{Reducibility of scalar  genralized Verma modules for types $B_n$ and $C_n$}
In this section, we consider the types $B_n$ and $C_n$. Let $M_{I}(\lambda)$ be a scalar generalized Verma module with highest weight $\lambda=z\xi$, where $\xi=\xi_p$ is the fundamental weight corresponding to the simple root $\alpha_{p}$.

\begin{prop}\label{B}
	Let $ \mathfrak{g}=\mathfrak{so}(2n+1,\mathbb{C}) $. $M_{I}(\lambda)$ is reducible if and only if	  
	
	\begin{enumerate}
		\item if $p=\mathrm{1}$, then
		
		$\quad z\in (\frac{3}{2}-n+\mathbb{Z}_{\ge 0})\cup(\mathbb{Z}_{\ge 0})$;		
		\item if $\mathrm{2}\le p\le n-1$, then
		\[	z\in\left\{
		\begin{array}{ll}
		1-n+\frac{p}{2}+\frac{1}{2}\mathbb{Z}_{\ge 0} &\textnormal{if $p$ is even or}\:\:3p<2n+1\\	     	  	   
		\frac{1}{2}-n+\frac{p}{2}+\frac{1}{2}\mathbb{Z}_{\ge 0} &\textnormal{if $p$ is odd and}\:\:3p\ge 2n+1	
		\end{array}	
		\right.;
		\]	
		\item if $p=n$, then
		\[	z\in\left\{
		\begin{array}{ll}
		2-n+\mathbb{Z}_{\ge 0} &\textnormal{if $n$ is  even}\\	     	  	   
		1-n+\mathbb{Z}_{\ge 0} &\textnormal{if $n$ is odd}\\	
		\end{array}	
		\right..
		\]
	\end{enumerate}		     
	
\end{prop}

\begin{proof}
	Take $\mathfrak{g}=\mathfrak{so}(2n+1,\mathbb{C})$, and $\Delta^+(\frl)=\{\alpha_1,\dots,\alpha_{n-1}\}$, 
	where
	$\alpha_i=e_i-e_{i+1}\:\:1\le i\le n-1$ and $\alpha_n=e_n$. When $M_I(\lambda)$ is of scalar type, we know that $\lambda=z\xi $ for some $z\in\mathbb{R}$ by Lemma \ref{scalar}.
	
	By the definition of fundamental weight, we can compute
	that $\xi=(\frac{1}{2},\frac{1}{2},\dots,\frac{1}{2})$,
	then
	$$\lambda+\rho=(\frac{1}{2}z+n-\frac{1}{2},\frac{1}{2}z+n-\frac{3}{2},\dots,\frac{1}{2}z+\frac{1}{2}),$$
	$$(\lambda+\rho)^-=(\frac{1}{2}z+n-\frac{1}{2},\dots,\frac{1}{2}z+\frac{1}{2},-\frac{1}{2}z-\frac{1}{2},\dots,-\frac{1}{2}z+\frac{1}{2}-n).$$ 
	When $n$ is even, we will have the follows.
	\begin{enumerate}
		\item If $z\in \mathbb{Z}$, then
		
		$(\lambda+\rho)^-$ is integral, by Lemma \ref{integral}, we know that
		$${\rm GKdim}\:L(\lambda)=n^2-\sum\limits_{i\ge 1}(i-1)p((\lambda+\rho)^-)_i^{odd}.$$
		When $\frac{1}{2}z+\frac{1}{2}>-\frac{1}{2}z-\frac{1}{2}\Rightarrow z>-1$, we have
		\begin{align*}
		p((\lambda+\rho)^-)=(\underbrace{1,1,\dots,1)}_{2n}\:\:{\rm and}\:\:
		p((\lambda+\rho)^-)^{odd}=(\underbrace{0,1,0,1,\dots,0,1)}_{2n}.
		\end{align*}
	Thus,	\begin{align}\label{0}
		{\rm GKdim}\:L(\lambda)\notag&=n^2-\sum\limits_{i\ge 1}(i-1)p((\lambda+\rho)^-)_i^{odd}\notag\\&=n^2-(0\cdot 0+1\cdot 1+2\cdot 0+3\cdot 1+4\cdot 0+\dots +(2n-2)\cdot 0+(2n-1)\cdot 1)\notag\\&=n^2-(1+3+\dots+2n-1)\notag\\&=n^2-\frac{2n\cdot n}{2}=0<\dim(\mathfrak{u}).
		\end{align}
		By Lemma \ref{reducible} we know that $M_{I}(\lambda)$ is reducible.
		
		\begin{enumerate}
			\item When $z=-1$, we will have		
			\[
	\tiny{\begin{tikzpicture}[scale=\domscale+0.25,baseline=-45pt]
			\hobox{0}{0}{n}
			\hobox{0}{1}{n-1}
			\hobox{0}{2}{\vdots}
			\hobox{0}{3}{1} 
			\end{tikzpicture}}\to 
			\tiny{\begin{tikzpicture}[scale=\domscale+0.25,baseline=-45pt]
			\hobox{0}{0}{n}
			\hobox{1}{0}{n+1}
			\hobox{0}{1}{n-1}
			\hobox{0}{2}{\vdots}
			\hobox{0}{3}{1} 
			\end{tikzpicture}}\dashrightarrow 
			\tiny{\begin{tikzpicture}[scale=\domscale+0.25,baseline=-45pt]
			\hobox{0}{0}{2n}
			\hobox{1}{0}{n+1}
			\hobox{0}{1}{2n-1}
			\hobox{0}{2}{\vdots}
			\hobox{0}{3}{1} 
			\end{tikzpicture}}=P(\lambda).
			\]
		So	$p((\lambda+\rho)^-)=(2,\underbrace{1,1,\dots,1)}_{2n-2}$ and
			$p((\lambda+\rho)^-)^{odd}=(\underbrace{1,1,0,\dots,1,0)}_{2n-1}$. 
			\begin{align*}
			{\rm GKdim}\:L(\lambda)&=n^2-\sum\limits_{i\ge 1}(i-1)p((\lambda+\rho)^-)_i^{odd}\\&=n^2-(0\cdot 1+1\cdot 1+2\cdot 0+3\cdot 1+\dots +(2n-3)\cdot 1+(2n-2)\cdot 0)\\&=n^2-(1+3+\dots+2n-3)\\&=n^2-n^2+2n-1=2n-1<\dim(\mathfrak{u}).
			\end{align*}	
			By Lemma \ref{reducible} we know that $M_{I}(\lambda)$ is reducible.
			
			\item When $z=-2$, we will have
			\[
	\tiny{\begin{tikzpicture}[scale=\domscale+0.25,baseline=-45pt]
			\hobox{0}{0}{n}
			\hobox{0}{1}{n-1}
			\hobox{0}{2}{\vdots}
			\hobox{0}{3}{1} 
			\end{tikzpicture}}\to 
			\tiny{\begin{tikzpicture}[scale=\domscale+0.25,baseline=-45pt]
			\hobox{0}{0}{n}
			\hobox{1}{0}{n+1}
			\hobox{0}{1}{n-1}
			\hobox{0}{2}{\vdots}
			\hobox{0}{3}{1} 
			\end{tikzpicture}}\to 
			\tiny{\begin{tikzpicture}[scale=\domscale+0.25,baseline=-45pt]
			\hobox{0}{0}{n}
			\hobox{1}{0}{n+2}
			\hobox{0}{1}{n-1}
			\hobox{1}{1}{n+1}
			\hobox{0}{2}{\vdots}
			\hobox{0}{3}{1}
			\end{tikzpicture}}\to 
			\tiny{\begin{tikzpicture}[scale=\domscale+0.25,baseline=-45pt]
			\hobox{0}{0}{n+3}
			\hobox{1}{0}{n+2}
			\hobox{0}{1}{n}
			\hobox{1}{1}{n+1}
			\hobox{0}{2}{\vdots}
			\hobox{0}{3}{1}    
			\end{tikzpicture}}\dashrightarrow 
			\tiny{\begin{tikzpicture}[scale=\domscale+0.25,baseline=-45pt]
			\hobox{0}{0}{2n}
			\hobox{1}{0}{n+2}
			\hobox{0}{1}{2n-1}
			\hobox{1}{1}{n+1}
			\hobox{0}{2}{\vdots}
			\hobox{0}{3}{1} 
			\end{tikzpicture}}=P(\lambda).
			\]
		So	$p((\lambda+\rho)^-)^{odd}=(\underbrace{1,1,0,1,0,1\dots,0,1)}_{2n-2}$, and
			\begin{align*}
			{\rm GKdim}\:L(\lambda)&=n^2-\sum\limits_{i\ge 1}(i-1)p((\lambda+\rho)^-)_i^{odd}\\&=n^2-(0\cdot 1+1\cdot 1+2\cdot 0+3\cdot 1+\dots +(2n-4)\cdot 0+(2n-3)\cdot 1)\\&=n^2-(1+3+\dots+2n-3)\\&=n^2-n^2+2n-1=2n-1<\dim(\mathfrak{u}).
			\end{align*}
			By Lemma \ref{reducible} we know that $M_{I}(\lambda)$ is reducible.
			
			\item When $z=2-n$, we will have
			\[
		\tiny{\begin{tikzpicture}[scale=\domscale+0.25,baseline=-45pt]
			\hobox{0}{0}{n}
			\hobox{0}{1}{n-1}
			\hobox{0}{2}{\vdots}
			\hobox{0}{3}{1} 
			\end{tikzpicture}}\to 
			\tiny{\begin{tikzpicture}[scale=\domscale+0.25,baseline=-45pt]
			\hobox{0}{0}{n}
			\hobox{1}{0}{n+1}
			\hobox{0}{1}{n-1}
			\hobox{0}{2}{\vdots}
			\hobox{0}{3}{1} 
			\end{tikzpicture}}\to 
			\tiny{\begin{tikzpicture}[scale=\domscale+0.25,baseline=-45pt]
			\hobox{0}{0}{n}
			\hobox{1}{0}{n+2}
			\hobox{0}{1}{n-1}
			\hobox{1}{1}{n+1}
			\hobox{0}{2}{\vdots}
			\hobox{0}{3}{1}
			\end{tikzpicture}}\to 
			\tiny{\begin{tikzpicture}[scale=\domscale+0.25,baseline=-45pt]
			\hobox{0}{0}{n}
			\hobox{1}{0}{2n-3}
			\hobox{0}{1}{n-1}
			\hobox{1}{1}{2n-4}
			\hobox{0}{2}{\vdots}
			\hobox{1}{2}{\vdots}
			\hobox{0}{3}{3}
			\hobox{1}{3}{n+1}
			\hobox{0}{4}{2}
			\hobox{0}{5}{1}    
			\end{tikzpicture}}\dashrightarrow 
			\tiny{\begin{tikzpicture}[scale=\domscale+0.25,baseline=-45pt]
			\hobox{0}{0}{2n}
			\hobox{1}{0}{2n-2}
			\hobox{0}{1}{2n-1}
			\hobox{1}{1}{2n-3}
			\hobox{0}{2}{n}
			\hobox{1}{2}{2n-4}
			\hobox{0}{3}{\vdots}
			\hobox{1}{3}{\vdots} 
			\hobox{0}{4}{2}
			\hobox{1}{4}{n+1}
			\hobox{0}{5}{1}
			\end{tikzpicture}}=P(\lambda).
			\]
			So $p((\lambda+\rho)^-)^{odd}=(\underbrace{1,1,1,\dots,1}_{n-2},0,1,0,1)$, and
			\begin{align*}
			{\rm GKdim}\:L(\lambda)&=n^2-\sum\limits_{i\ge 1}(i-1)p((\lambda+\rho)^-)_i^{odd}\\&=n^2-(1+2+3+\dots +n-3+n-1+n+1)\\&=n^2-\frac{(n-2)(n-3)}{2}-2n\\&=\frac{1}{2}n^2+\frac{1}{2}n-3<\dim(\mathfrak{u}).
			\end{align*}
			By Lemma \ref{reducible} we know that $M_{I}(\lambda)$ is reducible.
			
			\item When $z=1-n$, we will have
			\[
		\tiny{\begin{tikzpicture}[scale=\domscale+0.25,baseline=-45pt]
			\hobox{0}{0}{n}
			\hobox{0}{1}{n-1}
			\hobox{0}{2}{\vdots}
			\hobox{0}{3}{1} 
			\end{tikzpicture}}\to 
			\tiny{\begin{tikzpicture}[scale=\domscale+0.25,baseline=-45pt]
			\hobox{0}{0}{n}
			\hobox{1}{0}{n+1}
			\hobox{0}{1}{n-1}
			\hobox{0}{2}{\vdots}
			\hobox{0}{3}{1} 
			\end{tikzpicture}}\to 
			\tiny{\begin{tikzpicture}[scale=\domscale+0.25,baseline=-45pt]
			\hobox{0}{0}{n}
			\hobox{1}{0}{n+2}
			\hobox{0}{1}{n-1}
			\hobox{1}{1}{n+1}
			\hobox{0}{2}{\vdots}
			\hobox{0}{3}{1}
			\end{tikzpicture}}\dashrightarrow 
			\tiny{\begin{tikzpicture}[scale=\domscale+0.25,baseline=-45pt]
			\hobox{0}{0}{n}
			\hobox{1}{0}{2n-1}
			\hobox{0}{1}{n-1}
			\hobox{1}{1}{2n-2}
			\hobox{0}{2}{\vdots}
			\hobox{1}{2}{\vdots} 
			\hobox{0}{3}{2}
			\hobox{1}{3}{n+1}
			\hobox{0}{4}{1}
			\end{tikzpicture}}\to 
			\tiny{\begin{tikzpicture}[scale=\domscale+0.25,baseline=-45pt]
			\hobox{0}{0}{2n}
			\hobox{1}{0}{2n-1}
			\hobox{0}{1}{n}
			\hobox{1}{1}{2n-2}
			\hobox{0}{2}{\vdots}
			\hobox{1}{2}{\vdots} 
			\hobox{0}{3}{3}
			\hobox{1}{3}{n+1}
			\hobox{0}{4}{2}	
			\hobox{0}{5}{1}
			\end{tikzpicture}}=P(\lambda).
			\]
			So $p((\lambda+\rho)^-)^{odd}=(\underbrace{1,1,1,\dots,1}_{n})$, and
			\begin{align*}
			{\rm GKdim}\:L(\lambda)&=n^2-\sum\limits_{i\ge 1}(i-1)p((\lambda+\rho)^-)_i^{odd}\\&=n^2-(0\cdot 1+1\cdot 1+2\cdot 1+3\cdot 1+\dots +(n-1)\cdot 1)\\&=n^2-\frac{n(n-1)}{2}\\&=\frac{n(n+1)}{2}=\dim(\mathfrak{u}).
			\end{align*}		
			By Lemma \ref{reducible} we know that $M_{I}(\lambda)$ is irreducible.
			
		\end{enumerate}	

	Thus $z=2-n$ is the first reducible point.
\item If $z\notin\mathbb{Z}$, then	
$(\lambda+\rho)_{(0)}^-=0$,  $(\lambda+\rho)_{(\frac{1}{2})}^-=0$ and $ \lambda+\rho\in [\lambda]_3$.	
%	\end{enumerate}	

	From Theorem \ref{GKdim}  we will have
	\begin{align*}
		{\rm GKdim}\:L(\lambda)&= n^2-F_B((\lambda+\rho)_{(0)}^-)-F_B((\lambda+\rho)_{(\frac{1}{2})}^-)-\sum _{x\in [\lambda]_3} F_A(\tilde{x})\\
		&=n^2-\sum\limits_{i\ge 1}(i-1)p_i(\lambda+\rho)\\&=n^2-(0\cdot 1+1\cdot 1+2\cdot 1+3\cdot 1+\dots +(n-1)\cdot 1)\\&=n^2-\frac{(n)(n-1)}{2}\\&=\frac{n(n+1)}{2}=\dim(\mathfrak{u}).
	\end{align*}
	
	So $M_{I}(\lambda)$ is irreducible.

\end{enumerate}		

Hence if $n$ is even, $z\notin\mathbb{Z}$, $M_{I}(\lambda)$ is irreducible and if $z\in\mathbb{Z}$, $M_{I}(\lambda)$ is reducible if and only if $z\in 2-n+\mathbb{Z}_{\ge 0}$.

	When $n$ is odd,  we will have the follows.	
\begin{enumerate}
	\item If $z\in \mathbb{Z}$ and $z>-1$, from (\ref{0}) we have
	${\rm GKdim}\:L(\lambda)=0<\dim(\mathfrak{u})$.
	By Lemma \ref{reducible}, $M_{I}(\lambda)$ is reducible.
		
		\begin{enumerate}
			\item When $z=-1$, we will have
			\[
		\tiny{\begin{tikzpicture}[scale=\domscale+0.25,baseline=-45pt]
			\hobox{0}{0}{n}
			\hobox{0}{1}{n-1}
			\hobox{0}{2}{\vdots}
			\hobox{0}{3}{1} 
			\end{tikzpicture}}\to 
			\tiny{\begin{tikzpicture}[scale=\domscale+0.25,baseline=-45pt]
			\hobox{0}{0}{n}
			\hobox{1}{0}{n+1}
			\hobox{0}{1}{n-1}
			\hobox{0}{2}{\vdots}
			\hobox{0}{3}{1} 
			\end{tikzpicture}}\dashrightarrow 
			\tiny{\begin{tikzpicture}[scale=\domscale+0.25,baseline=-45pt]
			\hobox{0}{0}{2n}
			\hobox{1}{0}{n+1}
			\hobox{0}{1}{2n-1}
			\hobox{0}{2}{\vdots}
			\hobox{0}{3}{1} 
			\end{tikzpicture}}=P(\lambda).
			\]
		So	$ p((\lambda+\rho)^-)=(2,\underbrace{1,1,\dots,1)}_{2n-2}\:\:{\rm and}\:\:
			p((\lambda+\rho)^-)^{odd}=(\underbrace{1,1,0,\dots,1,0)}_{2n-1}.$
			\begin{align*}
			{\rm GKdim}\:L(\lambda)&=n^2-\sum\limits_{i\ge 1}(i-1)p((\lambda+\rho)^-)_i^{odd}\\&=n^2-(0\cdot 1+1\cdot 1+2\cdot 0+3\cdot 1+\dots +(2n-3)\cdot 1+(2n-2)\cdot 0)\\&=n^2-(1+3+\dots+2n-3)\\&=n^2-n^2+2n-1=2n-1<\dim(\mathfrak{u}).	
			\end{align*}	
			By Lemma \ref{reducible},  $M_{I}(\lambda)$ is reducible.
			
			\item When $z=-2$, we will have	
			\[
		\tiny{\begin{tikzpicture}[scale=\domscale+0.25,baseline=-45pt]
			\hobox{0}{0}{n}
			\hobox{0}{1}{n-1}
			\hobox{0}{2}{\vdots}
			\hobox{0}{3}{1} 
			\end{tikzpicture}}\to 
			\tiny{\begin{tikzpicture}[scale=\domscale+0.25,baseline=-45pt]
			\hobox{0}{0}{n}
			\hobox{1}{0}{n+1}
			\hobox{0}{1}{n-1}
			\hobox{0}{2}{\vdots}
			\hobox{0}{3}{1} 
			\end{tikzpicture}}\to 
			\tiny{\begin{tikzpicture}[scale=\domscale+0.25,baseline=-45pt]
			\hobox{0}{0}{n}
			\hobox{1}{0}{n+2}
			\hobox{0}{1}{n-1}
			\hobox{1}{1}{n+1}
			\hobox{0}{2}{\vdots}
			\hobox{0}{3}{1}
			\end{tikzpicture}}\to 
			\tiny{\begin{tikzpicture}[scale=\domscale+0.25,baseline=-45pt]
			\hobox{0}{0}{n+3}
			\hobox{1}{0}{n+2}
			\hobox{0}{1}{n}
			\hobox{1}{1}{n+1}
			\hobox{0}{2}{\vdots}
			\hobox{0}{3}{1}    
			\end{tikzpicture}}\dashrightarrow 
			\tiny{\begin{tikzpicture}[scale=\domscale+0.25,baseline=-45pt]
			\hobox{0}{0}{2n}
			\hobox{1}{0}{n+2}
			\hobox{0}{1}{2n-1}
			\hobox{1}{1}{n+1}
			\hobox{0}{2}{\vdots}
			\hobox{0}{3}{1} 
			\end{tikzpicture}}=P(\lambda).
			\]
		So	$p((\lambda+\rho)^-)^{odd}=(\underbrace{1,1,0,1,0,1\dots,0,1)}_{2n-2}.$
			\begin{align*}
			{\rm GKdim}\:L(\lambda)&=n^2-\sum\limits_{i\ge 1}(i-1)p((\lambda+\rho)^-)_i^{odd}\\&=n^2-(0\cdot 1+1\cdot 1+2\cdot 0+3\cdot 1+\dots +(2n-4)\cdot 0+(2n-1)\cdot 0)\\&=n^2-(1+3+\dots+2n-3)\\&=n^2=n^2+2n-1=2n-1<\dim(\mathfrak{u}).
			\end{align*}
			By Lemma \ref{reducible},  $M_{I}(\lambda)$ is reducible.	
			
			\item When $z=1-n$, we will have	
			\[
	\tiny{\begin{tikzpicture}[scale=\domscale+0.25,baseline=-45pt]
			\hobox{0}{0}{n}
			\hobox{0}{1}{n-1}
			\hobox{0}{2}{\vdots}
			\hobox{0}{3}{1} 
			\end{tikzpicture}}\to 
			\tiny{\begin{tikzpicture}[scale=\domscale+0.25,baseline=-45pt]
			\hobox{0}{0}{n}
			\hobox{1}{0}{n+1}
			\hobox{0}{1}{n-1}
			\hobox{0}{2}{\vdots}
			\hobox{0}{3}{1} 
			\end{tikzpicture}}\to 
			\tiny{\begin{tikzpicture}[scale=\domscale+0.25,baseline=-45pt]
			\hobox{0}{0}{n}
			\hobox{1}{0}{n+2}
			\hobox{0}{1}{n-1}
			\hobox{1}{1}{n+1}
			\hobox{0}{2}{\vdots}
			\hobox{0}{3}{1}
			\end{tikzpicture}}\dashrightarrow
			\tiny{\begin{tikzpicture}[scale=\domscale+0.25,baseline=-45pt]
			\hobox{0}{0}{n}
			\hobox{1}{0}{2n-1}
			\hobox{0}{1}{n-1}
			\hobox{1}{1}{2n-2}
			\hobox{0}{2}{\vdots}
			\hobox{1}{2}{\vdots} 
			\hobox{0}{3}{2}
			\hobox{1}{3}{n+1}
			\hobox{0}{4}{1}
			\end{tikzpicture}}\to 
			\tiny{\begin{tikzpicture}[scale=\domscale+0.25,baseline=-45pt]
			\hobox{0}{0}{2n}
			\hobox{1}{0}{2n-1}
			\hobox{0}{1}{n}
			\hobox{1}{1}{2n-2}
			\hobox{0}{2}{\vdots}
			\hobox{1}{2}{\vdots} 
			\hobox{0}{3}{3}
			\hobox{1}{3}{n+1}
			\hobox{0}{4}{2}	
			\hobox{0}{5}{1}
			\end{tikzpicture}}=P(\lambda).
			\]
		So	$p((\lambda+\rho)^-)^{odd}=(\underbrace{1,1,1,\dots,1}_{n-1},0,1).$
			\begin{align*}
			{\rm GKdim}\:L(\lambda)&=n^2-\sum\limits_{i\ge 1}(i-1)p((\lambda+\rho)^-)_i^{odd}\\&=n^2-(0\cdot 1+1\cdot 1+2\cdot 1+3\cdot 1+\dots +(n-2)\cdot 1+(n-1)\cdot 0+n\cdot 1)\\&=n^2=(1+2+\dots +n-2)-n\\&=n^2-\frac{(n-1)(n-2)}{2}-n\\&=\frac{1}{2}n^2+\frac{1}{2}n-1<\dim(\mathfrak{u}).
			\end{align*}
			By Lemma \ref{reducible},  $M_{I}(\lambda)$ is reducible.
			
			\item When $z=-n$, we will have	
			\[
	\tiny{\begin{tikzpicture}[scale=\domscale+0.25,baseline=-45pt]
			\hobox{0}{0}{n}
			\hobox{0}{1}{n-1}
			\hobox{0}{2}{\vdots}
			\hobox{0}{3}{1} 
			\end{tikzpicture}}\to 
			\tiny{\begin{tikzpicture}[scale=\domscale+0.25,baseline=-45pt]
			\hobox{0}{0}{n}
			\hobox{1}{0}{n+1}
			\hobox{0}{1}{n-1}
			\hobox{0}{2}{\vdots}
			\hobox{0}{3}{1} 
			\end{tikzpicture}}\to 
			\tiny{\begin{tikzpicture}[scale=\domscale+0.25,baseline=-45pt]
			\hobox{0}{0}{n}
			\hobox{1}{0}{n+2}
			\hobox{0}{1}{n-1}
			\hobox{1}{1}{n+1}
			\hobox{0}{2}{\vdots}
			\hobox{0}{3}{1}
			\end{tikzpicture}}\dashrightarrow 
			\tiny{\begin{tikzpicture}[scale=\domscale+0.25,baseline=-45pt]
			\hobox{0}{0}{n}
			\hobox{1}{0}{2n-1}
			\hobox{0}{1}{n-1}
			\hobox{1}{1}{2n-2}
			\hobox{0}{2}{\vdots}
			\hobox{1}{2}{\vdots}
			\hobox{0}{3}{2} 
			\hobox{1}{3}{n+1}  
			\hobox{0}{4}{2} 
			\end{tikzpicture}}\to
			\tiny{\begin{tikzpicture}[scale=\domscale+0.25,baseline=-45pt]
			\hobox{0}{0}{n}
			\hobox{1}{0}{2n}
			\hobox{0}{1}{n-1}
			\hobox{1}{1}{2n-1}
			\hobox{0}{2}{\vdots}
			\hobox{1}{2}{\vdots}
			\hobox{0}{3}{1} 
			\hobox{1}{3}{n+1}
			\end{tikzpicture}}=P(\lambda).
			\]
		So	$p((\lambda+\rho)^-)^{odd}=(\underbrace{1,1,1,\dots,1}_{n})$.
			\begin{align*}
			{\rm GKdim}\:L(\lambda)&=n^2-\sum\limits_{i\ge 1}(i-1)p((\lambda+\rho)^-)_i^{odd}\\&=n^2-(0\cdot 1+1\cdot 1+2\cdot 1+3\cdot 1+\dots +(n-1)\cdot 1)\\&=n^2-\frac{(n)(n-1)}{2}\\&=\frac{n(n+1)}{2}=\dim(\mathfrak{u}).
			\end{align*}
			By Lemma \ref{reducible},  $M_{I}(\lambda)$ is irreducible. 		
		\end{enumerate}	
			So $z=1-n$ is the first reducible point.
			
		\item If $z\notin\mathbb{Z}$, we will have similar arguments (i.e., when $n$ is even). Thus $M_{I}(\lambda)$ is irreducible when $z\notin\mathbb{Z}$.
	\end{enumerate}
	All in all, when $p=n$, $M_{I}(\lambda)$ is reducible if and only if
	\[	z\in\left\{
	\begin{array}{ll}
	2-n+\mathbb{Z}_{\ge 0} &\text{if $n$ is  even}\\	     	  	   
	1-n+\mathbb{Z}_{\ge 0} &\text{if $n$ is odd}.	
	\end{array}	
	\right.
	\]\\
	
	Now we take  $\Delta^+(\frl)=\{\alpha_2,\dots,\alpha_{n-1},\alpha_{n}\}$.
	
	By the definition of fundamental weight, we can compute that $\xi=(1,\underbrace{0,0,\dots,0}_{n-1})$.
	
	So we have 
	$$\lambda+\rho=(z+n-\frac{1}{2},n-\frac{3}{2},\dots,\frac{1}{2}),$$	
	$$(\lambda+\rho)^-=(z+n-\frac{1}{2},n-\frac{3}{2},\dots,\frac{1}{2},-\frac{1}{2},\dots,-z+\frac{1}{2}-n).$$
	\begin{enumerate}
		\item When $z\in\mathbb{Z}$, we will have the follows.
		
		\begin{enumerate}
			\item When $z>-1$, we will have
			$$p((\lambda+\rho)^-)_i^{odd}=(0,1,0,1,\dots,0,1).$$
		By Theorem \ref{integral}, we get
			\begin{align*}
			{\rm GKdim}\:L(\lambda)&=n^2-\sum\limits_{i\ge 1}(i-1)p((\lambda+\rho)^-)_i^{odd}\\&=n^2-(0\cdot 0+1\cdot 1+2\cdot 0+\dots+(2n-2)\cdot 0+(2n-1)\cdot 1)\\&=n^2-(1+3+\dots+2n-1)\\&=n^2-n^2=0<\dim(\mathfrak{u}).
			\end{align*}
			By Lemma \ref{reducible} we know that $M_{I}(\lambda)$ is reducible.
			
			\item When $z=-1$, we will have
			$$(\lambda+\rho)^-=(n-\frac{3}{2},n-\frac{3}{2},\dots,\frac{1}{2},-\frac{1}{2},\dots,\frac{3}{2}-n,\frac{3}{2}-n),$$
			\[
	\tiny{\begin{tikzpicture}[scale=\domscale+0.25,baseline=-45pt]
			\hobox{0}{0}{n}
			\hobox{1}{0}{1}
			\hobox{0}{1}{n-1}
			\hobox{0}{2}{\vdots}
			\hobox{0}{3}{2} 
			\end{tikzpicture}}\to 
			\tiny{\begin{tikzpicture}[scale=\domscale+0.25,baseline=-45pt]
			\hobox{0}{0}{n+1}
			\hobox{1}{0}{1}
			\hobox{0}{1}{n}
			\hobox{0}{2}{\vdots}
			\hobox{0}{3}{2} 
			\end{tikzpicture}}\dashrightarrow 
			\tiny{\begin{tikzpicture}[scale=\domscale+0.25,baseline=-45pt]
			\hobox{0}{0}{2n-1}
			\hobox{1}{0}{1}
			\hobox{0}{1}{2n-2}
			\hobox{0}{2}{\vdots}
			\hobox{0}{3}{2}
			\end{tikzpicture}}\to
			\tiny{\begin{tikzpicture}[scale=\domscale+0.25,baseline=-45pt]
			\hobox{0}{0}{2n-1}
			\hobox{1}{0}{2n}
			\hobox{0}{1}{2n-2}
			\hobox{1}{1}{1}
			\hobox{0}{2}{\vdots}
			\hobox{0}{3}{2}    
			\end{tikzpicture}}=P(\lambda).
			\]
			
		So	$p((\lambda+\rho)^-)=(2,2,\underbrace{1,\dots,1}_{2n-4})$ and $ p((\lambda+\rho)^-)^{odd}=(1,1,\underbrace{0,1,\dots,0,1}_{2n-4}).$
			\begin{align*}
			{\rm GKdim}\:L(\lambda)&=n^2-\sum\limits_{i\ge 1}(i-1)p((\lambda+\rho)^-)_i^{odd}\\&=n^2-(0\cdot 0+1\cdot 1+2\cdot 0+\dots+(2n-4)\cdot 0+(2n-3)\cdot 1)\\&=n^2-(1+3+\dots+2n-3)\\&=n^2-(n-1)^2=2n-1=\dim(\mathfrak{u}).
			\end{align*}
		
			By Lemma \ref{reducible},  $M_{I}(\lambda)$ is irreducible. 
		\end{enumerate}
Thus $z=0$ is the first reducible point.

		\item When $z\notin \frac{1}{2}\mathbb{Z}$, we will have
		$$(\lambda+\rho)=(z+n-\frac{1}{2},n-\frac{3}{2},\dots,\frac{1}{2}),$$
		$(\lambda+\rho)_{(0)}^-=(0)\:$ and $\:(\lambda+\rho)_{(\frac{1}{2})}^-=(n-\frac{3}{2},\dots,\frac{1}{2},-\frac{1}{2},\dots,\frac{3}{2}-n)$.
		\begin{align*}
		\sum _{x\in [\lambda]_3} F_A(\tilde{x})&=0\cdot 0+1\cdot 1+2\cdot 0+\dots+(2n-4)\cdot 0+(2n-3)\cdot 1\\&=1+3+\dots +2n-3\\&=\frac{(2n-2)(n-1)}{2}=(n-1)^2.
		\end{align*}		
%		If $z+n-\frac{1}{2}\notin(0,\frac{1}{2})+\mathbb{Z}$ and $(\frac{1}{2},1)+\mathbb{Z}$
%		
%		$\Rightarrow z\notin (\frac{1}{2},1)+\mathbb{Z}$ and $(0,\frac{1}{2})+\mathbb{Z}$ 
%		
%		$\Rightarrow z\notin\frac{1}{2}+\mathbb{Z}$
		
	So	we can conclude that 
		\begin{align*}
		{\rm GKdim}\:L(\lambda)=n^2-0-(n-1)^2=2n-1=\dim(\mathfrak{u})
		\end{align*}
		By Lemma \ref{reducible}, $M_{I}(\lambda)$ is irreducible.

%	If $z\in\mathbb{Z}_{\ge 0}$, let $\sum\limits _{x\in [\lambda]_3} F_A(\tilde{x})=k>0$
%	\begin{align*}
%	{\rm GKdim}\:L(\lambda)&=n^2-(n-1)^2-\sum\limits _{x\in [\lambda]_3} F_A(\tilde{x})\\&=2n-1-k<\dim(\mathfrak{u})
%	\end{align*}
%	By Lemma \ref{reducible} we know that $M_{I}(\lambda)$ is reducible.
%	
\item When $z\in \frac{1}{2}+\mathbb{Z}$, we will have the follows.
\begin{enumerate}	
		\item When $z=-\frac{1}{2}$, we will have
		$$(\lambda+\rho)^-=(n-1,n-\frac{3}{2},\dots,\frac{1}{2},-\frac{1}{2},\dots,\frac{3}{2}-n,1-n),$$
		$$(\lambda+\rho)_{(0)}^-=(n-1,1-n)\:\:{\rm and}\:\: (\lambda+\rho)_{(\frac{1}{2})}^-=(n-\frac{3}{2},\dots,\frac{1}{2},-\frac{1}{2},\dots,\frac{3}{2}-n),$$
		\begin{align*}
		\sum\limits_{i\ge 1}(i-1)p((\lambda+\rho)_{(\frac{1}{2})}^-)_i^{odd}=(n-1)^2.
		\end{align*}
		So ${\rm GKdim}\:L(\lambda)=n^2-(0\cdot 0+1\cdot 1)-(n-1)^2=2n-2<\dim(\mathfrak{u})$.
		
		By Lemma \ref{reducible}, $M_{I}(\lambda)$ is reducible.
		
		\item When $z=\frac{3}{2}-n$, we will have		
		$$(\lambda+\rho)^-=(1,n-\frac{3}{2},\dots,\frac{1}{2},-\frac{1}{2},\dots,\frac{3}{2}-n,-1)\:\:{\rm and}\:\: 
		(\lambda+\rho)_{(0)}^-=(1,-1).$$
		
		So ${\rm GKdim}\:L(\lambda)=2n-2<\dim(\mathfrak{u})$. By Lemma \ref{reducible},  $M_{I}(\lambda)$ is reducible.
		
		\item When $z=\frac{1}{2}-n$, we will have		
		$$(\lambda+\rho)^-=(0,n-\frac{3}{2},\dots,\frac{1}{2},-\frac{1}{2},\dots,\frac{3}{2}-n,0)\:\:{\rm and}\:\: 
		(\lambda+\rho)_{(0)}^-=(0,0).$$
		
		So ${\rm GKdim}\:L(\lambda)=2n-1=\dim(\mathfrak{u})$. By Lemma \ref{reducible}, $M_{I}(\lambda)$ is irreducible. 
	\end{enumerate}
So $z=\frac{3}{2}-n$ is the first reducible point.
	\end{enumerate}

	All in all, when $p=1$, $M_{I}(\lambda)$ is reducible if and only if $z\in (\mathbb{Z}_{\ge 0})\cup (\frac{3}{2}-n+\mathbb{Z}_{\ge 0})$.\\

	Finally, we consider the second case.
	
	Take  $\Delta^+(\frl)=\{\alpha_1,\dots,\alpha_{p-1},\alpha_{p+1},\dots,\alpha_{n}\}$. 
%	( $\alpha_i=e_i-e_{i+1} \:\:\:1\le i\le n-1$ and $\alpha_{n}=e_n$).

	By the definition of fundamental weight, we can compute that $\xi=(\underbrace{1,1,\dots,1}_{p},0,\cdots,0)$.
	
	So we have
	$$\lambda+\rho=(z+n-\frac{1}{2},\cdots,z+n-p+\frac{1}{2},n-p-\frac{1}{2},\cdots,\frac{1}{2})$$
	
	First we consider the case when $p$ is odd and $3p\ge 2n+1$.
	
	\begin{enumerate}
		\item When $z\in\mathbb{Z}$, we will have the follows.
		\begin{enumerate}
			\item When $z>-1$, we will have	
			$$p((\lambda+\rho)^-)^{odd}=(0,1,\dots,0,1).$$
			By Theorem \ref{integral}, we get 
			\begin{align*}
			{\rm GKdim}\:L(\lambda)&=n^2-(0\cdot 0+1\cdot 1+2\cdot 0+\dots+(2n-2)\cdot 0+(2n-1)\cdot 1)\\&=n^2-(1+3+\dots+2n-1)\\&=n^2-n^2=0<\dim(\mathfrak{u}).
			\end{align*}	
So	$M_{I}(\lambda)$ is reducible.		
			\item When $z=-1$, we will have
			$$\lambda+\rho=(n-\frac{3}{2},\dots,n-p-\frac{1}{2},n-p-\frac{1}{2},\cdots,\frac{1}{2}).$$
			\[
	\tiny{\begin{tikzpicture}[scale=\domscale+0.25,baseline=-45pt]
			\hobox{0}{0}{p}
			\hobox{1}{0}{p+1}
			\hobox{0}{1}{p-1}
			\hobox{0}{2}{\vdots}
			\hobox{0}{3}{1}
			\end{tikzpicture}}\to 
			\tiny{\begin{tikzpicture}[scale=\domscale+0.25,baseline=-45pt]
			\hobox{0}{0}{n}
			\hobox{1}{0}{p+1}
			\hobox{0}{1}{n-1}
			\hobox{0}{2}{\vdots}
			\hobox{0}{3}{1} 
			\end{tikzpicture}}\dashrightarrow 
			\tiny{\begin{tikzpicture}[scale=\domscale+0.25,baseline=-45pt]
			\hobox{0}{0}{k}
			\hobox{1}{0}{p+1}
			\hobox{0}{1}{k-1}
			\hobox{0}{2}{\vdots}
			\hobox{0}{3}{1}
			\end{tikzpicture}}\to
			\tiny{\begin{tikzpicture}[scale=\domscale+0.25,baseline=-45pt]
			\hobox{0}{0}{k}
			\hobox{1}{0}{k+1}
			\hobox{0}{1}{k-1}
			\hobox{1}{1}{p+1}
			\hobox{0}{2}{\vdots}
			\hobox{0}{3}{1}     
			\end{tikzpicture}}\dashrightarrow
			\tiny{\begin{tikzpicture}[scale=\domscale+0.25,baseline=-45pt]
			\hobox{0}{0}{2n}
			\hobox{1}{0}{k+1}
			\hobox{0}{1}{2n-1}
			\hobox{1}{1}{p+1}
			\hobox{0}{2}{\vdots}
			\hobox{0}{3}{1}     
			\end{tikzpicture}}=P(\lambda).
			\]
			
		In the above tableau, we denote	 $k=2n-p$.
		
		So	$p((\lambda+\rho)^-)^{odd}=(1,1,\underbrace{0,1,\dots,0,1}_{2n-4}).$
			\begin{align}\label{1}
			{\rm GKdim}\:L(\lambda)\notag&=n^2-\sum\limits_{i\ge 1}(i-1)p((\lambda+\rho)^-)_i^{odd}\notag\\&=n^2-(0\cdot 1+1\cdot 1+2\cdot 0+3\cdot 1+\dots +(2n-4)\cdot 0+(2n-3)\cdot 1)\notag\\&=n^2-n^2+2n-1=2n-1<\dim(\mathfrak{u}).
			\end{align}
			By Lemma \ref{reducible}, $M_{I}(\lambda)$ is reducible.
			
			\item When $z=\frac{1}{2}-n+\frac{p}{2}$, we will have
			$$(\lambda+\rho)^-=(\frac{p}{2},\dots,-\frac{p}{2}+1,n-p-\frac{1}{2},\dots,\frac{1}{2},-\frac{1}{2},\dots,-n+p+\frac{1}{2},\frac{p}{2}-1,\dots,-\frac{p}{2}).$$
			So $p((\lambda+\rho)^-)^{odd}=(\underbrace{1,2,\dots,1,2}_{2n-2p},\underbrace{1,\dots,1}_{3p-2n-1},0,1).$
			\begin{align*}
			{\rm GKdim}\:L(\lambda)&=n^2-\sum\limits_{i\ge 1}(i-1)p((\lambda+\rho)^-)_i^{odd}\\&=n^2-(2+2+3\cdot 2 +4+\dots+2n-2p-2+(2n-2p-1)\cdot 2)\\&-(2n-2p+\dots+p-2+p)\\&=n^2-2C^2_{n-p}-2(n-p)^2-2n+p-\frac{(2n-p-1)(3p-2n-2)}{2}\\&=2np-\frac{3}{2}p^2+\frac{p}{2}-1=\dim(\mathfrak{u})-1.
			\end{align*}
			By Lemma \ref{reducible}, $M_{I}(\lambda)$ is reducible.
			
			\item When $z=-\frac{1}{2}-n+\frac{p}{2}$, we will have
			$$(\lambda+\rho)^-=(\frac{p}{2}-1,\dots,-\frac{p}{2},n-p-\frac{1}{2},\dots,\frac{1}{2},-\frac{1}{2},\dots,-n+p+\frac{1}{2},\frac{p}{2},\dots,-\frac{p}{2}+1),$$
			 So $p((\lambda+\rho)^-)^{odd}=(\underbrace{1,2,\dots,1,2}_{2n-2p},\underbrace{1,1,\dots,1}_{3p-2n})$.
			\begin{align*}
			{\rm GKdim}\:L(\lambda)&=n^2-\sum\limits_{i\ge 1}(i-1)p((\lambda+\rho)^-)_i^{odd}\\&=n^2-(2+2+3\cdot 2 +4+\dots+2n-2p-2+(2n-2p-1)\cdot 2)\\&-(2n-2p+\dots+p-1)\\&=n^2-2C^2_{n-p}-2(n-p)^2-\frac{(2n-p-1)(3p-2n)}{2}\\&=2np-\frac{3}{2}p^2+\frac{p}{2}=\dim(\mathfrak{u}).
			\end{align*}
			By Lemma \ref{reducible}, $M_{I}(\lambda)$ is irreducible.
		\end{enumerate}	
So $z=\frac{1}{2}-n+\frac{p}{2}$ is the first reducible point.
		\item When $z\notin\frac{1}{2}\mathbb{Z}$, we have
		$$(\lambda+\rho)^-_{(\frac{1}{2})}=(n-p-\frac{1}{2},\dots,\frac{1}{2},-\frac{1}{2},\dots,-n+p+\frac{1}{2}),$$
		$$(\lambda+\rho)_{(z)}=(z+n-\frac{1}{2},\dots,z+n-p+\frac{1}{2})\in [\lambda]_3.$$
		\begin{align*}
		{\rm GKdim}\:L(\lambda)&=n^2-(n-p)^2-(1+2+\dots+p-1)\\&=2np-\frac{3}{2}p^2+\frac{p}{2}=\dim(\mathfrak{u}).
		\end{align*}
		So $M_{I}(\lambda)$ is irreducible.
		
		\item When $z\in\frac{1}{2}+\mathbb{Z}$, we will have the follows.
		\begin{enumerate}
			\item When $z\ge-\frac{1}{2}$, we will have
			${\rm GKdim}\:L(\lambda)=0$. 
			
			So $M_{I}(\lambda)$ is reducible.
			
			\item When $z=-\frac{3}{2}$, we will have
			\begin{align}\label{2}
			(\lambda+\rho)^-_{(\frac{1}{2})}=(n-p-\frac{1}{2},\dots,\frac{1}{2},-\frac{1}{2},\dots,-n+p+\frac{1}{2}),
			\end{align}
			\begin{align}\label{3}
			(\lambda+\rho)_{(0)}=(n-2,\dots,n-p-1,-n+p+1,\dots,-n+2).
			\end{align}
			We divide the discussion into two cases:
			\begin{enumerate}
				\item	 $p=n-1$. We will have 
				$$p((\lambda+\rho)^-_{(0)})^{odd}=(\underbrace{0,1,\dots,0,1}_{2p-2}).$$
				\begin{align*}
				{\rm GKdim}\:L(\lambda)&=n^2-(n-p)^2-(1+3+\dots+2p-3)\\&=n^2-(n-p)^2-(1+3+\dots+2p-3+2p-1)+2p-1\\&=2np-2p^2+2p-1\\&=2np-\frac{3}{2}p^2+\frac{p}{2}-\frac{1}{2}p^2+\frac{3}{2}p-1\\&<2np-\frac{3}{2}p^2+\frac{p}{2}=\dim(\mathfrak{u}).
				\end{align*}
				
		\item $p<n-1$. We will have
		$$p((\lambda+\rho)^-_{(0)})^{odd}=(\underbrace{0,1,\dots,0,1}_{2p}).$$
		\begin{align*}
		{\rm GKdim}\:L(\lambda)&=n^2-(n-p)^2-(1+3+\dots+2p-1)\\&=2np-2p^2\\&=2np-\frac{3}{2}p^2+\frac{p}{2}-\frac{1}{2}p^2-\frac{1}{2}p\\&<2np-\frac{3}{2}p^2+\frac{p}{2}=\dim(\mathfrak{u}).
		\end{align*}
			\end{enumerate}	
		
			All in all, when $z=-\frac{3}{2}$, $M_{I}(\lambda)$ is reducible.
			
			\item When $z=1-n+\frac{p}{2}$, we will have
			$$(\lambda+\rho)^-=(\frac{p}{2}+\frac{1}{2},\dots,-\frac{p}{2}+\frac{3}{2},n-p-\frac{1}{2},\dots,\frac{1}{2},-\frac{1}{2},\dots,-n+p+\frac{1}{2},\frac{p}{2}-\frac{3}{2},\dots,-\frac{p}{2}-\frac{1}{2}),$$
			$$(\lambda+\rho)^-_{(0)}=(\frac{p}{2}+\frac{1}{2},\dots,-\frac{p}{2}+\frac{3}{2},\frac{p}{2}-\frac{3}{2},\dots,-\frac{p}{2}-\frac{1}{2}),$$	
			$$(\lambda+\rho)^-_{(\frac{1}{2})}=(n-p-\frac{1}{2},\dots,\frac{1}{2},-\frac{1}{2},\dots,-n+p+\frac{1}{2}).$$
			\[
	\tiny{\begin{tikzpicture}[scale=\domscale+0.25,baseline=-45pt]
			\hobox{0}{0}{p}
			\hobox{0}{1}{p-1}
			\hobox{0}{2}{\vdots}
			\hobox{0}{3}{1}
			\end{tikzpicture}}\to 
			\tiny{\begin{tikzpicture}[scale=\domscale+0.25,baseline=-45pt]
			\hobox{0}{0}{p}
			\hobox{1}{0}{p+1}
			\hobox{0}{1}{p-1}
			\hobox{0}{2}{\vdots}
			\hobox{0}{3}{1} 
			\end{tikzpicture}}\dashrightarrow 
			\tiny{\begin{tikzpicture}[scale=\domscale+0.25,baseline=-45pt]
			\hobox{0}{0}{p}
			\hobox{1}{0}{2p-2}
			\hobox{0}{1}{p-1}
			\hobox{1}{1}{\vdots}
			\hobox{0}{2}{\vdots}
			\hobox{1}{2}{p+1}
			\hobox{0}{3}{1}
			\end{tikzpicture}}\dashrightarrow
			\tiny{\begin{tikzpicture}[scale=\domscale+0.25,baseline=-45pt]
			\hobox{0}{0}{2p}
			\hobox{1}{0}{2p-2}
			\hobox{0}{1}{2p-1}
			\hobox{1}{1}{\vdots}
			\hobox{0}{2}{\vdots}
			\hobox{1}{2}{p+1}
			\hobox{0}{3}{1}
			\end{tikzpicture}}=P(\lambda).
			\]
		So	$p((\lambda+\rho)^-_{(0)})^{odd}=(\underbrace{1,\dots,1}_{p-2},1,0,1).$
			\begin{align*}
			{\rm GKdim}\:L(\lambda)&=n^2-(n-p)^2-(1+2+\dots+p-2+p)\\&=n^2-(n-p)^2-p-\frac{(p-1)(p-2)}{2}\\&=2np-\frac{3}{2}p^2+\frac{p}{2}-1=\dim(\mathfrak{u})-1.
			\end{align*}
			By Lemma \ref{reducible}, $M_{I}(\lambda)$ is reducible.
			
			\item When $z=-n+\frac{p}{2}$, we will have
			$$(\lambda+\rho)^-_{(0)}=(\frac{p}{2}-\frac{1}{2},\dots,-\frac{p}{2}+\frac{1}{2},\frac{p}{2}-\frac{1}{2},\dots,-\frac{p}{2}+\frac{1}{2}).$$
			\begin{align*}
			{\rm GKdim}\:L(\lambda)&=n^2-(n-p)^2-(1+2+\dots+p-1)\\&=n^2-(n-p)^2-\frac{p(p-1)}{2}\\&=2np-\frac{3}{2}p^2+\frac{p}{2}=\dim(\mathfrak{u}).
			\end{align*}
			By Lemma \ref{reducible}, $M_{I}(\lambda)$ is irreducible.
		\end{enumerate}
		
		Hence $z=1-n+\frac{p}{2}$ is the first reducible point. 
	\end{enumerate}
	
So when $p$ is odd and $3p\ge 2n+1$, $M_{I}(\lambda)$ is reducible if and only if $z\in(\frac{1}{2}-n+\frac{p}{2}+\mathbb{Z}_{\ge 0})\cup(1-n+\frac{p}{2}+\mathbb{Z}_{\ge 0})=\frac{1}{2}-n+\frac{p}{2}+\frac{1}{2}\mathbb{Z}_{\ge 0}$.	
	
\hspace{1cm}	
	
	Next we consider the case when $p$ is even.
	\begin{enumerate}
		\item When $z\in\mathbb{Z}$, we have the follows.
		\begin{enumerate}
			\item When $z\ge 0$, we will have	
			
			${\rm GKdim}\:L(\lambda)=0<\dim(\mathfrak{u})\Rightarrow M_{I}(\lambda)$ is reducible.
			
			\item When $z=-1$, 			
			 $M_{I}(\lambda)$ is reducible since  (\ref{1}).
			
			\item When $z=1-n+\frac{p}{2}$, we divide the discussion into three cases:
		\begin{enumerate}	
			\item $3p\le 2n-2$. We have
			$$p((\lambda+\rho)^-)^{odd}=(\underbrace{1,2,\dots,1,2}_{p-2},1,1,1,1,\underbrace{0,1,\dots,0,1}_{2n-3p-2}).$$
			\begin{align*}
			{\rm GKdim}\:L(\lambda)&=n^2-(2+2+3\cdot 2 +4+\dots+p-4+(p-3)\cdot 2)\\
			&\quad-(p-2+p-1+p+p+1)-(p+3+\dots+2n-2p-1)\\&=n^2-\frac{(p-2)(p-4)}{4}-\frac{(p-2)^2}{2}-3p+3-\frac{(2n-p)(n-\frac{3p}{2})}{2}\\&=2np-\frac{3}{2}p^2+\frac{p}{2}-1=\dim(\mathfrak{u})-1.
			\end{align*}
			\item $n=p+1$ and $p>2$. We have 
				$$p((\lambda+\rho)^-)^{odd}=(1,2,\underbrace{1,1,\dots,1}_{2n-p-6},\underbrace{0,1,\dots,0,1}_{2p-2n+6}).$$
			\begin{align*}
			{\rm GKdim}\:L(\lambda)&=n^2-(2+2+3+\dots+2n-p-5)\\
		&\quad -(2n-p-3+2n-p-1+\dots+p+1)\\&=n^2-1-\frac{(2n-p-4)(2n-p-5)}{2}-\frac{(2n-2)(p-n+3)}{2}\\&=2np-\frac{1}{2}p^2+5n-\frac{7}{2}p-np-8\\&=2np-\frac{3}{2}p^2+\frac{1}{2}p-3=\dim(\mathfrak{u})-3.
			\end{align*}
			
			\item $3p\ge 2n$ and $n\neq p+1$. We have
			$$p((\lambda+\rho)^-)^{odd}=(\underbrace{1,2,\dots,1,2}_{p-2},\underbrace{1,1,\dots,1}_{2n-3p+2},\underbrace{0,1,\dots,0,1}_{3p-2n+2}).$$
			\begin{align*}
			{\rm GKdim}\:L(\lambda)&=n^2-(2+2+3\cdot 2 +4+\dots+p-4+(p-3)\cdot 2)\\
				&\quad-(p-2+p-1
		 +\dots+2n-2p-1)\\
		 	&\quad-(2n-2p+1+2n-2p+3\dots+p+2)\\&=n^2-\frac{(p-2)(p-4)}{4}-\frac{(p-2)^2}{2}-\frac{(2n-p-3)(2n-3p+2)}{2}\\
		 	&\quad-\frac{(2n-p+3)(\frac{3}{2}p-n+\frac{3}{2})}{2}\\
		 	&=2np-\frac{3}{2}p^2-\frac{3p}{2}+n-\frac{13}{4}<\dim(\mathfrak{u}).
			\end{align*}
		\end{enumerate}

			To sum up, $z=1-n+\frac{p}{2}$ is a reducible point.
			
			\item When $z=-n+\frac{p}{2}$, we will have
			$$(\lambda+\rho)^-=(\frac{p}{2}-\frac{1}{2},\dots,-\frac{p}{2}+\frac{1}{2},n-p-\frac{1}{2},\dots,\frac{1}{2},-\frac{1}{2},\dots,-n+p+\frac{1}{2},\frac{p}{2}-\frac{1}{2},\dots,-\frac{p}{2}+\frac{1}{2}).$$	
			Also we divide the discussion into two cases:
				\begin{enumerate}	
			\item $3p\ge 2n$. We have
				$$p((\lambda+\rho)^-)^{odd}=(\underbrace{1,2,\dots,1,2}_{2n-2p},\underbrace{1,1,\dots,1}_{3p-2n}).$$
			\begin{align*}
			{\rm GKdim}\:L(\lambda)&=n^2-(2+2+3\cdot 2 +4+\dots+2n-2p-2+(2n-2p-1)\cdot 2)\\&
			\quad-(2n-2p+2n-2p+1+\dots+p-1)\\&=n^2-2\cdot\frac{(n-p)(n-p-1)}{2}-2\cdot\frac{(2n-2p)(n-p)}{2}-\frac{(2n-p-1)(3p-2n)}{2}\\&=2np-\frac{3}{2}p^2+\frac{1}{2}p=\dim(\mathfrak{u}).
			\end{align*}
		\item $3p<2n$. We have
		$$p((\lambda+\rho)^-)^{odd}=(\underbrace{1,2,\dots,1,2}_{p},\underbrace{0,1,\dots,0,1}_{2n-3p}).$$
		\begin{align*}
		{\rm GKdim}\:L(\lambda)&=n^2-(2+2+3\cdot 2 +4+\dots+p-2+(p-1)\cdot 2)\\
		&\quad-(p+1+p+3+\dots+2n-2p-1)\\&=n^2-\frac{p(p-2)}{4}-\frac{p^2}{2}-\frac{(2n-p)(n-\frac{3p}{2})}{2}\\&=2np-\frac{3}{2}p^2+\frac{p}{2}=\dim(\mathfrak{u}).
		\end{align*}	
	\end{enumerate}	
		
			To sum up, $z=-n+\frac{p}{2}$ is an irreducible point. 		
		\end{enumerate}
	
	So $z=1-n+\frac{p}{2}$ is the first reducible point.
		\item When  $z\notin\frac{1}{2}\mathbb{Z}$,
		we can easily get that $M_{I}(\lambda)$ is irreducible.
		
		\item When  $z\in\frac{1}{2}+\mathbb{Z}$, we have the follows.
		\begin{enumerate}
			\item When $z\ge -\frac{1}{2}$, we will have
			\begin{align}\label{17}
			{\rm GKdim}\:L(\lambda)\notag&=n^2-(1+3+\dots+2p-1)-(n-p)^2\notag\\&=n^2-p^2-n^2+2np-p^2\notag\\&=2np-2p^2<\dim(\mathfrak{u}).
			\end{align}		
			By Lemma \ref{reducible}, $M_{I}(\lambda)$ is reducible.	
			
			\item When $z=-\frac{3}{2}$, by (\ref{2}) and (\ref{3}), $M_{I}(\lambda)$ is reducible.	
			
			\item When $z=\frac{3}{2}-n+\frac{p}{2}$, we will have	
			$$\lambda+\rho=(\frac{p}{2}+1,\dots,-\frac{p}{2}+2,n-p-\frac{1}{2},\dots,\frac{1}{2}),$$
			$$(\lambda+\rho)^-_{(0)}=(\frac{p}{2}+1,\dots,-\frac{p}{2}+2,\frac{p}{2}-2,\dots,-\frac{p}{2}-1),$$
			$$p((\lambda+\rho)^-)^{odd}=(\underbrace{1,1,\dots,1}_{p-3},1,0,1,0,1,0).$$
			\begin{align*}
			{\rm GKdim}\:L(\lambda)&=n^2-(1+2+3+\dots+p-3+p-1+p+1)-(n-p)^2\\&=n^2-\frac{(p-2)(p-3)}{2}-2p-(n-p)^2\\&=2np-\frac{3}{2}p^2+\frac{p}{2}-3=\dim(\mathfrak{u})-3.
			\end{align*}	
			So $z=\frac{3}{2}-n+\frac{p}{2}$ is a reducible point.	
			
			\item When $z=\frac{1}{2}-n+\frac{p}{2}$, we will have
				
			$$\lambda+\rho=(\frac{p}{2},\dots,-\frac{p}{2}+1,n-p-\frac{1}{2},\dots,\frac{1}{2}),$$
			$$(\lambda+\rho)^-_{(0)}=(\frac{p}{2},\dots,-\frac{p}{2}+1,\frac{p}{2}-1,\dots,-\frac{p}{2}),$$
			$$p((\lambda+\rho)^-)^{odd}=(\underbrace{1,1,\dots,1}_{p},1,0).$$
			\begin{align}\label{4}
			{\rm GKdim}\:L(\lambda)\notag&=n^2-(1+2+3+\dots+p-2+p-1)-(n-p)^2\notag\\&=n^2-\frac{p(p-1)}{2}-(n-p)^2\notag\\&=2np-\frac{3}{2}p^2+\frac{p}{2}=\dim(\mathfrak{u}).
			\end{align}
		\end{enumerate}
		By Lemma \ref{reducible}, $M_{I}(\lambda)$ is irreducible.		
	
	\end{enumerate}

	So	$z=\frac{3}{2}-n+\frac{p}{2}$ is the first reducible point. 

All in all, when $p$ is even, $M_{I}(\lambda)$ is reducible if and only if $z\in(1-n+\frac{p}{2}+\mathbb{Z}_{\ge 0})\cup(\frac{3}{2}-n+\frac{p}{2}+\mathbb{Z}_{\ge 0})=1-n+\frac{p}{2}+\frac{1}{2}\mathbb{Z}_{\ge 0}$.\\

	Finally we consider the case when $p$ is odd and $3p<2n+1$.
	\begin{enumerate}
		\item When $z\in\mathbb{Z}_{\ge 0}$, we will have
		
		${\rm GKdim}\:L(\lambda)=0$.	
		
			\item When $z=-1$, we will have
		
		${\rm GKdim}\:L(\lambda)=2n-1$.
		
			\item When $z=\frac{3}{2}-n+\frac{p}{2}$, we will have
		$$\lambda+\rho=(\frac{p}{2}+1,\dots,-\frac{p}{2}+2,n-p-\frac{1}{2},\dots,\frac{1}{2}).$$
		
		%		\item If $p$ is even
		%		$$(\lambda+\rho)^-_{(0)}=(\frac{p}{2}+1,\dots,-\frac{p}{2}+2,\frac{p}{2}-2,\dots,-\frac{p}{2}-1)$$
		%		$$p((\lambda+\rho)^-_{(0)})^{odd}=(\underbrace{1,1,\dots,1}_{p-3},1,0,1,0,1,0)$$	
		%		\begin{align*}
		%		{\rm GKdim}\:L(\lambda)&=n^2-(1+2+3+\dots+p-3+p-1+p+1)-(n-p)^2\\&=n^2-\frac{(p-2)(p-3)}{2}-(n-p)^2-2p\\&=2np-\frac{3}{2}p^2+\frac{p}{2}-3=\dim(\mathfrak{u})-3
		%		\end{align*}	
		
		Since $p$ is odd, we have the follows.
		\begin{enumerate}
			\item When $3p=2n-1$, we will have
			$$p((\lambda+\rho)^-_{(0)})^{odd}=(\underbrace{1,2,\dots,1,2}_{p-3},1,1,1,1,0,1).$$	
			\begin{align*}
			{\rm GKdim}\:L(\lambda)&=n^2-(2+2+3\cdot 2+\dots+p-5+(p-4)\cdot 2)-(p-3+\dots+p+p+2)\\&=n^2-\frac{(p-3)(p-5)}{4}-\frac{(p-3)^2}{2}-5p+4\\&=n^2-\frac{3}{4}p^2-\frac{17}{4}\\&=2np-\frac{3}{2}p^2+\frac{p}{2}-4=\dim(\mathfrak{u})-4.
			\end{align*}		
			
			\item When $3p<2n-1$, we will have
			$$p((\lambda+\rho)^-_{(0)})^{odd}=(\underbrace{1,2,\dots,1,2}_{p-3},1,1,1,1,1,1,\underbrace{0,1,\dots,0,1}_{2n-3p-3}).$$	
			\begin{align*}
			{\rm GKdim}\:L(\lambda)&=n^2-(2+2+3\cdot 2+\dots+p-5+(p-4)\cdot 2)-(p-3+\dots+p+1+p+2)\\&-(p+4+p+6+\dots+2n-2p-1)\\&=n^2-\frac{(p-3)(p-5)}{4}-\frac{(p-3)^2}{2}-5p+5-\frac{(2n-p+1)(n-\frac{3}{2}p-\frac{1}{2})}{2}\\&=2np-\frac{3}{2}p^2+\frac{p}{2}-3=\dim(\mathfrak{u})-3.
			\end{align*}
		\end{enumerate}

		Hence when $z=\frac{3}{2}-n+\frac{p}{2}$, $M_{I}(\lambda)$ will be reducible.

			\item When $z=\frac{1}{2}-n+\frac{p}{2}$, we will have
		$$\lambda+\rho=(\frac{p}{2},\dots,-\frac{p}{2}+1,n-p-\frac{1}{2},\dots,\frac{1}{2}).$$
		Since $p$ is odd, we have
			$$p((\lambda+\rho)^-_{(0)})^{odd}=(\underbrace{1,2,\dots,1,2}_{p-1},1,1,\underbrace{0,1,\dots,0,1}_{2n-3p-1}).$$	
			\begin{align*}
			{\rm GKdim}\:L(\lambda)&=n^2-(2+2+3\cdot 2+\dots+p-3+(p-2)\cdot 2+p-1)\\&-(p+p+2+p+4+\dots+2n-2p-1)\\&=n^2-\frac{(p-1)(p+1)}{4}-\frac{(p-1)^2}{2}-\frac{(2n-p-1)(n-\frac{3}{2}p+\frac{1}{2})}{2}\\&=2np-\frac{3}{2}p^2+\frac{p}{2}=\dim(\mathfrak{u}).
			\end{align*}	
			So when $z=\frac{1}{2}-n+\frac{p}{2}$, $M_{I}(\lambda)$ is irreducible.
	
		\item When $z\in\frac{1}{2}+\mathbb{Z}$ and $z\ge -\frac{1}{2}$, by(\ref{17}) we have
		
		${\rm GKdim}\:L(\lambda)=2np-2p^2$.

		\item When $z=1-n+\frac{p}{2}$, we will have
				$$(\lambda+\rho)^-_{(0)}=(\frac{p}{2}+\frac{1}{2},\dots,-\frac{p}{2}+\frac{3}{2},\frac{p}{2}-\frac{3}{2},\dots,-\frac{p}{2}-\frac{1}{2}),$$
		$$p((\lambda+\rho)^-_{(0)})^{odd}=(\underbrace{1,1,\dots,1}_{p-2},1,0,1,0).$$	
		\begin{align*}
		{\rm GKdim}\:L(\lambda)&=n^2-(1+2+3+\dots+p-3+p-2+p)-(n-p)^2\\&=n^2-\frac{(p-1)(p-2)}{2}-(n-p)^2-p\\&=2np-\frac{3}{2}p^2+\frac{p}{2}-1=\dim(\mathfrak{u})-1.
		\end{align*}

		Hence when $z=1-n+\frac{p}{2}$, $M_{I}(\lambda)$ is reducible.

		\item When $z=-n+\frac{p}{2}$, we will have
		$$(\lambda+\rho)^-_{(0)}=(\frac{p}{2}-\frac{1}{2},\dots,-\frac{p}{2}+\frac{1}{2},\frac{p}{2}-\frac{1}{2},\dots,-\frac{p}{2}+\frac{1}{2}).$$	
		\begin{align*}
		{\rm GKdim}\:L(\lambda)&=n^2-(1+2+\dots+p-1)-(n-p)^2\\&=n^2-\frac{p(p-1)}{2}-(n-p)^2\\&=2np-\frac{3}{2}p^2+\frac{p}{2}=\dim(\mathfrak{u}).
		\end{align*}	
			
		So when $z=-n+\frac{p}{2}$, $M_{I}(\lambda)$ is irreducible.
	\end{enumerate}

	In other words, $z=1-n+\frac{p}{2}$ and $z=\frac{3}{2}-n+\frac{p}{2}$ are the first two reducible points, and in this case, $M_{I}(\lambda)$ is reducible if and only if $z\in(1-n+\frac{p}{2}+\mathbb{Z}_{\ge 0})\cup(\frac{3}{2}-n+\frac{p}{2}+\mathbb{Z}_{\ge 0})=1-n+\frac{p}{2}+\frac{1}{2}\mathbb{Z}_{\ge 0}$ by Lemma \ref{GKdown}.
	
	So far, we have completed the proof of all cases of the second case of type $B_n$.
\end{proof}
\begin{example}
	Let $\frg=\frso(7, \mathbb{C})$. We consider the parabolic subalgebra $\mathfrak{q}$ corresponding to the subset $I=\Pi\setminus\{\alpha_3\}$. 	In this case, we have $p=3$ and $n=3$. So the scalar generalized Verma module $M_{I}(z\xi_3)$ is reducible if and only if 
	$z\in 1-n+\mathbb{Z}_{\ge 0}=-2+\mathbb{Z}_{\ge 0}$. This result is the same with \cite{He} and \cite{KT}.
	
%	From \cite{KT}, we know that $M_q(\Lambda_t)$ is reducible if and only if $t\in 1+\mathbb{Z}_{\ge 0}$. In this case, the hightest wight of $M_q(\Lambda_t)$ is $\lambda-\rho$, so
%	we should minus $c_k=n$. (see \cite{HK} for more details). Hence, if the hightest wight of $M_q(\Lambda_t)$ is $\lambda$, $M_q(\Lambda_t)$ is reducible if and only if $t\in 1-n+\mathbb{Z}_{\ge 0}=-2+\mathbb{Z}_{\ge 0}$.
\end{example}

The arguments of $\mathfrak{sp}(2n,\mathbb{C})$ are similar to that of $\mathfrak{so}(2n+1,\mathbb{C})$. Next, we will discuss the reducibility of the scalar generalized Verma modules of $ \mathfrak{g}=\mathfrak{sp}(2n,\mathbb{C}) $.
\begin{prop}\label{C}
	Let $ \mathfrak{g}=\mathfrak{sp}(2n,\mathbb{C}) $. The scalar generalized Verma module $M_{I}(\lambda)$ is reducible if and only if
\end{prop}
\begin{enumerate}
	\item if $p=\mathrm{1}$, then
		$ z\in \mathbb{Z}_{\ge 0};$
	
	\item if $\mathrm{2}\le p\le n-1$, then
	\[	
	z\in\left\{
	\begin{array}{ll}
	\frac{p}{2}-n+\frac{1}{2}+\frac{1}{2}\mathbb{Z}_{\ge 0} &\text{if $p$ is odd or $3p>2n$}\\	     	  	   
	\frac{p}{2}-n+\frac{1}{2}\mathbb{Z}_{\ge 0} &\textnormal{if $p$ is even and\:\:$3p\le 2n$};\\	
	\end{array}	
	\right.
	\]	
	
	\item if $p=n$, then
	$z\in \frac{1}{2}-\frac{n}{2}+\frac{1}{2}\mathbb{Z}_{\ge 0}$.	
\end{enumerate}

\begin{proof}
	First, let's consider the  case $p=1$.
	
	Take  $\Delta^+(\frl)=\{\alpha_2,\dots,\alpha_{n}\}$, where	
	$\alpha_i=e_i-e_{i+1}\:\:1\le i\le n-1$ and $\alpha_n=2e_n$. When $M_I(\lambda)$ is of scalar type, $\lambda=z\xi $ for some $z\in\mathbb{R}$ by Lemma \ref{scalar}.
	
	By the definition of fundamental weight, we can compute
	that $\xi=(1,0,\dots,0)$,
	so
	$$\lambda+\rho=(z+n,n-1,n-1,\dots,1),$$	
	$$(\lambda+\rho)^-=(z+n,\dots,1,-1,\dots,1-n,-z-n).$$ 
	Then we have the follows.
	\begin{enumerate}
		\item When $z\in \mathbb{Z}$ and $z>-1$, 
			$(\lambda+\rho)^-$ is an integral weight.
				$$p((\lambda+\rho)^-)=(\underbrace{1,1,\dots,1}_{2n})\:\:{\rm and}\:\: p((\lambda+\rho)^-)_i^{odd}=(\underbrace{0,1,0,1,\dots,0,1}_{2n}).$$
			
			By Lemma \ref{integral}, we obtain that
				\begin{align*}
		{\rm GKdim}\:L(\lambda)&=n^2-\sum\limits_{i\ge 1}(i-1)p((\lambda+\rho)^-)_i^{odd}\\&=n^2-(0\cdot 0+1\cdot 1+2\cdot 0+3\cdot 1+\dots +(2n-2)\cdot 0+(2n-1)\cdot 1)\\&=n^2-n^2=0<\dim(\mathfrak{u}).
		\end{align*}
		By Lemma \ref{reducible}, $M_{I}(\lambda)$ is reducible.
		\begin{enumerate}	
						\item When $z=-1$, we will have
			$$(\lambda+\rho)^-=(n-1,n-1,\dots,1,-1,\dots,1-n,1-n).$$
			\[
	\tiny{\begin{tikzpicture}[scale=\domscale+0.25,baseline=-45pt]
			\hobox{0}{0}{1}
			\hobox{1}{0}{2} 
			\end{tikzpicture}}\to 
			\tiny{\begin{tikzpicture}[scale=\domscale+0.25,baseline=-45pt]
			\hobox{0}{0}{3}
			\hobox{1}{0}{2}
			\hobox{0}{1}{1} 
			\end{tikzpicture}}\dashrightarrow 
			\tiny{\begin{tikzpicture}[scale=\domscale+0.25,baseline=-45pt]
			\hobox{0}{0}{n}
			\hobox{1}{0}{2}
			\hobox{0}{1}{\vdots}
			\hobox{0}{2}{3}
			\hobox{0}{3}{1}
			\end{tikzpicture}}\dashrightarrow 
			\tiny{\begin{tikzpicture}[scale=\domscale+0.25,baseline=-45pt]
			\hobox{0}{0}{2n-1}
			\hobox{1}{0}{2n}
			\hobox{0}{1}{2n-2}
			\hobox{1}{1}{2}
			\hobox{0}{2}{\vdots}
			\hobox{0}{3}{3}
			\hobox{0}{4}{1}     
			\end{tikzpicture}}=P(\lambda).
			\]
			$${\rm So}\:\:p((\lambda+\rho)^-)=(2,2,\underbrace{1,1,\dots,1}_{2n-4})\:\:{\rm and}\:\: p((\lambda+\rho)^-)^{odd}=(1,1,0,1,\dots,0,1).$$
			\begin{align*}
			{\rm GKdim}\:L(\lambda)&=n^2-\sum\limits_{i\ge 1}(i-1)p((\lambda+\rho)^-)_i^{odd}\\&=n^2-(0\cdot 0+1\cdot 1+2\cdot 0+3\cdot 1+\dots +(2n-4)\cdot 0+(2n-3)\cdot 1)\\&=n^2-n^2+2n-1=\dim(\mathfrak{u}).
			\end{align*}
			By Lemma \ref{reducible}, $M_{I}(\lambda)$ is irreducible. 
		\end{enumerate}
	
		\item When $z\in \frac{1}{2}+\mathbb{Z}$, we will have
		$$(\lambda+\rho)^-=(z+n,\dots,1,-1,\dots,1-n,-z-n),$$
		$$(\lambda+\rho)_{(0)}^-=(n-1,\dots,1,-1,\dots,1-n)\:\:{\rm and}\:\:(\lambda+\rho)_{(\frac{1}{2})}^-=(z+n,-z-n).$$
		
		By Theorem \ref{GKdim}, we get
		\begin{align*}
		{\rm GKdim}\:L(\lambda)&=n^2-(1\cdot 1+2\cdot 0+3\cdot 1+\dots +(2n-4)\cdot 0+(2n-3)\cdot 1)\\&=n^2-n^2+2n-1=\dim(\mathfrak{u}).
		\end{align*}
		By Lemma \ref{reducible},  $M_{I}(\lambda)$ is irreducible.
		
		\item When $z\notin \frac{1}{2}\mathbb{Z}$, we have
		$$(\lambda+\rho)_{(0)}^-=(n-1,\dots,1,-1,\dots,1-n),$$
		$$((\lambda+\rho)_{(0)}^-)^{odd}=(\underbrace{1,1,1,\dots,1}_{2n-2})\:\:{\rm and}\:\: (\lambda+\rho)_{(z)}=(z+n)\in[\lambda]_3.$$
		By Theorem \ref{GKdim}, we get
		\begin{align*}
		{\rm GKdim}\:L(\lambda)&=n^2-(1\cdot 1+2\cdot 0+3\cdot 1+\dots +(2n-4)\cdot 0+(2n-3)\cdot 1)\\&=n^2-n^2+2n-1\\&=2n-1=\dim(\mathfrak{u}).
		\end{align*}
			By Lemma \ref{reducible},  $M_{I}(\lambda)$ is irreducible.
	\end{enumerate}

	In summary, when $p=1$, $M_{I}(\lambda)$ is reducible if and only if $z\in\mathbb{Z}_{\ge 0}$.\\
	
	Now we  consider the case  $p=n$.
	
	Take  $\Delta^+(\frl)=\{\alpha_1,\dots,\alpha_{n-1}\}$, 
where	
	$\alpha_i=e_i-e_{i+1}\:\:1\le i\le n-1$ and $\alpha_n=2e_n$. When $M_I(\lambda)$ is of scalar type, $\lambda=z\xi $ for some $z\in\mathbb{R}$ by Lemma \ref{scalar}.
	
	By the definition of fundamental weight, we can compute
	that $\xi=(1,1,\dots,1)$, then
	$$(\lambda+\rho)^-=(z+n,z+n-1,\dots,z+1,-z-1,\dots,-z+1-n,-z-n).$$
	
%	Whether $n$ is odd or even, one of $\frac{1-n}{2}$ and $1-\frac{n}{2}$ must be an integer, and the difference between them is $\frac{1}{2}$. Let's assume that $1-\frac{n}{2}$ is an integer.
First we consider the case when $p=n$ is even.
	\begin{enumerate}
		\item If $z\in\mathbb{Z}$ and $z>-1$,
	 $(\lambda+\rho)^-$ is an integral weight, so
			$$p((\lambda+\rho)^-)=(\underbrace{1,1,\dots,1}_{2n}),$$
	$$
		p((\lambda+\rho)^-)^{odd}=(0,1,0,1,\dots,0,1).
	$$
		By Lemma \ref{integral}, we get
		\begin{align*}
		{\rm GKdim}\:L(\lambda)&=n^2-\sum\limits_{i\ge 1}(i-1)p((\lambda+\rho)^-)_i^{odd}\\&=n^2-(0\cdot 0+1\cdot 1+2\cdot 0+3\cdot 1+\dots +(2n-2)\cdot 0+(2n-1)\cdot 1)\\&=n^2-n^2=0<\dim(\mathfrak{u}).
		\end{align*}
		By Lemma \ref{reducible},  $M_{I}(\lambda)$ is reducible.
		
		\begin{enumerate}
			\item When $z=-1$, we will have
			$$(\lambda+\rho)^-=(n-1,n-2,\dots,0,0,\dots,2-n,1-n).$$
			\[
	\tiny{\begin{tikzpicture}[scale=\domscale+0.25,baseline=-45pt]
			\hobox{0}{0}{n}
			\hobox{0}{1}{n-1}
			\hobox{0}{2}{\vdots}
			\hobox{0}{3}{1} 
			\end{tikzpicture}}\to 
			\tiny{\begin{tikzpicture}[scale=\domscale+0.25,baseline=-45pt]
			\hobox{0}{0}{n}
			\hobox{1}{0}{n+1}
			\hobox{0}{1}{n-1}
			\hobox{0}{2}{\vdots}
			\hobox{0}{3}{1} 
			\end{tikzpicture}}\dashrightarrow 
			\tiny{\begin{tikzpicture}[scale=\domscale+0.25,baseline=-45pt]
			\hobox{0}{0}{2n}
			\hobox{1}{0}{n+1}
			\hobox{0}{1}{2n-1}
			\hobox{0}{2}{\vdots}
			\hobox{0}{3}{1} 
			\end{tikzpicture}}=P(\lambda).
			\]
		So	$p((\lambda+\rho)^-)=(2,\underbrace{1,1,\dots,1)}_{2n-2}\:\:{\rm and}\:\:
			p((\lambda+\rho)^-)^{odd}=(\underbrace{1,1,0,\dots,1,0)}_{2n-1}.$
			\begin{align*}
			{\rm GKdim}\:L(\lambda)&=n^2-\sum\limits_{i\ge 1}(i-1)p((\lambda+\rho)^-)_i^{odd}\\&=n^2-(0\cdot 1+1\cdot 1+2\cdot 0+3\cdot 1+\dots +(2n-3)\cdot 1+(2n-2)\cdot 0)\\&=n^2-(1+3+\dots+2n-3)\\&=n^2-n^2+2n-1=2n-1<\dim(\mathfrak{u}).
			\end{align*}	
			By Lemma \ref{reducible}, $M_{I}(\lambda)$ is reducible.
			
			\item When $z=1-\frac{n}{2}$, we will have
			$$(\lambda+\rho)^-=(\frac{n}{2}+1,\frac{n}{2},\dots,-\frac{n}{2}+2,\frac{n}{2}-2,\dots,-\frac{n}{2},-\frac{n}{2}-1).$$
			\[
	\tiny{\begin{tikzpicture}[scale=\domscale+0.25,baseline=-45pt]
			\hobox{0}{0}{n}
			\hobox{0}{1}{n-1}
			\hobox{0}{2}{\vdots}
			\hobox{0}{3}{1} 
			\end{tikzpicture}}\to 
			\tiny{\begin{tikzpicture}[scale=\domscale+0.25,baseline=-45pt]
			\hobox{0}{0}{n}
			\hobox{1}{0}{n+1}
			\hobox{0}{1}{n-1}
			\hobox{0}{2}{\vdots}
			\hobox{0}{3}{1} 
			\end{tikzpicture}}\to
			\tiny{\begin{tikzpicture}[scale=\domscale+0.25,baseline=-45pt]
			\hobox{0}{0}{n}
			\hobox{1}{0}{n+2}
			\hobox{0}{1}{n-1}
			\hobox{1}{1}{n+1}
			\hobox{0}{2}{\vdots}
			\hobox{0}{3}{1} 
			\end{tikzpicture}}\dashrightarrow 
			\tiny{\begin{tikzpicture}[scale=\domscale+0.25,baseline=-45pt]
			\hobox{0}{0}{n}
			\hobox{1}{0}{2n-3}
			\hobox{0}{1}{n-1}
			\hobox{1}{1}{\vdots}
			\hobox{0}{2}{\vdots}
			\hobox{1}{2}{n+1}
			\hobox{0}{3}{1} 
			\end{tikzpicture}}\dashrightarrow 
			\tiny{\begin{tikzpicture}[scale=\domscale+0.25,baseline=-45pt]
			\hobox{0}{0}{2n}
			\hobox{1}{0}{2n-3}
			\hobox{0}{1}{2n-1}
			\hobox{1}{1}{\vdots}
			\hobox{0}{2}{\vdots}
			\hobox{1}{2}{n+1}
			\hobox{0}{3}{1} 
			\end{tikzpicture}}=P(\lambda).
			\]	
		So	$p((\lambda+\rho)^-)^{odd}=(\underbrace{1,1,\dots,1}_{n-3},1,0,1,0,1,0).$	
			\begin{align}\label{33}
			{\rm GKdim}\:L(\lambda)&=n^2-\sum\limits_{i\ge 1}(i-1)p((\lambda+\rho)^-)_i^{odd}\notag
			\\&=n^2-(1\cdot 1+2\cdot 1+3\cdot 1+\dots +(n-3)\cdot 1+(n-1)\cdot 1+(n+1)\cdot 1)\\
			&=n^2-(1+2+\dots+n-3)-2n\notag\\
			&=n^2-\frac{(n-2)(n-3)}{2}-2n\notag\\
			&=\frac{1}{2}n^2+\frac{1}{2}n-3<\dim(\mathfrak{u})\notag.
			\end{align}
			By Lemma \ref{reducible}, $M_{I}(\lambda)$ is reducible. 
			
			\item When $z=-\frac{n}{2}$, we will have
			$$(\lambda+\rho)^-=(\frac{n}{2},\frac{n}{2}-1,\dots,-\frac{n}{2}+1,\frac{n}{2}-1,\dots,-\frac{n}{2}+1,-\frac{n}{2}).$$
			\[
	\tiny
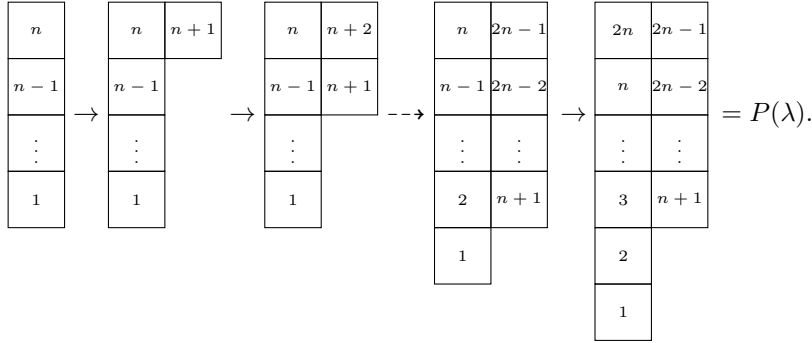
{\begin{tikzpicture}[scale=\domscale+0.25,baseline=-45pt]
			\hobox{0}{0}{n}
			\hobox{0}{1}{n-1}
			\hobox{0}{2}{\vdots}
			\hobox{0}{3}{1} 
			\end{tikzpicture}}\to 
			\tiny{\begin{tikzpicture}[scale=\domscale+0.25,baseline=-45pt]
			\hobox{0}{0}{n}
			\hobox{1}{0}{n+1}
			\hobox{0}{1}{n-1}
			\hobox{0}{2}{\vdots}
			\hobox{0}{3}{1} 
			\end{tikzpicture}}\to 
			\tiny{\begin{tikzpicture}[scale=\domscale+0.25,baseline=-45pt]
			\hobox{0}{0}{n}
			\hobox{1}{0}{n+2}
			\hobox{0}{1}{n-1}
			\hobox{1}{1}{n+1}
			\hobox{0}{2}{\vdots}
			\hobox{0}{3}{1}
			\end{tikzpicture}}\dashrightarrow
			\tiny{\begin{tikzpicture}[scale=\domscale+0.25,baseline=-45pt]
			\hobox{0}{0}{n}
			\hobox{1}{0}{2n-1}
			\hobox{0}{1}{n-1}
			\hobox{1}{1}{2n-2}
			\hobox{0}{2}{\vdots}
			\hobox{1}{2}{\vdots} 
			\hobox{0}{3}{2}
			\hobox{1}{3}{n+1}
			\hobox{0}{4}{1}
			\end{tikzpicture}}\to 
			\tiny{\begin{tikzpicture}[scale=\domscale+0.25,baseline=-45pt]
			\hobox{0}{0}{2n}
			\hobox{1}{0}{2n-1}
			\hobox{0}{1}{n}
			\hobox{1}{1}{2n-2}
			\hobox{0}{2}{\vdots}
			\hobox{1}{2}{\vdots} 
			\hobox{0}{3}{3}
			\hobox{1}{3}{n+1}
			\hobox{0}{4}{2}	
			\hobox{0}{5}{1}
			\end{tikzpicture}}=P(\lambda).
			\]	
		So	$p((\lambda+\rho)^-)^{odd}=(\underbrace{1,1,1,\dots,1}_{n}).$
			\begin{align*}
			{\rm GKdim}\:L(\lambda)&=n^2-\sum\limits_{i\ge 1}(i-1)p((\lambda+\rho)^-)_i^{odd}\\&=n^2-(0\cdot 1+1\cdot 1+2\cdot 1+3\cdot 1+\dots +(n-1)\cdot 1)\\&=n^2-\frac{(n)(n-1)}{2}\\&=\frac{n(n+1)}{2}=\dim(\mathfrak{u}).
			\end{align*}
				By Lemma \ref{reducible},  $M_{I}(\lambda)$ is irreducible.
		\end{enumerate}	
	 So $z=1-\frac{n}{2}$ is the first reducible point.	
		
		\item If $z\notin \frac{1}{2}\mathbb{Z}$, we will have		
		$(\lambda+\rho)_{(0)}^-=(0)$ and $(\lambda+\rho)_{(\frac{1}{2})}^-=(0)$, then
		$$\lambda+\rho=(z+n,z+n-1,\dots,z+1)\in [\lambda]_3.$$	
		
		By Theorem \ref{GKdim}, we get
		\begin{align*}
		{\rm GKdim}\:L(\lambda)&=n^2-\sum\limits _{x\in [\lambda]_3} F_A(\tilde{x})\\&=n^2-(0\cdot 1+1\cdot 1+2\cdot 1+3\cdot 1+\dots +(n-1)\cdot 1)\\&=n^2-\frac{n(n-1)}{2}\\&=\frac{n(n+1)}{2}=\dim(\mathfrak{u}).
		\end{align*}	
		By Lemma \ref{reducible},  $M_{I}(\lambda)$ is irreducible.
		
		\item If  $z\in \frac{1}{2}+\mathbb{Z}$, we have
		$$(\lambda+\rho)^-=(z+n,z+n-1,\dots,z+1,-z-1,\dots,-z+1-n,-z-n),$$	
		$$(\lambda+\rho)_{(0)}^-=(0)\:\:{\rm and}\:\:  (\lambda+\rho)_{(\frac{1}{2})}^-=(\lambda+\rho)^-.$$
		\begin{enumerate}
			
			\item When $z\in \frac{1}{2}+\mathbb{Z}_{\ge -1}$, we will have
			$$p((\lambda+\rho)_{(\frac{1}{2})}^-)^{ev}=(\underbrace{1,0,\dots,1,0}_{2n}).$$
			By Theorem \ref{GKdim}, we get 
			\begin{align*}
			{\rm GKdim}\:L(\lambda)&=n^2-\sum\limits_{i\ge 1}(i-1)p((\lambda+\rho)_{(\frac{1}{2})}^-)^{ev}\\&=n^2-(0\cdot 1+1\cdot 0+2\cdot 1+3\cdot 0+\dots +2(n-1)\cdot 1+(2n-1)\cdot 0)\\&=n^2-2\cdot\frac{n(n-1)}{2}=n<\dim(\mathfrak{u}).
			\end{align*}	
			By Lemma \ref{reducible},  $M_{I}(\lambda)$ is reducible.
			
			\item When $z=-\frac{3}{2}$, we will have
        
			\[
	\tiny
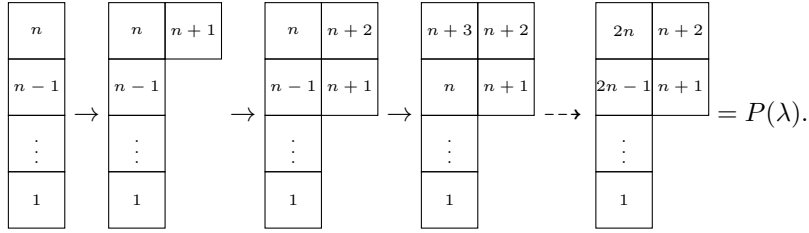
{\begin{tikzpicture}[scale=\domscale+0.25,baseline=-45pt]
			\hobox{0}{0}{n}
			\hobox{0}{1}{n-1}
			\hobox{0}{2}{\vdots}
			\hobox{0}{3}{1} 
			\end{tikzpicture}}\to 
			\tiny{\begin{tikzpicture}[scale=\domscale+0.25,baseline=-45pt]
			\hobox{0}{0}{n}
			\hobox{1}{0}{n+1}
			\hobox{0}{1}{n-1}
			\hobox{0}{2}{\vdots}
			\hobox{0}{3}{1} 
			\end{tikzpicture}}\to 
			\tiny{\begin{tikzpicture}[scale=\domscale+0.25,baseline=-45pt]
			\hobox{0}{0}{n}
			\hobox{1}{0}{n+2}
			\hobox{0}{1}{n-1}
			\hobox{1}{1}{n+1}
			\hobox{0}{2}{\vdots}
			\hobox{0}{3}{1}
			\end{tikzpicture}}\to 
			\tiny{\begin{tikzpicture}[scale=\domscale+0.25,baseline=-45pt]
			\hobox{0}{0}{n+3}
			\hobox{1}{0}{n+2}
			\hobox{0}{1}{n}
			\hobox{1}{1}{n+1}
			\hobox{0}{2}{\vdots}
			\hobox{0}{3}{1}    
			\end{tikzpicture}}\dashrightarrow 
			\tiny{\begin{tikzpicture}[scale=\domscale+0.25,baseline=-45pt]
			\hobox{0}{0}{2n}
			\hobox{1}{0}{n+2}
			\hobox{0}{1}{2n-1}
			\hobox{1}{1}{n+1}
			\hobox{0}{2}{\vdots}
			\hobox{0}{3}{1} 
			\end{tikzpicture}}=P(\lambda).
			\]
		So	$p((\lambda+\rho)^-)^{ev}=(\underbrace{1,1,1,0,1,0\dots,1,0}_{2n-2}).$
			\begin{align*}
			{\rm GKdim}\:L(\lambda)&=n^2-\sum\limits_{i\ge 1}(i-1)p((\lambda+\rho)^-)_i^{ev}\\&=n^2-(0\cdot 1+1\cdot 1+2\cdot 1+3\cdot 0+\dots +(2n-4)\cdot 1+(2n-3)\cdot 0)\\&=n^2-(1+2+4+\dots+2n-4)\\&=n^2-1-2\cdot\frac{(n-1)(n-2)}{2}=3n-3<\dim(\mathfrak{u}).
			\end{align*}
			By Lemma \ref{reducible},  $M_{I}(\lambda)$ is reducible.
			
			\item When $z=\frac{1}{2}-\frac{n}{2}$, we will have
			$$(\lambda+\rho)^-=(\frac{n}{2}+\frac{1}{2},\frac{n}{2}-\frac{1}{2},\dots,\frac{3}{2}-\frac{n}{2},\frac{n}{2}-\frac{3}{2},\dots,\frac{1}{2}-\frac{n}{2},-\frac{1}{2}-\frac{n}{2}).$$
			\[
		\tiny{\tiny
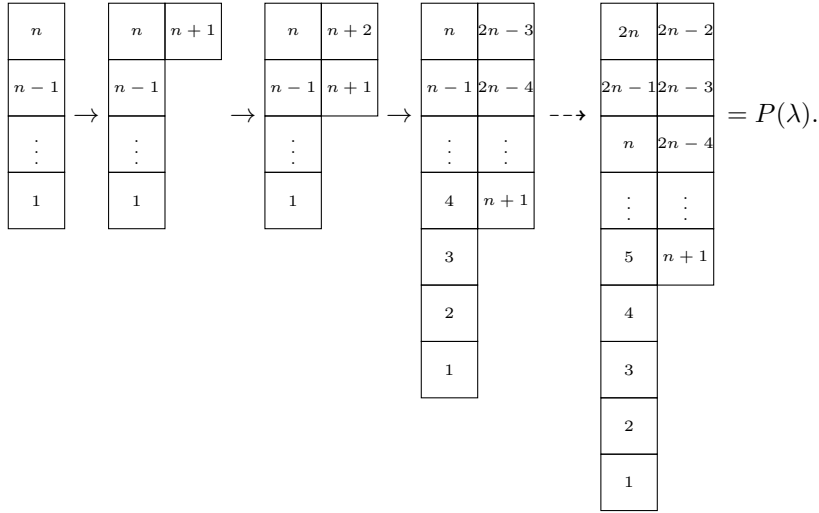
{\begin{tikzpicture}[scale=\domscale+0.25,baseline=-45pt]
			\hobox{0}{0}{n}
			\hobox{0}{1}{n-1}
			\hobox{0}{2}{\vdots}
			\hobox{0}{3}{1} 
			\end{tikzpicture}}}\to 
			\tiny{\tiny{\begin{tikzpicture}[scale=\domscale+0.25,baseline=-45pt]
			\hobox{0}{0}{n}
			\hobox{1}{0}{n+1}
			\hobox{0}{1}{n-1}
			\hobox{0}{2}{\vdots}
			\hobox{0}{3}{1} 
			\end{tikzpicture}}}\to 
			\tiny{\tiny{\begin{tikzpicture}[scale=\domscale+0.25,baseline=-45pt]
			\hobox{0}{0}{n}
			\hobox{1}{0}{n+2}
			\hobox{0}{1}{n-1}
			\hobox{1}{1}{n+1}
			\hobox{0}{2}{\vdots}
			\hobox{0}{3}{1}
			\end{tikzpicture}}}\to 
		\tiny{	\tiny{\begin{tikzpicture}[scale=\domscale+0.25,baseline=-45pt]
			\hobox{0}{0}{n}
			\hobox{1}{0}{2n-3}
			\hobox{0}{1}{n-1}
			\hobox{1}{1}{2n-4}
			\hobox{0}{2}{\vdots}
			\hobox{1}{2}{\vdots}
			\hobox{0}{3}{4}
			\hobox{1}{3}{n+1}
			\hobox{0}{4}{3}
			\hobox{0}{5}{2}
			\hobox{0}{6}{1}    
			\end{tikzpicture}}}\dashrightarrow 
			\tiny{\tiny{\begin{tikzpicture}[scale=\domscale+0.25,baseline=-45pt]
			\hobox{0}{0}{2n}
			\hobox{1}{0}{2n-2}
			\hobox{0}{1}{2n-1}
			\hobox{1}{1}{2n-3}
			\hobox{0}{2}{n}
			\hobox{1}{2}{2n-4}
			\hobox{0}{3}{\vdots}
			\hobox{1}{3}{\vdots} 
			\hobox{0}{4}{5}
			\hobox{1}{4}{n+1}
			\hobox{0}{5}{4}
			\hobox{0}{6}{3}
			\hobox{0}{7}{2}
			\hobox{0}{8}{1}
			\end{tikzpicture}}}=P(\lambda).
			\]
		So	$p((\lambda+\rho)^-)^{ev}=(\underbrace{1,1,1,\dots,1}_{n-2},1,0,1,0).$
			\begin{align*}
			{\rm GKdim}\:L(\lambda)&=n^2-\sum\limits_{i\ge 1}(i-1)p((\lambda+\rho)^-)_i^{ev}\\&=n^2-(1\cdot 1+2\cdot 1+\dots +(n-3)\cdot 1+(n-2)\cdot 1+(n-1)\cdot 0+n\cdot 1)\\&=n^2-(1+2+\dots+n-2)-n\\&=n^2-\frac{(n-1)(n-2)}{2}-n\\&=\frac{1}{2}n^2+\frac{1}{2}n-1<\dim(\mathfrak{u}).
			\end{align*}
			By Lemma \ref{reducible}, $M_{I}(\lambda)$ is reducible.
			
			\item When $z=-\frac{1}{2}-\frac{n}{2}$, we will have
			$$(\lambda+\rho)^-=(\frac{n}{2}-\frac{1}{2},\frac{n}{2}-\frac{3}{2},\dots,\frac{1}{2}-\frac{n}{2},\frac{n}{2}-\frac{1}{2},\dots,\frac{3}{2}-\frac{n}{2},\frac{1}{2}-\frac{n}{2}).$$
			\[
	\tiny{\begin{tikzpicture}[scale=\domscale+0.25,baseline=-45pt]
			\hobox{0}{0}{n}
			\hobox{0}{1}{n-1}
			\hobox{0}{2}{\vdots}
			\hobox{0}{3}{1} 
			\end{tikzpicture}}\to 
			\tiny{\begin{tikzpicture}[scale=\domscale+0.25,baseline=-45pt]
			\hobox{0}{0}{n}
			\hobox{1}{0}{n+1}
			\hobox{0}{1}{n-1}
			\hobox{0}{2}{\vdots}
			\hobox{0}{3}{1} 
			\end{tikzpicture}}\to 
			\tiny{\begin{tikzpicture}[scale=\domscale+0.25,baseline=-45pt]
			\hobox{0}{0}{n}
			\hobox{1}{0}{n+2}
			\hobox{0}{1}{n-1}
			\hobox{1}{1}{n+1}
			\hobox{0}{2}{\vdots}
			\hobox{0}{3}{1}
			\end{tikzpicture}}\dashrightarrow 
			\tiny{\begin{tikzpicture}[scale=\domscale+0.25,baseline=-45pt]
			\hobox{0}{0}{n}
			\hobox{1}{0}{2n}
			\hobox{0}{1}{n-1}
			\hobox{1}{1}{2n-1}
			\hobox{0}{2}{\vdots}
			\hobox{1}{2}{\vdots}
			\hobox{0}{3}{1} 
			\hobox{1}{3}{n+1}
			\end{tikzpicture}}=P(\lambda).
			\]
		So	$p((\lambda+\rho)^-)^{ev}=(\underbrace{1,1,1,\dots,1}_{n}).$
			\begin{align*}
			{\rm GKdim}\:L(\lambda)&=n^2-\sum\limits_{i\ge 1}(i-1)p((\lambda+\rho)^-)_i^{ev}\\&=n^2-(0\cdot 1+1\cdot 1+2\cdot 1+\dots +(n-3)\cdot 1+(n-2)\cdot 1+(n-1)\cdot 1)\\&=n^2-(1+2+\dots+n-1)\\&=n^2-\frac{n(n-1)}{2}=\frac{n(n+1)}{2}=\dim(\mathfrak{u}).
			\end{align*}
			By Lemma \ref{reducible}, $M_{I}(\lambda)$ is irreducible. And $z=\frac{1}{2}-\frac{n}{2}$ is the first reducible point when $z\in \frac{1}{2}+\mathbb{Z}$.
		\end{enumerate}

		When $p=n$ is odd, the argument is similar. 
		
		In summary, for the case $p=n$, $M_{I}(\lambda)$ is reducible if and only if $z\in\{\frac{1-n}{2}+\mathbb{Z}_{\ge 0}\}\cup\{1-\frac{n}{2}+\mathbb{Z}_{\ge 0}\}=\frac{1}{2}-\frac{n}{2}+\frac{1}{2}\mathbb{Z}_{\ge 0}$.
	\end{enumerate}
	
	Finally, let's consider the  case $2\leq p\leq n-1$. 	Under the circumstances, the arguments of type $B_n$ and type $C_n$ are almost the same, so  we will omit the proof of this part.
\end{proof}

\section{Reducibility of scalar  genralized Verma modules for type $D_n$}
In this section, we consider the type $D_n$. Let $M_{I}(\lambda)$ be a scalar generalized Verma module with highest weight $\lambda=z\xi$, where $\xi=\xi_p$ is the fundamental weight corresponding to the simple root $\alpha_{p}$. We have the following proposition.

\begin{prop}\label{D}
	Let $ \mathfrak{g}=\mathfrak{so}(2n,\mathbb{C}) $. $M_{I}(\lambda)$ is reducible if and only if	  
	
	\begin{enumerate}
		\item if $p=\mathrm{1}$, then
		$\quad z\in 2-n+\mathbb{Z}_{\ge 0}$;
		
		\item if $\mathrm{2}\le p\le n-2$, then
		\[	z\in\left\{
		\begin{array}{ll}
		\frac{3}{2}+\frac{p}{2}-n+\frac{1}{2}\mathbb{Z}_{\ge 0} &\textnormal{if $p$ is odd or}\:\:3p<2n\\	     	  	   
		1+\frac{p}{2}-n+\frac{1}{2}\mathbb{Z}_{\ge 0} &\textnormal{if $p$ is even and}\:\:3p\ge 2n;	
		\end{array}	
		\right.
		\]	
		
		\item if $p=n-1$ or $n$, then
		\[	z\in\left\{
		\begin{array}{ll}
		2-n+\mathbb{Z}_{\ge 0} &\textnormal{if $n$ is even}\\	     	  	   
		3-n+\mathbb{Z}_{\ge 0} &\textnormal{if $n$ is odd.}	
		\end{array}	
		\right.
		\]
	\end{enumerate}		     
	
\end{prop}	

\begin{proof}
For  $p=1$, we take  $\Delta^+(\frl)=\{\alpha_2,\dots,\alpha_{n}\}$, where	
	$\alpha_i=e_i-e_{i+1}\:\:1\le i\le n-1$ and $\alpha_n=e_{n-1}+e_{n+1}$. When $M_I(\lambda)$ is of scalar type, we know that $\lambda=z\xi $ for some $z\in\mathbb{R}$ by Lemma \ref{scalar}.
	
	By the definition of fundamental weight, we can compute
	that $\xi=(1,0,\dots,0)$, (see details in\cite{KN}),
so
	$$\lambda+\rho=(z+n-1,n-2,\dots,0),$$
	$$(\lambda+\rho)^-=(z+n-1,n-2,\dots,0,0,\dots,2-n,1-n-z).$$
	
	Then we have the follows.
	\begin{enumerate}
		\item When $z\in\mathbb{Z}$ and $z>-1$, $\lambda+\rho$ is  integral.
		$$p((\lambda+\rho)^-)=(1,\underbrace{0,1,0,1,\dots,0,1}_{2n-2}).$$		
		By Lemma \ref{integral}, we get 
		\begin{align*}
		 {\rm GKdim}\:L(\lambda)&= n^2-n-\sum\limits_{i\ge 1}(i-1)p((\lambda+\rho)^-)_i^{ev}\\&=n^2-n-(0\cdot 1+1\cdot 0+2\cdot 1+\dots+(2n-2)\cdot 1)\\&=n^2-n-2(1+2+\dots+n-1)\\&=n^2-n-2\cdot\frac{n(n-1)}{2}=0<\dim(\mathfrak{u}).		
		\end{align*}
		By Lemma \ref{reducible}, $M_{I}(\lambda)$ is reducible.
		\begin{enumerate}
			
			\item When $z=-1$, we will have
			$$(\lambda+\rho)^-=(n-2,n-2,\dots,0,0,\dots,2-n,2-n).$$		
			\[
	\tiny{\begin{tikzpicture}[scale=\domscale+0.25,baseline=-45pt]
			\hobox{0}{0}{n}
			\hobox{1}{0}{2}
			\hobox{0}{1}{n-1}
			\hobox{0}{2}{\vdots}
			\hobox{0}{3}{1} 
			\end{tikzpicture}}\to 
			\tiny{\begin{tikzpicture}[scale=\domscale+0.25,baseline=-45pt]
			\hobox{0}{0}{n}
			\hobox{1}{0}{n+1}
			\hobox{0}{1}{n-1}
			\hobox{1}{1}{2}
			\hobox{0}{2}{\vdots}
			\hobox{0}{3}{1} 
			\end{tikzpicture}}\dashrightarrow 
			\tiny{\begin{tikzpicture}[scale=\domscale+0.25,baseline=-45pt]
			\hobox{0}{0}{2n-1}
			\hobox{1}{0}{n+1}
			\hobox{0}{1}{2n-2}
			\hobox{1}{1}{2}
			\hobox{0}{2}{\vdots}
			\hobox{0}{3}{1}
			\end{tikzpicture}}\to 
			\tiny{\begin{tikzpicture}[scale=\domscale+0.25,baseline=-45pt]
			\hobox{0}{0}{2n-1}
			\hobox{1}{0}{2n}
			\hobox{0}{1}{2n-2}
			\hobox{1}{1}{n+1}
			\hobox{0}{2}{2n-3}
			\hobox{1}{2}{2}
			\hobox{0}{3}{\vdots}
			\hobox{0}{4}{1}    
			\end{tikzpicture}} =P(\lambda).
			\]
	So		$p((\lambda+\rho)^-)^{ev}=(1,1,1,\underbrace{0,1,0,1,\dots,0,1}_{2n-6}).$
			\begin{align*}
			{\rm GKdim}\:L(\lambda)&=n^2-n-\sum\limits_{i\ge 1}(i-1)p((\lambda+\rho)^-)_i^{ev}\\&=n^2-n-(0\cdot 1+1\cdot 1+2\cdot 1+\dots +(2n-5)\cdot 0+(2n-4)\cdot 1\\&=n^2-n-1-2(1+2+\dots+n-2)\\&=n^2-n-1-2\cdot\frac{(n-1)(n-2)}{2}=2n-3<\dim(\mathfrak{u}).
			\end{align*}
			By Lemma \ref{reducible}, $M_{I}(\lambda)$ is reducible.
			
			\item When $z=2-n$, we will have
			$$(\lambda+\rho)^-=(1,n-2,\dots,0,0,\dots,2-n,-1).$$	
			\[
	\tiny{\begin{tikzpicture}[scale=\domscale+0.25,baseline=-45pt]
			\hobox{0}{0}{1}
			\hobox{1}{0}{4}
			\hobox{0}{1}{3}
			\hobox{0}{2}{2} 
			\end{tikzpicture}}\dashrightarrow 
			\tiny{\begin{tikzpicture}[scale=\domscale+0.25,baseline=-45pt]
			\hobox{0}{0}{1}
			\hobox{1}{0}{n-2}
			\hobox{0}{1}{n-1}
			\hobox{0}{2}{\vdots}
			\hobox{0}{3}{2} 
			\end{tikzpicture}}\to
			\tiny{\begin{tikzpicture}[scale=\domscale+0.25,baseline=-45pt]
			\hobox{0}{0}{1}
			\hobox{1}{0}{n-1}
			\hobox{0}{1}{n-2}
			\hobox{0}{2}{\vdots}
			\hobox{0}{3}{2}
			\end{tikzpicture}}\dashrightarrow 
			\tiny{\begin{tikzpicture}[scale=\domscale+0.25,baseline=-45pt]
			\hobox{0}{0}{2n-1}
			\hobox{1}{0}{n+1}
			\hobox{0}{1}{2n-2}
			\hobox{1}{1}{n-1}
			\hobox{0}{2}{\vdots}
			\hobox{0}{3}{2}    
			\end{tikzpicture}}\to
			\tiny{\begin{tikzpicture}[scale=\domscale+0.25,baseline=-45pt]
			\hobox{0}{0}{2n-1}
			\hobox{1}{0}{2n}
			\hobox{0}{1}{2n-2}
			\hobox{1}{1}{n+1}
			\hobox{0}{2}{2n-3}
			\hobox{1}{2}{n-1}
			\hobox{0}{3}{\vdots}
			\hobox{0}{4}{2}  
			\end{tikzpicture}}=P(\lambda).	 
			\]
		So	$p((\lambda+\rho)^-)^{ev}=(1,1,1,\underbrace{0,1,0,1,\dots,0,1}_{2n-6}).$
			\begin{align*}
			{\rm GKdim}\:L(\lambda)&=n^2-n-\sum\limits_{i\ge 1}(i-1)p((\lambda+\rho)^-)_i^{ev}\\&=n^2-n-(0\cdot 1+1\cdot 1+2\cdot 1+\dots +(2n-5)\cdot 0+(2n-4)\cdot 1\\&=n^2-n-1-2(1+2+\dots+n-2)\\&=n^2-n-1-2\cdot\frac{(n-1)(n-2)}{2}=2n-3<\dim(\mathfrak{u}).
			\end{align*}
			By Lemma \ref{reducible}, $M_{I}(\lambda)$ is reducible.
			
			\item When $z=1-n$, we will have
$$(\lambda+\rho)^-=(0,n-2,\dots,0,0,\dots,2-n,0).$$
			\[
	\tiny{\begin{tikzpicture}[scale=\domscale+0.25,baseline=-45pt]
			\hobox{0}{0}{1}
			\hobox{1}{0}{3}
			\hobox{0}{1}{2}
			\end{tikzpicture}}\dashrightarrow 
			\tiny{\begin{tikzpicture}[scale=\domscale+0.25,baseline=-45pt]
			\hobox{0}{0}{1}
			\hobox{1}{0}{n}
			\hobox{0}{1}{n-1}
			\hobox{0}{2}{\vdots}
			\hobox{0}{3}{2} 
			\end{tikzpicture}}\to
			\tiny{\begin{tikzpicture}[scale=\domscale+0.25,baseline=-45pt]
			\hobox{0}{0}{1}
			\hobox{1}{0}{n}
			\hobox{2}{0}{n+1}
			\hobox{0}{1}{n-1}
			\hobox{0}{2}{\vdots}
			\hobox{0}{3}{2}
			\end{tikzpicture}}\dashrightarrow
			\tiny{\begin{tikzpicture}[scale=\domscale+0.25,baseline=-45pt]
			\hobox{0}{0}{2n-1}
			\hobox{1}{0}{n}
			\hobox{2}{0}{n+1}
			\hobox{3}{0}{2n}
			\hobox{0}{1}{2n-2}
			\hobox{0}{2}{\vdots}
			\hobox{0}{3}{2}  
			\end{tikzpicture}}=P(\lambda).	 
			\]
		So	$p((\lambda+\rho)^-)^{ev}=(2,\underbrace{0,1,0,1,\dots,0,1}_{2n-4}).$
			\begin{align*}
			{\rm GKdim}\:L(\lambda)&=n^2-n-\sum\limits_{i\ge 1}(i-1)p((\lambda+\rho)^-)_i^{ev}\\&=n^2-n-(0\cdot 2+1\cdot 0+2\cdot 1+\dots +(2n-5)\cdot 0+(2n-4)\cdot 1)\\&=n^2-n-2(1+2+\dots+n-2)\\&=n^2-n-2\cdot\frac{(n-1)(n-2)}{2}=2n-2=\dim(\mathfrak{u}).
			\end{align*}
			By Lemma \ref{reducible}, $M_{I}(\lambda)$ is irreducible.
		\end{enumerate}
		 So $z=2-n$ is the first reducible point when $z\in \mathbb{Z}$.
		
		\item When $z\in\frac{1}{2}+\mathbb{Z}$, we have
		$$p((\lambda+\rho)^-_{(0)})^{ev}=(1,\underbrace{0,1,\dots,0,1}_{2n-4})\:\:\text{and}\:\:	p((\lambda+\rho)^-_{(\frac{1}{2})})^{ev}=(1).$$
		By Theorem \ref{GKdim}, we get
		\begin{align*}
		{\rm GKdim}\:L(\lambda)&=n^2-n-(0\cdot 1+1\cdot 0+2\cdot 1+\dots +(2n-5)\cdot 0+(2n-4)\cdot 1)-(0\cdot 1)\\&=n^2-n-2(1+2+\dots+n-2)\\&=n^2-n-2\cdot\frac{(n-1)(n-2)}{2}=2n-2=\dim(\mathfrak{u}).
		\end{align*}
		By Lemma \ref{reducible}, $M_{I}(\lambda)$ is irreducible.
		
		\item When $z\notin\frac{1}{2}\mathbb{Z}$, we have
		$$p((\lambda+\rho)^-_{(0)})^{ev}=(1,\underbrace{0,1,\dots,0,1}_{2n-4})\:\:\text{and}\:\:	(\lambda+\rho)_{(z)}=(z+n-1)\in[\lambda]_3.$$
		
	By Theorem \ref{GKdim}, we get	
		\begin{align*}
		{\rm GKdim}\:L(\lambda)&=n^2-n-(0\cdot 1+1\cdot 0+2\cdot 1+\dots +(2n-5)\cdot 0+(2n-4)\cdot 1)-(0\cdot 1)\\&=n^2-n-2(1+2+\dots+n-2)\\&=n^2-n-2\cdot\frac{(n-1)(n-2)}{2}=2n-2=\dim(\mathfrak{u}).
		\end{align*}
	By Lemma \ref{reducible}, $M_{I}(\lambda)$ is irreducible.
\end{enumerate}
	
	 Combining (2) and (3), we can get a conclusion that if $z\notin\mathbb{Z}$, $M_{I}(\lambda)$ is irreducible.
	
	In summary, when $p=1$, $M_{I}(\lambda)$ is reducible if and only if $z\in 2-n+\mathbb{Z}_{\ge 0}$.\\
	
Now we consider the case $p=n-1$.
	
	Take  $\Delta^+(\frl)=\{\alpha_1,\dots,\alpha_{n-2},\alpha_{n}\}$, 
where $\alpha_i=e_i-e_{i+1}\:\:1\le i\le n-1$ and $\alpha_n=e_{n-1}+e_{n+1}$. When $M_I(\lambda)$ is of scalar type, $\lambda=z\xi $ for some $z\in\mathbb{R}$ by Lemma \ref{scalar}.
	
	By the definition of fundamental weight, we can compute
	that $\xi=(\frac{1}{2},\frac{1}{2},\dots,\frac{1}{2})$, (see details in\cite{KN}), then
	$$(\lambda+\rho)^-=(\frac{1}{2}z+n-1,\dots,\frac{1}{2}z,-\frac{1}{2}z,\dots,1-n-\frac{1}{2}z).$$
	First we consider the case when $n$ is even.
	\begin{enumerate}
		\item If $z\in \mathbb{Z}$ and $z>0$,		
		$\lambda+\rho$ is integral and $p((\lambda+\rho)^-)^{ev}=(1,0,\dots,1,0)$.
		
		By Lemma \ref{integral}, we get
		\begin{align*}
		{\rm GKdim}\:L(\lambda)&= n^2-n-\sum\limits_{i\ge 1}(i-1)p((\lambda+\rho)^-)_i^{ev}\\&=n^2-n-(0\cdot 1+1\cdot 0+2\cdot 1+\dots+(2n-2)\cdot 1)\\&=n^2-n-2(1+2+\dots+n-1)\\&=n^2-n-2\cdot\frac{n(n-1)}{2}=0<\dim(\mathfrak{u}).
		\end{align*}
		So $M_{I}(\lambda)$ is reducible.
				
		\begin{enumerate}
			\item When $z=0$, we will have
			\begin{align}\label{5}
			(\lambda+\rho)^-=(n-1,n-2,\dots,0,0,\dots,2-n,1-n),
			\end{align}
			$$p((\lambda+\rho)^-)^{ev}=(1,\underbrace{0,1,\dots,0,1}_{2n-2}).$$
		By Theorem \ref{GKdim}, we get
			\begin{align*}
			{\rm GKdim}\:L(\lambda)&= n^2-n-\sum\limits_{i\ge 1}(i-1)p((\lambda+\rho)^-)_i^{ev}\\&=n^2-n-(0\cdot 1+1\cdot 0+2\cdot 1+\dots+(2n-2)\cdot 1)\\&=n^2-n-2(1+2+\dots+n-1)\\&=n^2-n-2\cdot\frac{n(n-1)}{2}=0<\dim(\mathfrak{u}).	
			\end{align*}
			So $M_{I}(\lambda)$ is reducible.	
			
			\item When $z=2-n$, we will have
			$$(\lambda+\rho)^-=(\frac{1}{2}n,\frac{1}{2}n-1,\dots,1-\frac{1}{2}n,\frac{1}{2}n-1,\dots,1-\frac{1}{2}n,-\frac{1}{2}n).$$	
			\[
	\tiny{\begin{tikzpicture}[scale=\domscale+0.25,baseline=-45pt]
			\hobox{0}{0}{n}
			\hobox{0}{1}{n-1}
			\hobox{0}{2}{\vdots}
			\hobox{0}{3}{1} 
			\end{tikzpicture}}\to 
			\tiny{\begin{tikzpicture}[scale=\domscale+0.25,baseline=-45pt]
			\hobox{0}{0}{n}
			\hobox{1}{0}{n+1}
			\hobox{0}{1}{n-1}
			\hobox{0}{2}{\vdots}
			\hobox{0}{3}{1} 
			\end{tikzpicture}}\to 
			\tiny{\begin{tikzpicture}[scale=\domscale+0.25,baseline=-45pt]
			\hobox{0}{0}{n}
			\hobox{1}{0}{n+2}
			\hobox{0}{1}{n-1}
			\hobox{1}{1}{n+1}
			\hobox{0}{2}{\vdots}
			\hobox{0}{3}{1}
			\end{tikzpicture}}\dashrightarrow 
			\tiny{\begin{tikzpicture}[scale=\domscale+0.25,baseline=-45pt]
			\hobox{0}{0}{n}
			\hobox{1}{0}{2n-1}
			\hobox{0}{1}{n-1}
			\hobox{1}{1}{2n-2}
			\hobox{0}{2}{\vdots}
			\hobox{1}{2}{\vdots} 
			\hobox{0}{3}{2}
			\hobox{1}{3}{n+1}
			\hobox{0}{4}{1}
			\end{tikzpicture}}\to 
			\tiny{\begin{tikzpicture}[scale=\domscale+0.25,baseline=-45pt]
			\hobox{0}{0}{2n}
			\hobox{1}{0}{2n-1}
			\hobox{0}{1}{n}
			\hobox{1}{1}{2n-2}
			\hobox{0}{2}{\vdots}
			\hobox{1}{2}{\vdots} 
			\hobox{0}{3}{3}
			\hobox{1}{3}{n+1}
			\hobox{0}{4}{2}	
			\hobox{0}{5}{1}
			\end{tikzpicture}}=P(\lambda).
			\]
		So	$p(\lambda+\rho)^-=(\underbrace{2,2,\dots,2}_{n-1},1,1)\:\:{\rm and}\: \:p((\lambda+\rho)^-)^{ev}=(\underbrace{1,1,\dots,1}_{n-1},0,1).$
			\begin{align*}
			{\rm GKdim}\:L(\lambda)&= n^2-n-\sum\limits_{i\ge 1}(i-1)p((\lambda+\rho)^-)_i^{ev}\\&=n^2-n-(0\cdot 1+1\cdot 1+2\cdot 0+\dots+(n-2)\cdot 1+(n-1)\cdot 0+n\cdot 1)\\&=n^2-n-(1+2+\dots+n-2)-n\\&=n^2-n-\frac{(n-1)(n-2)}{2}-n\\&=\frac{1}{2}n^2-\frac{1}{2}n-1<\frac{1}{2}n^2-\frac{1}{2}n=\dim(\mathfrak{u}).	
			\end{align*}
		So $M_{I}(\lambda)$ is reducible.
			
			\item When $z=1-n$, we will have
			$$(\lambda+\rho)^-=(\frac{1}{2}n-\frac{1}{2},\dots,\frac{1}{2}-\frac{1}{2}n,\frac{1}{2}-\frac{1}{2}n,\dots,\frac{1}{2}-\frac{1}{2}n).$$
			\[
	\tiny{\begin{tikzpicture}[scale=\domscale+0.25,baseline=-45pt]
			\hobox{0}{0}{n}
			\hobox{0}{1}{n-1}
			\hobox{0}{2}{\vdots}
			\hobox{0}{3}{1} 
			\end{tikzpicture}}\to 
			\tiny{\begin{tikzpicture}[scale=\domscale+0.25,baseline=-45pt]
			\hobox{0}{0}{n}
			\hobox{1}{0}{n+1}
			\hobox{0}{1}{n-1}
			\hobox{0}{2}{\vdots}
			\hobox{0}{3}{1} 
			\end{tikzpicture}}\to 
			\tiny{\begin{tikzpicture}[scale=\domscale+0.25,baseline=-45pt]
			\hobox{0}{0}{n}
			\hobox{1}{0}{n+2}
			\hobox{0}{1}{n-1}
			\hobox{1}{1}{n+1}
			\hobox{0}{2}{\vdots}
			\hobox{0}{3}{1}
			\end{tikzpicture}}\dashrightarrow 
			\tiny{\begin{tikzpicture}[scale=\domscale+0.25,baseline=-45pt]
			\hobox{0}{0}{n}
			\hobox{1}{0}{2n}
			\hobox{0}{1}{n-1}
			\hobox{1}{1}{2n-1}
			\hobox{0}{2}{\vdots}
			\hobox{1}{2}{\vdots} 
			\hobox{0}{3}{1}
			\hobox{1}{3}{n+1}
			\end{tikzpicture}}=P(\lambda).
			\]
		So	$p((\lambda+\rho)^-)^{ev}=(1,1,\dots,1).$
			\begin{align*}
			{\rm GKdim}\:L(\lambda)&= n^2-n-\sum\limits_{i\ge 1}(i-1)p((\lambda+\rho)^-)_i^{ev}\\&=n^2-n-(0\cdot 1+1\cdot 1+2\cdot 1+\dots+(n-1)\cdot 1)\\&=n^2-n-(1+2+\dots+n-1)\\&=n^2-n-\frac{n(n-1)}{2}\\&=\frac{1}{2}n^2-\frac{1}{2}n=\dim(\mathfrak{u}).	
			\end{align*}
			So 	 $M_{I}(\lambda)$ is irreducible. 
		\end{enumerate}
	Thus $z=2-n$ is the first reducible point when $z\in\mathbb{Z}$.
		
		\item If $z\notin\mathbb{Z}$, we have
		$$(\lambda+\rho)^-=(\frac{1}{2}z+n-1,\dots,\frac{1}{2}z,-\frac{1}{2}z,\dots,1-n-\frac{1}{2}z),$$
		$$p((\lambda+\rho)^-_{(\frac{1}{2})})^{ev}=	p((\lambda+\rho)^-_{(0)})^{ev}=(0)\:\:\text{and}\:\:(\lambda+\rho)=(\frac{1}{2}z+n-1,\dots,\frac{1}{2}z)\in[\lambda]_3.$$
		By Theorem \ref{GKdim}, we get
		\begin{align*}
		{\rm GKdim}\:L(\lambda)&= n^2-n-(0\cdot 1+1\cdot 1+2\cdot 1+\dots+(n-1)\cdot 1)\\&=n^2-n-(1+2+\dots+n-1)\\&=n^2-n-\frac{n(n-1)}{2}\\&=\frac{1}{2}n^2-\frac{1}{2}n=\dim(\mathfrak{u}).	
		\end{align*}
		So $M_{I}(\lambda)$ is irreducible by Lemma \ref{reducible}.
	\end{enumerate}	
	
To sum up, when $n$ is even, $M_{I}(\lambda)$ is reducible if and only if $z\in2-n+\mathbb{Z}_{\ge 0}$.\\
	
	When $n$ is odd,  the arguments are similar to the case when $n$ is even. Finally, we can get the result that $M_{I}(\lambda)$ is reducible if and only if $z\in3-n+\mathbb{Z}_{\ge 0}$. We omit the details here.\\

	All in all, when $p=n-1$, $M_{I}(\lambda)$ is reducible if and only if	
	\[	z\in\left\{
	\begin{array}{ll}
	2-n+\mathbb{Z}_{\ge 0} &\text{if $n$ is  even}\\	     	  	   
	3-n+\mathbb{Z}_{\ge 0} &\text{if $n$ is odd}.	
	\end{array}	
	\right.
	\]
	
	The arguments of the case when $p=n$ are similar to that of $p=n-1$, we omit the details here.\\

	Finally we consider the case when $2\le p\le n-2$.
	
	First we consider the case when p is  even and $3p\ge 2n$. 
	\begin{enumerate}
		\item If $z\in\mathbb{Z}$ and $z > -1$, we have
			$$(\lambda+\rho)^-=(z+n-1,\dots,z+n-p,n-p-1,\dots,0,0,\dots,-n+p+1,\dots,-z-n+1),$$	
		$$p((\lambda+\rho)^-)^{ev}=(1,\underbrace{0,1,\dots,0,1}_{2n-2},1).$$	
		By 	Theorem \ref{integral}, we get 
		\begin{align}\label{7}
		{\rm GKdim}\:L(\lambda)\notag&= n^2-n-(0\cdot 1+1\cdot 0+2\cdot 1+\dots+(2n-2)\cdot 1)\notag\\&=n^2-n-2(1+2+\dots+n-1)\notag\\&=n^2-n-2\cdot\frac{n(n-1)}{2}=0<\dim(\mathfrak{u}).
		\end{align}
		By Lemma \ref{reducible}, $M_{I}(\lambda)$ is reducible.	
		\begin{enumerate}	
			\item When $z=-1$, we will have	
			$$(\lambda+\rho)^-=(n-2,\dots,n-p-1,n-p-1,\dots,0,0,\dots,-n+p+1,\dots,-n+2).$$	
			\[
		\tiny{\begin{tikzpicture}[scale=\domscale+0.25,baseline=-45pt]
			\hobox{0}{0}{p}
			\hobox{1}{0}{p+1}
			\hobox{0}{1}{p-1}
			\hobox{0}{2}{\vdots}
			\hobox{0}{3}{1} 
			\end{tikzpicture}}\dashrightarrow  
			\tiny{\begin{tikzpicture}[scale=\domscale+0.25,baseline=-45pt]
			\hobox{0}{0}{n}
			\hobox{1}{0}{p+1}
			\hobox{0}{1}{n-1}
			\hobox{0}{2}{\vdots}
			\hobox{0}{3}{1} 
			\end{tikzpicture}}\to 
			\tiny{\begin{tikzpicture}[scale=\domscale+0.25,baseline=-45pt]
			\hobox{0}{0}{n}
			\hobox{1}{0}{n+1}
			\hobox{0}{1}{n-1}
			\hobox{1}{1}{p+1}
			\hobox{0}{2}{n-2}
			\hobox{0}{3}{\vdots}
			\hobox{0}{4}{1} 
			\end{tikzpicture}}\dashrightarrow 
			\tiny{\begin{tikzpicture}[scale=\domscale+0.25,baseline=-45pt]
			\hobox{0}{0}{k}
			\hobox{1}{0}{k+1}
			\hobox{0}{1}{k-1}
			\hobox{1}{1}{n+1}
			\hobox{0}{2}{\vdots}
			\hobox{1}{2}{p+1} 
			\hobox{0}{3}{1}
			\end{tikzpicture}}\dashrightarrow 
			\tiny{\begin{tikzpicture}[scale=\domscale+0.25,baseline=-45pt]
			\hobox{0}{0}{2n}
			\hobox{1}{0}{k+1}
			\hobox{0}{1}{2n-1}
			\hobox{1}{1}{n+1}
			\hobox{0}{2}{\vdots}
			\hobox{1}{2}{p+1} 
			\hobox{0}{3}{1}
			\end{tikzpicture}}=P(\lambda).
			\]	
			In the above Young tableau, $k=2n-p$.	
		
		So	$p((\lambda+\rho)^-)^{ev}=(1,1,1,\underbrace{0,1,\dots,0,1}_{2n-6},1).$	
			
			By 	Theorem \ref{integral}, we get
			\begin{align*}
			{\rm GKdim}\:L(\lambda)&= n^2-n-(0\cdot 1+1\cdot 1+2\cdot 1+\dots+(2n-4)\cdot 1)\\&=n^2-n-1-2(1+2+\dots+n-2)\\&=n^2-n-1-2\cdot\frac{(n-1)(n-2)}{2}=2n-3<\dim(\mathfrak{u}).	
			\end{align*}
			By Lemma \ref{reducible}, $M_{I}(\lambda)$ is reducible.	
			
			\item When $z=\frac{p}{2}-n+1$, we will have
			$$(\lambda+\rho)^-=(\frac{p}{2},\dots,-\frac{p}{2}+1,n-p-1,\dots,0,0,\dots,-n+p+1,\frac{p}{2}-1,\dots,-\frac{p}{2}),$$
			$$p((\lambda+\rho)^-)^{ev}=(2,\underbrace{1,2,\dots,1,2}_{2n-2p-2},\underbrace{1,\dots,1}_{3p-2n},0,1).$$	
			By 	Theorem \ref{integral}, we get 
			\begin{align}\label{8}
			{\rm GKdim}\:L(\lambda)\notag&= n^2-n-(1\cdot 1+2\cdot 2+\dots+(2n-2p-2)\cdot 2)-(2n-2p-1+\dots+p-2)-p\notag\\&=n^2-n-(n-p-1)^2-4\cdot\frac{(n-p)(n-p-1)}{2}-\frac{(2n-p-3)(3p-2n)}{2}-p\notag\\&=2np-\frac{3}{2}p^2-\frac{1}{2}p-1=\dim(\mathfrak{u})-1<\dim(\mathfrak{u}).	
			\end{align}
			By Lemma \ref{reducible}, $M_{I}(\lambda)$ is reducible.
			
			\item When $z=\frac{p}{2}-n$, we will have
			$$(\lambda+\rho)^-=(\frac{p}{2}-1,\dots,-\frac{p}{2},n-p-1,\dots,0,0,\dots,-n+p+1,\frac{p}{2},\dots,-\frac{p}{2}+1),$$
			$$p((\lambda+\rho)^-)^{ev}=(2,\underbrace{1,2,\dots,1,2}_{2n-2p-2},\underbrace{1,\dots,1}_{3p-2n+1}).$$	
			By 	Theorem \ref{integral}  and (\ref{8}), we get
			\begin{align*}
			{\rm GKdim}\:L(\lambda) &=n^2-n-(n-p-1)^2-4\cdot C^2_{n-p}-(2n-2p-1+\dots+p-2)-p+1\\&=2np-\frac{3}{2}p^2-\frac{1}{2}p=\dim(\mathfrak{u}).	
			\end{align*}
			By Lemma \ref{reducible}, $M_{I}(\lambda)$ is irreducible.
		\end{enumerate}
		
		So $z=\frac{p}{2}-n+1$ is the first reducible point when $z\in \mathbb{Z}$.	
		
		\item If $z\notin\frac{1}{2}\mathbb{Z}$, we have	
		$$(\lambda+\rho)^-_{(0)}=(n-p-1,\dots,0,0,\dots,-n+p+1)\text{~and~} (\lambda+\rho)_{(z)}=(z+n-1,\dots,z+n-p)\in[\lambda]_3.$$	
		\begin{align}\label{9}
		{\rm GKdim}\:L(\lambda) &=n^2-n-(1+\dots+p-1)-(2+\dots+2n-2p-2)\notag\\&=n^2-n-\frac{p(p-1)}2-2\cdot\frac{(n-p)(n-p-1)}{2}\notag\\&=2pn-\frac{3}{2}p^2-\frac{1}{2}p=\dim(\mathfrak{u}).	
		\end{align}
			So $M_{I}(\lambda)$ is irreducible.

		\item If  $z\in\frac{1}{2}+\mathbb{Z}$ and $z\ge p-n$, we have	
		\begin{align}\label{10}
		p((\lambda+\rho)^-_{(0)})^{ev}=(1,\underbrace{0,1,\dots,0,1}_{2n-2p-2})\text{~and~}	p((\lambda+\rho)^-_{(\frac{1}{2})})^{ev}=(\underbrace{1,0,\dots,1,0}_{2p}).
		\end{align}	
		\begin{align*}
		{\rm GKdim}\:L(\lambda) &=n^2-n-(1+\dots+2p-2)-(2+\dots+2n-2p-2)\\&=n^2-n-2\cdot\frac{p(p-1)}2-2\cdot\frac{(n-p)(n-p-1)}{2}\\&=2pn-\frac{3}{2}p^2-\frac{1}{2}p-C^2_p<\dim(\mathfrak{u}).	
		\end{align*}
		By Lemma \ref{reducible}, $M_{I}(\lambda)$ is reducible.
		 
		\begin{enumerate}
			\item When $z=\frac{p}{2}-n+1+\frac{1}{2}$, we will have
			$$(\lambda+\rho)^-=(\frac{p}{2}+\frac{1}{2},\dots,-\frac{p}{2}+\frac{3}{2},n-p-1,\dots,0,0,\dots,-n+p+1,\frac{p}{2}-\frac{3}{2},\dots,-\frac{p}{2}-\frac{1}{2}).$$
			\[
		\tiny{\begin{tikzpicture}[scale=\domscale+0.25,baseline=-45pt]
			\hobox{0}{0}{p}
			\hobox{1}{0}{p+1}
			\hobox{0}{1}{p-1}
			\hobox{0}{2}{\vdots}
			\hobox{0}{3}{1} 
			\end{tikzpicture}}\dashrightarrow  
			\tiny{\begin{tikzpicture}[scale=\domscale+0.25,baseline=-45pt]
			\hobox{0}{0}{p}
			\hobox{1}{0}{2p-2}
			\hobox{0}{1}{p-1}
			\hobox{1}{1}{\vdots}
			\hobox{0}{2}{\vdots}
			\hobox{1}{2}{p+1}
			\hobox{0}{3}{1} 
			\end{tikzpicture}}\dashrightarrow  
			\tiny{\begin{tikzpicture}[scale=\domscale+0.25,baseline=-45pt]
			\hobox{0}{0}{2p}
			\hobox{1}{0}{2p-2}
			\hobox{0}{1}{2p-1}
			\hobox{1}{1}{\vdots}
			\hobox{0}{2}{\vdots}
			\hobox{1}{2}{p+1}
			\hobox{0}{3}{1} 
			\end{tikzpicture}}=P(\lambda).
			\]	
		So	$p((\lambda+\rho)^-_{(\frac{1}{2})})^{ev}=(\underbrace{1,1,\dots,1}_{p-2},1,0,1).$
			\begin{align*}
			{\rm GKdim}\:L(\lambda) &=n^2-n-(n-p)^2+n-p-(1+2+\dots+p-2+p)\\&=n^2-n-(n-p)^2+n-p-\frac{(p-1)(p-2)}{2}-p\\&=2pn-\frac{3}{2}p^2-\frac{1}{2}p-1<\dim(\mathfrak{u}).	
			\end{align*}
			By Lemma \ref{reducible}, $M_{I}(\lambda)$ is reducible. 
			
			\item When $z=\frac{p}{2}-n+\frac{1}{2}$, we will have	
			$$(\lambda+\rho)^-=(\frac{p}{2}-\frac{1}{2},\dots,-\frac{p}{2}+\frac{1}{2},n-p-1,\dots,0,0,\dots,-n+p+1,\frac{p}{2}-\frac{1}{2},\dots,-\frac{p}{2}+\frac{1}{2}).$$
			\[
		\tiny{\begin{tikzpicture}[scale=\domscale+0.25,baseline=-45pt]
			\hobox{0}{0}{p}
			\hobox{0}{1}{p-1}
			\hobox{0}{2}{\vdots}
			\hobox{0}{3}{1} 
			\end{tikzpicture}}\to 
			\tiny{\begin{tikzpicture}[scale=\domscale+0.25,baseline=-45pt]
			\hobox{0}{0}{p}
			\hobox{1}{0}{p+1}
			\hobox{0}{1}{p-1}
			\hobox{0}{2}{\vdots}
			\hobox{0}{3}{1} 
			\end{tikzpicture}}\dashrightarrow  
			\tiny{\begin{tikzpicture}[scale=\domscale+0.25,baseline=-45pt]
			\hobox{0}{0}{p}
			\hobox{1}{0}{2p}
			\hobox{0}{1}{p-1}
			\hobox{1}{1}{2p-1}
			\hobox{0}{2}{\vdots}
			\hobox{1}{2}{\vdots}
			\hobox{0}{3}{1} 
			\hobox{1}{3}{p+1}
			\end{tikzpicture}}=P(\lambda).
			\]	
		So	\begin{align}\label{11}
			p((\lambda+\rho)^-_{(\frac{1}{2})})^{ev}=(\underbrace{1,1,1,\dots,1}_{p}).
			\end{align}
			\begin{align*}
			{\rm GKdim}\:L(\lambda) &=n^2-n-(n-p)^2+n-p-(1+2+\dots+p-2+p-1)\\&=n^2-n-(n-p)^2+n-p-\frac{(p-1)(p-2)}{2}-p+1\\&=2pn-\frac{3}{2}p^2-\frac{1}{2}p=\dim(\mathfrak{u}).
			\end{align*}
				By Lemma \ref{reducible}, $M_{I}(\lambda)$ is irreducible.
		\end{enumerate}
	 Hence, $z=\frac{p}{2}-n+1+\frac{1}{2}$ is the first reducible point when  $z\in\frac{1}{2}+\mathbb{Z}$.
	\end{enumerate}
	
 All in all, when $p$ is even and $3p\geq 2n$, $M_{I}(\lambda)$ is reducible if and only if $z\in(\frac{p}{2}-n+1+\mathbb{Z}_{\ge 0})\cup(\frac{p}{2}-n+1+\frac{1}{2}+\mathbb{Z}_{\ge 0})=\frac{p}{2}-n+1+\frac{1}{2}\mathbb{Z}_{\ge 0}$.	\\
	
	Next, let's look at the case when $p$ is odd.	
	\begin{enumerate}
		\item When $z\in\mathbb{Z}$ and $z>-1$,
		 $M_{I}(\lambda)$ is reducible by  (\ref{7}).
		
		\begin{enumerate}
			\item  When $z=-1$,  $M_{I}(\lambda)$ is reducible by (\ref{8}).
			
			\item When $z=\frac{p}{2}-n+\frac{3}{2}$, we will have
			$$(\lambda+\rho)^-=(\frac{p}{2}+\frac{1}{2},\dots,-\frac{p}{2}+\frac{1}{2},n-p-1,\dots,0,0,\dots,-n+p+1,\frac{p}{2}-\frac{1}{2},\dots,-\frac{p}{2}-\frac{1}{2}).$$		
			We divide the discussion into three cases:
				\begin{enumerate}
				\item  $p$ is odd and $3p>2n$. We have
				$$ p((\lambda+\rho)^-)^{ev}=(2,\underbrace{1,2,\dots,1,2}_{2n-2p-2},\underbrace{1,\dots,1}_{3p-2n-1},1,0,1).$$
				
				By (\ref{8}), $M_{I}(\lambda)$ is reducible.
		\item $p$ is odd and $2n-5<3p<2n$. We have 
		$$ p((\lambda+\rho)^-)^{ev}=(2,\underbrace{1,2,\dots,1,2}_{p-3},\underbrace{1,\dots,1}_{2n-3p+1},\underbrace{0,1,\dots,0,1}_{3p-2n+3}).$$
		\begin{align*}
		{\rm GKdim}\:L(\lambda) &=n^2-n-(1+4+\dots+p-4+(p-3)\cdot 2)-(p-2+\dots+2n-2p-2)\\
		&\quad-(2n-2p+2n-2p+2\dots+p+1)\\&=n^2-n-\frac{(p-3)\frac{p-3}{2}}{2}-4\cdot \frac{\frac{p-1}{2}\cdot\frac{p-3}{2}}{2}-\frac{(2n-p-4)(2n-3p+1)}{2}\\
		&\quad-(n-\frac{p}{2}+\frac{1}{2})(\frac{3p}{2}-n+\frac{3}{2})\\&=2pn-\frac{3}{2}p^2-2p+n-\frac{5}{2}<\dim(\mathfrak{u}).
		\end{align*}	
		By Lemma \ref{reducible}, $M_{I}(\lambda)$ is reducible.
	\item $p$ is odd and $3p\le 2n-5$. We have 
	$$ p((\lambda+\rho)^-)^{ev}=(2,\underbrace{1,2,\dots,1,2}_{p-3},1,1,1,1,\underbrace{0,1,\dots,0,1}_{2n-3p-3}).$$	
	\begin{align*}
	{\rm GKdim}\:L(\lambda) &=n^2-n-(1+4+\dots+p-4+(p-3)\cdot 2)-(p-2+p-1+p)\\&\quad-(p+1+p+3+\dots+2n-2p-2)\\&=n^2-n-\frac{(p-3)\frac{p-3}{2}}{2}-4\cdot \frac{\frac{p-1}{2}\cdot\frac{p-3}{2}}{2}-\frac{(2n-p-1)(n-\frac{3}{2}p-\frac{1}{2})}{2}\\&=2pn-\frac{3}{2}p^2-\frac{1}{2}p-1<\dim(\mathfrak{u}).	
	\end{align*}	
	
	By Lemma \ref{reducible}, $M_{I}(\lambda)$ is reducible. 
				\end{enumerate}
				
		 Combining with (i),(ii) and (iii), $z=\frac{p}{2}-n+\frac{3}{2}$	is a reducible point.
			
			\item When $z=\frac{p}{2}-n+\frac{1}{2}$, we will have	
			$$(\lambda+\rho)^-=(\frac{p}{2}-\frac{1}{2},\dots,-\frac{p}{2}+\frac{1}{2},n-p-1,\dots,0,0,\dots,-n+p+1,\frac{p}{2}-\frac{1}{2},\dots,-\frac{p}{2}+\frac{1}{2}).$$
			We divide the discussion into two cases:
				\begin{enumerate}		
			\item $p$ is odd and $3p+1\ge 2n$. We have
			$$ p((\lambda+\rho)^-)^{ev}=(2,\underbrace{1,2,\dots,1,2}_{2n-2p-2},\underbrace{1,\dots,1}_{3p-2n+1}).$$	
			\begin{align*}
			{\rm GKdim}\:L(\lambda) &=n^2-n-(1+4+\dots+(2n-2p-2)\cdot 2)-(2n-2p-1+\dots+p-2)-p+1\\&=2pn-\frac{3}{2}p^2-\frac{1}{2}p=\dim(\mathfrak{u}).	
			\end{align*}	
			By Lemma \ref{reducible}, $M_{I}(\lambda)$ is irreducible.		
				
			\item $p$ is odd and $3p+1< 2n$. We have $$ p((\lambda+\rho)^-)^{ev}=(2,\underbrace{1,2,\dots,1,2}_{p-1},\underbrace{0,1,\dots,0,1}_{2n-3p-1}).$$	
			\begin{align*}
			{\rm GKdim}\:L(\lambda) &=n^2-n-(1+4+\dots+p-2+(p-1)\cdot 2)-(p+1+p+3+\dots+2n-2p-2)\\&=n^2-n-\frac{(p-1)\frac{p-1}{2}}{2}-4\cdot \frac{\frac{p+1}{2}\cdot\frac{p-1}{2}}{2}-\frac{(2n-p-1)(n-\frac{3}{2}p-\frac{1}{2})}{2}\\&=2pn-\frac{3}{2}p^2-\frac{1}{2}p=\dim(\mathfrak{u}).	
			\end{align*}
			By Lemma \ref{reducible}, $M_{I}(\lambda)$ is irreducible.
		\end{enumerate}
				 From  (i) and (ii), we can see that $z=\frac{p}{2}-n+\frac{1}{2}$ is an irreducible point.
				\end{enumerate}
			So $z=\frac{p}{2}-n+\frac{3}{2}$ is the first reducible point when $z\in \mathbb{Z}$.
		
		\item If $z\notin \frac{1}{2}\mathbb{Z}$, 
		$M_{I}(\lambda)$ is irreducible by (\ref{9}).
		
		\item If $z\in\frac{1}{2}+\mathbb{Z}$ and $z\ge p-n$,
		$M_{I}(\lambda)$ is reducible by (\ref{10}).
		\begin{enumerate}
			\item When $z=\frac{p}{2}-n+2$, we will have
			$$(\lambda+\rho)^-_{(\frac{1}{2})}=(\frac{p}{2}+1,\dots,-\frac{p}{2}+2,\frac{p}{2}-2,\dots,-\frac{p}{2}-1).$$
			\[
		\tiny{\begin{tikzpicture}[scale=\domscale+0.25,baseline=-45pt]
			\hobox{0}{0}{p}
			\hobox{0}{1}{p-1}
			\hobox{0}{2}{\vdots}
			\hobox{0}{3}{1} 
			\end{tikzpicture}}\to 
			\tiny{\begin{tikzpicture}[scale=\domscale+0.25,baseline=-45pt]
			\hobox{0}{0}{p}
			\hobox{1}{0}{p+1}
			\hobox{0}{1}{p-1}
			\hobox{0}{2}{\vdots}
			\hobox{0}{3}{1} 
			\end{tikzpicture}}\dashrightarrow  
			\tiny{\begin{tikzpicture}[scale=\domscale+0.25,baseline=-45pt]
			\hobox{0}{0}{p}
			\hobox{1}{0}{2p-3}
			\hobox{0}{1}{p-1}
			\hobox{1}{1}{\vdots}
			\hobox{0}{2}{\vdots}
			\hobox{1}{2}{p+1}
			\hobox{0}{3}{1} 
			\end{tikzpicture}}\to  
			\tiny{\begin{tikzpicture}[scale=\domscale+0.25,baseline=-45pt]
			\hobox{0}{0}{2p}
			\hobox{1}{0}{2p-3}
			\hobox{0}{1}{2p-1}
			\hobox{1}{1}{\vdots}
			\hobox{0}{2}{\vdots}
			\hobox{1}{2}{p+1}
			\hobox{0}{3}{1} 
			\end{tikzpicture}}=P(\lambda).
			\]	
		So	$p((\lambda+\rho)^-_{(\frac{1}{2})})^{ev}=(\underbrace{1,1,1,\dots,1}_{p-3},1,0,1,0,1,0).$
			\begin{align*}
			{\rm GKdim}\:L(\lambda) &=n^2-n-(n-p)^2+n-p-(1+2+\dots+p-2+p)-2\\&=n^2-n-(n-p)^2+n-p-\frac{(p-1)(p-2)}{2}-p-2\\&=2pn-\frac{3}{2}p^2-\frac{1}{2}p-3<\dim(\mathfrak{u}).	
			\end{align*}
			By Lemma \ref{reducible}, $M_{I}(\lambda)$ is reducible.
			
			\item When $z=\frac{p}{2}-n+1$, we will have
			$$(\lambda+\rho)^-_{(\frac{1}{2})}=(\frac{p}{2},\dots,-\frac{p}{2}+1,\frac{p}{2}-1,\dots,-\frac{p}{2}).$$
			\[
		\tiny{\begin{tikzpicture}[scale=\domscale+0.25,baseline=-45pt]
			\hobox{0}{0}{p}
			\hobox{0}{1}{p-1}
			\hobox{0}{2}{\vdots}
			\hobox{0}{3}{1} 
			\end{tikzpicture}}\to 
			\tiny{\begin{tikzpicture}[scale=\domscale+0.25,baseline=-45pt]
			\hobox{0}{0}{p}
			\hobox{1}{0}{p+1}
			\hobox{0}{1}{p-1}
			\hobox{0}{2}{\vdots}
			\hobox{0}{3}{1} 
			\end{tikzpicture}}\dashrightarrow  
			\tiny{\begin{tikzpicture}[scale=\domscale+0.25,baseline=-45pt]
			\hobox{0}{0}{p}
			\hobox{1}{0}{2p-1}
			\hobox{0}{1}{p-1}
			\hobox{1}{1}{\vdots}
			\hobox{0}{2}{\vdots}
			\hobox{1}{2}{p+1}
			\hobox{0}{3}{1} 
			\end{tikzpicture}}\to  
			\tiny{\begin{tikzpicture}[scale=\domscale+0.25,baseline=-45pt]
			\hobox{0}{0}{2p}
			\hobox{1}{0}{2p-1}
			\hobox{0}{1}{p}
			\hobox{1}{1}{\vdots}
			\hobox{0}{2}{\vdots}
			\hobox{1}{2}{p+1}
			\hobox{0}{3}{1} 
			\end{tikzpicture}}=P(\lambda).
			\]	
		So $p((\lambda+\rho)^-_{(\frac{1}{2})})^{ev}=(\underbrace{1,1,1,\dots,1}_{p-1},1).$
			
			By (\ref{11}), $M_{I}(\lambda)$ is irreducible.
		\end{enumerate}
So $z=\frac{p}{2}-n+2$ is the first reducible point when $z\in\frac{1}{2}+\mathbb{Z}$.

	\end{enumerate}

	To sum up, when $p$ is odd, $M_{I}(\lambda)$ is reducible if and only if $z\in (\frac{p}{2}-n+\frac{3}{2}+\mathbb{Z}_{\ge 0})\cup(\frac{p}{2}-n+2+\mathbb{Z}_{\ge 0})=\frac{p}{2}-n+\frac{3}{2}+\frac{1}{2}\mathbb{Z}_{\ge 0}$.\\
	
	Ultimately, we discuss the case when $p$ is even and $3p < 2n$.

	\begin{enumerate}
		\item If $z\in \mathbb{Z}$ and $z>-1$,	
		 $M_{I}(\lambda)$ is reducible by (\ref{7}).
		\begin{enumerate}	
			\item When $z=-1$, we will have
			$$
			p((\lambda+\rho)^-)^{ev}=(1,1,1,\underbrace{0,1,\dots,0,1}_{2p-6}).$$
			\begin{align*}
			{\rm GKdim}\:L(\lambda) &=n^2-n-(1\cdot 1+2\cdot 1+\dots+(2n-4)\cdot 1)\\&=2n-3<\dim(\mathfrak{u}).	
			\end{align*}
			By Lemma \ref{reducible}, $M_{I}(\lambda)$ is reducible.
			
			\item When $z=\frac{p}{2}-n+2$, we will have
			$$(\lambda+\rho)^-=(\frac{p}{2}+1,\dots,-\frac{p}{2}+2,n-p-1,\dots,0,0,\dots,-n+p+1,\frac{p}{2}-2,\dots,-\frac{p}{2}-1).$$	
			
			Since $p$ is even, we have the follows.
		\begin{enumerate}	
			\item $p=2$ and $3p\ge 2n-4$. The  the corresponding Young tableau has only two columns. So it is easy to calculate that ${\rm GKdim}\:L(\lambda)<\dim(\mathfrak{u})$.	
			\item $p=2$ and $3p< 2n-4$. Similar to the above case.
			\item $p>2$ and $3p\ge 2n-4$. We have 
				\begin{align}
			\label{12}&p((\lambda+\rho)^-)^{ev}=(2,\underbrace{1,2,\dots,1,2}_{p-4},1,1,1,1,0,1)\\{\rm or~~} &p((\lambda+\rho)^-)^{ev}=\label{13}(2,\underbrace{1,2,\dots,1,2}_{p-4},1,1,1,1,1,1).
			\end{align}	
			For (\ref{12}), we get
			\begin{align*}
			{\rm GKdim}\:L(\lambda) &=n^2-n-(1\cdot 1+2\cdot 2+\dots+(p-4)\cdot 2)-(p-3+\dots+p+p+2)\\&=n^2-n-\frac{(p-4)\cdot\frac{p-4}{2}}{2}-4\cdot \frac{\frac{p-2}{2}\cdot\frac{p-4}{2}}{2}-5p+4\\&=n^2-n-\frac{3}{4}p^2-4<\dim(\mathfrak{u})\quad({\rm since}\:\:3p=2n-2 \text{~or~}  2n-4).
			\end{align*}
			By Lemma \ref{reducible}, $M_{I}(\lambda)$ is reducible.
			
				For (\ref{13}), we get	
			$${\rm GKdim}\:L(\lambda) =n^2-n-\frac{3}{4}p^2-p-3<\dim(\mathfrak{u})\quad({\rm since}\:\:3p=2n-2 \text{~or~}  2n-4).$$
			By Lemma \ref{reducible}, $M_{I}(\lambda)$ is reducible.
			\item	$p>2$ and $3p< 2n-4$. We have
			$$((\lambda+\rho)^-)^{ev}=(2,\underbrace{1,2,\dots,1,2}_{p-4},1,1,1,1,1,1,\underbrace{0,1,\dots,0,1}_{2n-3p-4}).$$	
			\begin{align*}
			{\rm GKdim}\:L(\lambda) &=n^2-n-(1\cdot 1+2\cdot 2+\dots+(p-4)\cdot 2)-(p-3+\dots+p+p+1)\\&\quad-(p+2+p+4+\dots+2n-2p-2)\\&=n^2-n-\frac{(p-4)\cdot\frac{p-4}{2}}{2}-4\cdot \frac{\frac{p-2}{2}\cdot\frac{p-4}{2}}{2}-5p+5-\frac{(2n-p)(n-\frac{3}{2}p-1)}{2}\\&=2pn-\frac{3}{2}p^2-\frac{1}{2}p-3<\dim(\mathfrak{u}).
			\end{align*}
			By Lemma \ref{reducible}, $M_{I}(\lambda)$ is reducible.
			\end{enumerate}	
		
	Hence when $z=\frac{p}{2}-n+2$, $M_I(\lambda)$ will be reducible.	
			\item When $z=\frac{p}{2}-n+1$, we will have	
			$$(\lambda+\rho)^-=(\frac{p}{2},\dots,-\frac{p}{2}+1,n-p-1,\dots,0,0,\dots,-n+p+1,\frac{p}{2}-1,\dots,-\frac{p}{2}),$$
			$$p((\lambda+\rho)^-)^{ev}=(2,\underbrace{1,2,\dots,1,2}_{p-2},1,1,\underbrace{0,1,\dots,0,1}_{2n-3p-2}).$$
			\begin{align*}
			{\rm GKdim}\:L(\lambda) &=n^2-n-(1\cdot 1+2\cdot 2+\dots+(p-4)\cdot 2)-(p-1+p)\\&\quad-(p+2+p+4+\dots+2n-2p-2)\\&=n^2-n-\frac{(p-2)\cdot\frac{p-2}{2}}{2}-4\cdot \frac{\frac{p}{2}\cdot\frac{p-2}{2}}{2}-2p+1-\frac{(2n-p)(n-\frac{3}{2}p-1)}{2}\\&=2pn-\frac{3}{2}p^2-\frac{1}{2}p=\dim(\mathfrak{u}).
			\end{align*}
			By Lemma \ref{reducible}, $M_{I}(\lambda)$ is irreducible.
		\end{enumerate}
		
		 So far we have proved that $z=\frac{p}{2}-n+2$ is the first reducible point when $z\in\mathbb{Z}$ and $3p<2n$.
		
		\item If $z\notin\frac{1}{2}\mathbb{Z}$,
 $M_{I}(\lambda)$ is irreducible by (\ref{9}).
		
		\item If  $z\in\frac{1}{2}+\mathbb{Z}$ and $z\ge p-n$,
		 $M_{I}(\lambda)$ is reducible by (\ref{10}).
		\begin{enumerate}
			\item When $z=\frac{p}{2}-n+\frac{3}{2}$, we will have
			$$p(\lambda+\rho)^-_{(\frac{1}{2})}=(\frac{p}{2}+\frac{1}{2},\dots,-\frac{p}{2}+\frac{1}{2},\frac{p}{2}-\frac{1}{2},\dots,-\frac{p}{2}-\frac{1}{2}).$$
			\[
		\tiny{\begin{tikzpicture}[scale=\domscale+0.25,baseline=-45pt]
			\hobox{0}{0}{p}
			\hobox{0}{1}{p-1}
			\hobox{0}{2}{\vdots}
			\hobox{0}{3}{1} 
			\end{tikzpicture}}\to 
			\tiny{\begin{tikzpicture}[scale=\domscale+0.25,baseline=-45pt]
			\hobox{0}{0}{p}
			\hobox{1}{0}{p+1}
			\hobox{0}{1}{p-1}
			\hobox{0}{2}{\vdots}
			\hobox{0}{3}{1} 
			\end{tikzpicture}}\dashrightarrow  
			\tiny{\begin{tikzpicture}[scale=\domscale+0.25,baseline=-45pt]
			\hobox{0}{0}{p}
			\hobox{1}{0}{2p-1}
			\hobox{0}{1}{p-1}
			\hobox{1}{1}{\vdots}
			\hobox{0}{2}{\vdots}
			\hobox{1}{2}{p+1}
			\hobox{0}{3}{1} 
			\end{tikzpicture}}\to  
			\tiny{\begin{tikzpicture}[scale=\domscale+0.25,baseline=-45pt]
			\hobox{0}{0}{2p}
			\hobox{1}{0}{2p-1}
			\hobox{0}{1}{p}
			\hobox{1}{1}{\vdots}
			\hobox{0}{2}{\vdots}
			\hobox{1}{2}{p+1}
			\hobox{0}{3}{1} 
			\end{tikzpicture}}=P(\lambda).
			\]	
		So	$p((\lambda+\rho)^-_{(\frac{1}{2})})^{ev}=(\underbrace{1,1,1,\dots,1}_{p-1},0,1).$
			\begin{align*}
			{\rm GKdim}\:L(\lambda) &=n^2-n-(n-p)^2+n-p-(1+2+\dots+p-2+p)\\&=n^2-n-(n-p)^2+n-p-\frac{(p-1)(p-2)}{2}-p\\&=2pn-\frac{3}{2}p^2-\frac{1}{2}p-1<\dim(\mathfrak{u}).	
			\end{align*}
			By Lemma \ref{reducible}, $M_{I}(\lambda)$ is reducible.
			
			\item When $z=\frac{p}{2}-n+\frac{1}{2}$, we will have
			$$p(\lambda+\rho)^-_{(\frac{1}{2})}=(\frac{p}{2}-\frac{1}{2},\dots,-\frac{p}{2}-\frac{1}{2},\frac{p}{2}+\frac{1}{2},\dots,-\frac{p}{2}+\frac{1}{2}),$$
			$$p((\lambda+\rho)^-_{(\frac{1}{2})})^{ev}=(\underbrace{1,1,1,\dots,1}_{p}).$$
				By (\ref{11}), $M_{I}(\lambda)$ is irreducible.
		\end{enumerate}
	Hence, $z=\frac{p}{2}-n+\frac{3}{2}$ is the first reducible point when  $z\in\frac{1}{2}+\mathbb{Z}$ and $3p<2n$.
	\end{enumerate}

	All in all, when $3p<2n$, $M_{I}(\lambda)$ is reducible if and only if $z\in (\frac{p}{2}-n+\frac{3}{2}+\mathbb{Z}_{\ge 0})\cup(\frac{p}{2}-n+2+\mathbb{Z}_{\ge 0})=\frac{p}{2}-n+\frac{3}{2}+\frac{1}{2}\mathbb{Z}_{\ge 0}$.
	
	Finally we  completed  the proof of our proposition.

\end{proof}

\section{Gelfand-Kirillov dimensions of scalar  generalized verma modules }

Let $\mathfrak{g}$ be a finite-dimensional complex semisimple Lie
algebra. For a maximal parabolic subalgebra $\frq=\frl\oplus\fru$, if $\fru$ is abelian, the corresponding scalar generalized Verma module $M_I(\lambda)$ will be called \emph{Hermitian type}. These modules are very important in the calssification of unitary highest weight moudles, see for example
\cite{EJ,EHW}.

We use 	$L(z\xi_p)$ to denote the irreducible quotient of the scalar generalized Verma module  $M_I(z\xi_p)$.
In this section, we will give the specific Gelfand-Kirillov dimensions of highest weight moudles $L(z\xi_p)$.

\begin{prop}
	Let $ \mathfrak{g}=\mathfrak{sl}(n,\mathbb{C}) $. Denote $r=\min\{p,n-p\}$. Then 	\[		{\rm GKdim}\:L(z\xi_p)=\left\{
	\begin{array}{ll}
	p(n-p) &\textnormal{if}\:z<1 -r \emph{~or~} z\notin \mathbb{Z}\\	     	  	   
	k(n-k) &\textnormal{if}\:z=-k\in \mathbb{Z}, 1\le k\le r-1\\	
	0 &\textnormal{if}\:z\in \mathbb{Z}_{\geq 0}.\\
	\end{array}	
	\right.
	\]	
\end{prop}

\begin{proof}
	This is a direct corollary of Lemma \ref{dominant}.
\end{proof}
From \cite{EH}, we know that the highest weight module $L(-k\xi_p)$ is the $k$-th Wallach representation of $SU(p,q)$ for $1
\leq k \leq r-1$, which is  unitary.

\begin{prop}
	Let $ \mathfrak{g}=\mathfrak{so}(2n+1,\mathbb{C})$.	Then we have	
		\begin{enumerate}		
		\item If $p=1$, then
		%	\[		{\rm GKdim}\:L(z\xi_1)=\left\{
		%\begin{array}{ll}
		%	2n-1 &\textnormal{if}\:\:z<\frac{3}{2}-n\\	     	  	   
		%	2n-2 &\textnormal{if}\:\:z=-\frac{3}{2}-n+k, k\ge0\\	
		%	0 &\textnormal{if}\:\:z\in \mathbb{Z}_{\geq 0}.
		%\end{array}	
		%\right.
		%\]
		
		\[		{\rm GKdim}\:L(z\xi_1)=\left\{
		\begin{array}{ll}
		2n-1 &\textnormal{if}\:z<\frac{3}{2}-n \emph{~or~} z\in \mathbb{Z}_{< 0} \emph{~or~} z\notin \frac{1}{2}\mathbb{Z}\\	     	  	   
		2n-2 &\textnormal{if}\:z=\frac{3}{2}-n+k, k\in \mathbb{Z}_{\geq 0}\\	
		0 &\textnormal{if}\:z\in \mathbb{Z}_{\geq 0}.\\
		\end{array}	
		\right.
		\]

		\item If $2\le p\le n-1$, then
		\begin{enumerate}
			\item When $3p\ge 2n+1$ and $p$ is odd, we have
			\[		{\rm GKdim}\:L(z\xi_p)=\left\{
			\begin{array}{ll}
			0 &\textnormal{if}\:\:z\in\mathbb{Z}_{\geq 0}\\
			2np-\frac{3}{2}p^2+\frac{p}{2}& \textnormal{if~} z<\frac{1}{2}-n+\frac{p}{2} \textnormal{~or~}	z\notin\frac{1}{2}\mathbb{Z}\\
			%	2np-2p^2 &\textnormal{if}\:\:z\in -\frac{1}{2}+\mathbb{Z}_{\geq 0}\\

			k(2n-2k+1) &\textnormal{if}\:\:z=-k, 1\le k\le n-p\\
			2np-2p^2+2kp-2k^2+n  &\textnormal{if}\:\:z=p-n-k, 1\le k\le n-p\\
			2np-2p^2+2kp-2k^2+k  &\textnormal{if}\:\:z=p-n-k, n-p< k\le \frac{p-1}{2}\\
			2np-2p^2 &\textnormal{if}\:\:z=-k-\frac{1}{2},  k<n-p \\	     	  	   
			%	2np-2p^2+(k-n+p)(2n-2k-1) &\textnormal{if}\:\:z=-k-\frac{1}{2}, n-p\le k\le n-\frac{p+3}{2}\\
			2np-2p^2+k(2p-2k-3)+2p-1 &\textnormal{if}\:\:z=p-n-k-\frac{1}{2}, 0\le k\le \frac{p-3}{2}.
			\end{array}	
			\right.
			\]	
			
			\item When $p$ is even, we have
			\[		{\rm GKdim}\:L(z\xi_p)=\left\{
			\begin{array}{ll}
			%	2np-2p^2 &\textnormal{if}\:\:z\in -\frac{1}{2}+\mathbb{Z}_{\geq 0}\\
			0 &\textnormal{if}\:\:z\in\mathbb{Z}_{\geq 0}\\
			2np-\frac{3}{2}p^2+\frac{1}{2}p &\textnormal{if}\:\:z<1-n+\frac{p}{2} \textnormal{~or~}	z\notin\frac{1}{2}\mathbb{Z}\\
			k(2n-2k+1) &\textnormal{if}\:\:z=-k, p\ge k\emph{~and~} 1\le k\le n-p\\
			p(2n-2p+1) &\textnormal{if}\:\:z=-k,  p< k\le n-p\\
			2np-2p^2+2kp-2k^2+n  &\textnormal{if}\:\:z=p-n-k,
			k<\frac{p}{2}\emph{~and~} 1\le k\le n-p\\
			2np-2p^2+2kp-2k^2+k  &\textnormal{if}\:\:z=p-n-k,  n-p< k\le \frac{p}{2}-1\\
			
			%2np-\frac{3}{2}p^2+\frac{1}{2}p &\textnormal{if}\:\:z=-(n-p)-k, k\ge \frac{1}{2}p\\
			
			2np-2p^2 &\textnormal{if}\:\:z=-k-\frac{1}{2}, k<n-p \\	     	  	   
			2np-2p^2+k(2p-2k-3)+2p-1 &\textnormal{if}\:\:z=p-n-k-\frac{1}{2}, 0\le k\le \frac{p-4}{2}.

			%		2n-1 &\textnormal{if}\:\:p<k\emph{~and~}z=-1\\
			
			\end{array}	
			\right.
			\]	
			
			\item When $3p<2n+1$ and $p$ is odd, we have
			\[		{\rm GKdim}\:L(z\xi_p)=\left\{
			\begin{array}{ll}
			0 &\textnormal{if}\:\:z\in \mathbb{Z}_{\geq 0}\\
			2np-\frac{3}{2}p^2+\frac{p}{2}& \textnormal{if~} z<1-n+\frac{p}{2} \textnormal{~or~}	z\notin\frac{1}{2}\mathbb{Z}\\
			k(2n-2k+1) &\textnormal{if}\:\:z=-k, p\ge k\emph{~and~}1\le k\le n-p\\
			p(2n-2p+1) &\textnormal{if}\:\:z=-k,  p< k\le n-p\\
			2np-2p^2+2kp-2k^2+n  &\textnormal{if}\:\:z=p-n-k, 1\le k\le 2p-n\\
			%2p-n\leq n-p
			2np-2p^2+(2k+1)(p-k) &\textnormal{if}\:\:z=p-n-k, 2p-n< k\leq \frac{p-3}{2}\\
			2np-2p^2 &\textnormal{if}\:\:z=-k-\frac{1}{2}, k<n-p \\	     	  	   
			2np-2p^2+k(2p-2k-3)+2p-1 &\textnormal{if}\:\:z=p-n-k-\frac{1}{2}, 0\le k\le \frac{p-3}{2}.

			%	2np-2p^2 &\textnormal{if}\:\:z\in -\frac{1}{2}+\mathbb{Z}_{\geq 0}\\	     	  	   
			%		2np-2p^2+k(2n-2k-3)+2p-1 &\textnormal{if}\:\:z=-(n-p)-k-\frac{1}{2}, 1<n-p\le k-1.	

			%	2np-2p^2+p &\textnormal{if}\:\:z=-k, p< k\emph{~and~}1\le k\le n-p\\

			%	2np-2p^2+(2k+1)(p-k) &\textnormal{if}\:\:z=-(n-p)-k, p+1>2k\emph{~and~}p< n-p+k\\	
			%	2np-\frac{3}{2}p^2+\frac{1}{2}p &\textnormal{if}\:\:z=-(n-p)-k, p+1\le 2k\emph{~and~}p< n-p+k.

			\end{array}	
			\right.
			\]
	\end{enumerate}

\item If $p=n$, then
\begin{enumerate}	
\item When $n$ is even, we have
\[		{\rm GKdim}\:L(z\xi_n)=\left\{
\begin{array}{ll}
\frac{1}{2}n^2+\frac{1}{2}n &\textnormal{if}\:\:z<2-n \emph{~or~} z\notin \mathbb{Z}\\	     	  	   
k(2n-2k+1) &\textnormal{if}\:\:z=-2k\:\:\textnormal{or}\:-2k+1, 1\le k\le \frac{n-2}{2}\\	
0 &\textnormal{if}\:\:z\in \mathbb{Z}_{\geq 0}.
\end{array}	
\right.
\]

\item When $n$ is odd, we have
\[		{\rm GKdim}\:L(z\xi_n)=\left\{
\begin{array}{ll}
\frac{1}{2}n^2+\frac{1}{2}n &\textnormal{if}\:\:z<1-n \emph{~or~} z\notin \mathbb{Z}\\	     	  	   
k(2n-2k+1) &\textnormal{if}\:\:z=-2k\:\:\textnormal{or}\:-2k+1, 1\le k\le \frac{n-1}{2}\\	
0 &\textnormal{if}\:\:z\in \mathbb{Z}_{\geq 0}.
\end{array}	
\right.
\]
\end{enumerate}	
\end{enumerate}	  	
%		\item if $p=n$, then
%		\begin{enumerate}	
%			\item $n$ is even\\
%			\[		{\rm GKdim}\:L(\lambda)=\left\{
%			\begin{array}{ll}
%			\frac{1}{2}n^2+\frac{1}{2}n &\textnormal{if}\:\:z<2-n\\	     	  	   
%			\frac{1}{2}n^2+\frac{1}{2}n-2i^2-i &\textnormal{if}\:\:z=2i-n\:\:\textnormal{or}\:\:2i+1-n\:\:(1\le i\le \frac{n-2}{2})\\	
%			0 &\textnormal{if}\:\:z\ge 0
%			\end{array}	
%			\right.
%			\]
%			
%			\item $n$ is odd
%			\[		{\rm GKdim}\:L(\lambda)=\left\{
%			\begin{array}{ll}
%			\frac{1}{2}n^2+\frac{1}{2}n &\textnormal{if}\:\:z<1-n\\	     	  	   
%			\frac{1}{2}n^2+\frac{1}{2}n-2i^2+i &\textnormal{if}\:\:z=2i-n\:\:\textnormal{or}\:\:2i-1-n\:\:(1\le i\le \frac{n-1}{2})\\	
%			0 &\textnormal{if}\:\:z\ge 0
%			\end{array}	
%			\right.
%			\]
%		\end{enumerate}	
%	\end{enumerate}		     
\end{prop}

\begin{proof}
	Some details of the arguments can be found in the proof of Proposition \ref{B}.
First we consider the case $p=1$.
\begin{enumerate}
	\item If $z\in\mathbb{Z}$, we know that $z=0$ is the first reducible point by Proposition \ref{B}. And 
	$ {\rm GKdim}\:L(z\xi_1)=0$ when $z\in \mathbb{Z}_{\geq 0}$. So 	$ {\rm GKdim}\:L(z\xi_1)=\dim (\mathfrak{u})=2n-1$ when  $z\in \mathbb{Z}_{< 0}$.
	
	\item If $z\in\frac{3}{2}-n+\mathbb{Z}_{\ge 0}$, we have the follows.
	\begin{enumerate}
		\item When $z=\frac{3}{2}-n$, we will have
		${\rm GKdim}\:L(\lambda)=2n-2$ from the proof the Proposition \ref{B}.
		
		\item When $z=\frac{3}{2}-n+k\:\:(k\ge 1)$, we will have
		$$(\lambda+\rho)^-=(z+n-\frac{1}{2},n-\frac{3}{2},\dots,\frac{1}{2},-\frac{1}{2},\dots,\frac{3}{2}-n,-z-n+\frac{1}{2}),$$	
		$$(\lambda+\rho)^-_{(0)}=(z+n-\frac{1}{2},-z-n+\frac{1}{2}).$$
		Since $z+n-\frac{1}{2}>-z-n+\frac{1}{2}$, we get
		$$p((\lambda+\rho)^-_{(0)})^{odd}=(0,1)\:\:{\rm and}\:\: \sum_{i\ge 1}(i-1)p((\lambda+\rho)^-_{(\frac{1}{2})})^{odd}=(n-1)^2.$$
		So \begin{align*}
		{\rm GKdim}\:L(\lambda)&=n^2-(n-1)^2-(0\cdot 0+1\cdot 1)\\&=2n-2.
		\end{align*}
	\end{enumerate}
\end{enumerate}

Thus, we completed the proof of this case.\\

Now we consider the case  $p=n$.

We can get from the previous calculation that
$\xi=(\frac{1}{2},\frac{1}{2},\dots,\frac{1}{2})$.
\begin{enumerate}
	\item $n$ is even. From Proposition \ref{B}, we know $z=2-n$ is the first reducible point for $M_I(z\xi_n)$.
	\begin{enumerate}
		\item When $z=2-n$, we will have
		$$(\lambda+\rho)^-=(\frac{1}{2}n+\frac{1}{2},\dots,\frac{3}{2}-\frac{1}{2}n,\frac{1}{2}n-\frac{3}{2},\dots,-\frac{1}{2}-\frac{1}{2}n),$$
		$$p((\lambda+\rho)^-)^{odd}=(\underbrace{1,1,\dots,1}_{n-2},0,1,0,1).$$
		\begin{align}\label{14}
			{\rm GKdim}\:L(\lambda)\notag&=n^2-(1\cdot 1+2\cdot 1+\dots+(n-3)\cdot 1+(n-1)\cdot 1+(n+1)\cdot 1)\notag\\&=n^2-(1+2+\dots+n-3)-2n\notag\\&=n^2-\frac{(n-2)(n-3)}{2}-2n=\frac{1}{2}n^2+\frac{1}{2}n-3.
		\end{align}
		
		\item When $z=3-n$, we will have
		$$(\lambda+\rho)^-=(\frac{1}{2}n+1,\dots,2-\frac{1}{2}n,\frac{1}{2}n-2,\dots,-1-\frac{1}{2}n),$$
		\begin{align*}
			p((\lambda+\rho)^-)^{odd}&=(\underbrace{1,1,\dots,1}_{n-3},1,0,1,0,1,0)\\&=(\underbrace{1,1,\dots,1}_{n-2},0,1,0,1).
		\end{align*}
		By (\ref{14}), $	{\rm GKdim}\:L(\lambda)=\frac{1}{2}n^2+\frac{1}{2}n-3$. 
		
		\item When $z=4-n$, we will have
		$$(\lambda+\rho)^-=(\frac{1}{2}n+\frac{3}{2},\dots,\frac{5}{2}-\frac{1}{2}n,\frac{1}{2}n-\frac{5}{2},\dots,-\frac{3}{2}-\frac{1}{2}n),$$
		$$p((\lambda+\rho)^-)^{odd}=(\underbrace{1,1,\dots,1}_{n-4},\underbrace{0,1,\dots,0,1}_{8}).$$
		\begin{align}\label{15}
			{\rm GKdim}\:L(\lambda)\notag&=n^2-(1\cdot 1+2\cdot 1+\dots+(n-5)\cdot 1)-4n\notag\\&=n^2-(1+2+\dots+n-5)-4n\notag\\&=n^2-\frac{(n-4)(n-5)}{2}-4n=\frac{1}{2}n^2+\frac{1}{2}n-10.
		\end{align}
		
		\item When $z=5-n$, we will have
		\begin{align*}
			p((\lambda+\rho)^-)^{odd}&=(\underbrace{1,1,\dots,1}_{n-5},\underbrace{1,0,\dots,1,0}_{10})\\&=(\underbrace{1,1,\dots,1}_{n-4},\underbrace{0,1,\dots,0,1}_{8}).
		\end{align*}
		By (\ref{15}), $	{\rm GKdim}\:L(\lambda)=\frac{1}{2}n^2+\frac{1}{2}n-10$.
		
		\item  when $z=-2$ or $-1$
		
		${\rm GKdim}\:L(\lambda)=2n-1$.
		
		So far, we have calculated all the ${\rm GKdim}\:L(\lambda)$, it is not difficult to find the law.
	\end{enumerate}
	
	\item $n$ is odd. From Proposition \ref{B}, we know $z=1-n$ is the first reducible point for $M_I(z\xi_n)$.
	
	The process is similar to the case when $n$ is even, so we can get that 
	\[		{\rm GKdim}\:L(\lambda)=\left\{
	\begin{array}{ll}
	\frac{1}{2}n^2+\frac{1}{2}n &\textnormal{if}\:\:z<1-n \textnormal{~or~} z\notin \mathbb{Z}\\	     	  	   
	k(2n-2k+1) &\textnormal{if}\:\:z=-2k\:\:\textnormal{or}\:-2k+1, 1\le k\le \frac{n-1}{2}\\	
	0 &\textnormal{if}\:\:z\in \mathbb{Z}_{\geq 0}.\\
	\end{array}	
	\right.\\
	\]
\end{enumerate}		

Now we consider the case when $p$ is odd and $3p\ge 2n+1$.

\begin{enumerate}
	\item When $z\in\mathbb{Z}$, $z=\frac{1}{2}-n+\frac{p}{2}$ is the first reducible point of $M_I(z\xi_p)$ by Proposition \ref{B}. We will have the follows.
	\begin{enumerate}
		\item When $z\in \mathbb{Z}_{}\geq 0$, we will have	
		$	{\rm GKdim}\:L(\lambda)=0$ by Proposition \ref{B}.
		
		\item When $z=-1$, we will have
		$$\lambda+\rho=(n-\frac{3}{2},\dots,n-p-\frac{1}{2},n-p-\frac{1}{2},\frac{1}{2}).$$
%		\[
%		\tiny{\begin{tikzpicture}[scale=\domscale+0.25,baseline=-45pt]
%			\hobox{0}{0}{p}
%			\hobox{1}{0}{p+1}
%			\hobox{0}{1}{p-1}
%			\hobox{0}{2}{\vdots}
%			\hobox{0}{3}{1}
%			\end{tikzpicture}}\to 
%		\tiny{\begin{tikzpicture}[scale=\domscale+0.25,baseline=-45pt]
%			\hobox{0}{0}{n}
%			\hobox{1}{0}{p+1}
%			\hobox{0}{1}{n-1}
%			\hobox{0}{2}{\vdots}
%			\hobox{0}{3}{1} 
%			\end{tikzpicture}}\dashrightarrow 
%		\tiny{\begin{tikzpicture}[scale=\domscale+0.25,baseline=-45pt]
%			\hobox{0}{0}{k}
%			\hobox{1}{0}{p+1}
%			\hobox{0}{1}{k-1}
%			\hobox{0}{2}{\vdots}
%			\hobox{0}{3}{1}
%			\end{tikzpicture}}\to
%		\tiny{\begin{tikzpicture}[scale=\domscale+0.25,baseline=-45pt]
%			\hobox{0}{0}{k}
%			\hobox{1}{0}{k+1}
%			\hobox{0}{1}{k-1}
%			\hobox{1}{1}{p+1}
%			\hobox{0}{2}{\vdots}
%			\hobox{0}{3}{1}     
%			\end{tikzpicture}}\dashrightarrow
%		\tiny{\begin{tikzpicture}[scale=\domscale+0.25,baseline=-45pt]
%			\hobox{0}{0}{2n}
%			\hobox{1}{0}{k+1}
%			\hobox{0}{1}{2n-1}
%			\hobox{1}{1}{p+1}
%			\hobox{0}{2}{\vdots}
%			\hobox{0}{3}{1}     
%			\end{tikzpicture}}=P(\lambda).
%		\]
%		
%		In the above tableau, we denote	 $k=2n-p$.
		
		So	$p((\lambda+\rho)^-)^{odd}=(1,1,\underbrace{0,1,\dots,0,1}_{2n-4}).$
		\begin{align}\label{1}
			{\rm GKdim}\:L(\lambda)\notag&=n^2-\sum\limits_{i\ge 1}(i-1)p((\lambda+\rho)^-)_i^{odd}\notag\\&=n^2-(0\cdot 1+1\cdot 1+2\cdot 0+3\cdot 1+\dots +(2n-4)\cdot 0+(2n-3)\cdot 1)\notag\\&=n^2-n^2+2n-1=2n-1<\dim(\mathfrak{u}).
		\end{align}
		
		\item When $z=-2$, we have $p((\lambda+\rho)^-)=(2^4,1^{2n-8})$. 
		
		So 
		$p((\lambda+\rho)^-)^{odd}=(1,1,1,1,\underbrace{0,1,\dots,0,1}_{2n-8})$.
		\begin{align}
			{\rm GKdim}\:L(\lambda)\notag&=n^2-\sum\limits_{i\ge 1}(i-1)p((\lambda+\rho)^-)_i^{odd}\notag\\&=n^2-(1\cdot 1+2\cdot 1+3\cdot 1+5\cdot 1+\dots +(2n-4-1)\cdot 1)\notag\\&=n^2-2-(n-2)^2=4n-6<\dim(\mathfrak{u}).
		\end{align}
		\item When $z=-k$, $1\leq k\leq n-p$, we have	$$p((\lambda+\rho)^-)=(2^{2k},1^{2n-4k}),$$ 
		and	$	{\rm GKdim}\:L(\lambda)=2k(n-k+\frac{1}{2})$.
		
		\item When $z=-(n-p)-k$, $1\leq k\leq n-p$, we have $${p((\lambda+\rho)^-)=(3^{2k},2^{2n-2p-2k},1^{4p-2k-2n})}.$$
		So $	{\rm GKdim}\:L(\lambda)=2np-2p^2+2kp-2k^2+n$.
		
		\item When $z=-(n-p)-k$, $1+n-p\leq k\leq \frac{p-1}{2}$, we have $$p((\lambda+\rho)^-)=(3^{2n-2p},2^{2k-2n+2p},1^{2p-4k}).$$
		So $	{\rm GKdim}\:L(\lambda)=2np-2p^2+2kp-2k^2+k$.	
		
		%\textcolor{red}{When $p=n-1$ and $p\le 2(n-p+k)-3$, we have	
		%$$p((\lambda+\rho)^-)=(3^2,2^{n-3}),$$	
		%So $	{\rm GKdim}\:L(\lambda)=\frac{1}{2}n^2+\frac{3}{2}n-2$.}	
		
		\item When $z=\frac{1}{2}-n+\frac{p}{2}$, we will have
		$$(\lambda+\rho)^-=(\frac{p}{2},\dots,-\frac{p}{2}+1,n-p-\frac{1}{2},\dots,\frac{1}{2},-\frac{1}{2},\dots,-n+p+\frac{1}{2},\frac{p}{2}-1,\dots,-\frac{p}{2}).$$
		So $p((\lambda+\rho)^-)^{odd}=(\underbrace{1,2,\dots,1,2}_{2n-2p},\underbrace{1,\dots,1}_{3p-2n-1},0,1).$
		
	$
			{\rm GKdim}\:L(\lambda)=2np-\frac{3}{2}p^2+\frac{p}{2}-1.
		$

		%		\item When $z=-\frac{1}{2}-n+\frac{p}{2}$, we will have
		%		$$(\lambda+\rho)^-=(\frac{p}{2}-1,\dots,-\frac{p}{2},n-p-\frac{1}{2},\dots,\frac{1}{2},-\frac{1}{2},\dots,-n+p+\frac{1}{2},\frac{p}{2},\dots,-\frac{p}{2}+1),$$
		%		So $p((\lambda+\rho)^-)^{odd}=(\underbrace{1,2,\dots,1,2}_{2n-2p},\underbrace{1,1,\dots,1}_{3p-2n})$.
		%		\begin{align*}
		%		{\rm GKdim}\:L(\lambda)&=n^2-\sum\limits_{i\ge 1}(i-1)p((\lambda+\rho)^-)_i^{odd}\\&=n^2-(2+2+3\cdot 2 +4+\dots+2n-2p-2+(2n-2p-1)\cdot 2)\\&-(2n-2p+\dots+p-1)\\&=n^2-2C^2_{n-p}-2(n-p)^2-\frac{(2n-p-1)(3p-2n)}{2}\\&=2np-\frac{3}{2}p^2+\frac{p}{2}=\dim(\mathfrak{u}).
		%		\end{align*}
		
	\end{enumerate}	
	% $z=\frac{1}{2}-n+\frac{p}{2}$ is the first reducible point.
	\item When $z\notin\frac{1}{2}\mathbb{Z}$, $M_{I}(\lambda)$ is irreducible.
	
	\item When $z\in\frac{1}{2}+\mathbb{Z}$,  $z=1-n+\frac{p}{2}$ is the first reducible point of $M_{I}(z\xi_p)$  by Proposition \ref{B}. We will have the follows.
	\begin{enumerate}
		\item When $z\ge-\frac{1}{2}$, we will have
		${\rm GKdim}\:L(\lambda)=n^2-p^2-(n-p)^2=2np-2p^2$.

		\item When $z=-\frac{3}{2}$, we will have
		\begin{align}\label{2}
			(\lambda+\rho)^-_{(\frac{1}{2})}=(n-p-\frac{1}{2},\dots,\frac{1}{2},-\frac{1}{2},\dots,-n+p+\frac{1}{2}),
		\end{align}
		\begin{align}\label{3}
			(\lambda+\rho)_{(0)}=(n-2,\dots,n-p-1,-n+p+1,\dots,-n+2).
		\end{align}
		We divide the discussion into two cases:
		\begin{enumerate}
			\item	 $p=n-1$. We will have $$p((\lambda+\rho)^-_{(0)})=(2, 1^{2p-2}),$$
			$$p((\lambda+\rho)^-_{(0)})^{odd}=(\underbrace{0,1,\dots,0,1}_{2p-2}).$$
			$
				{\rm GKdim}\:L(\lambda)=2np-2p^2+2p-1.
			$
			
			\item $p<n-1$. We will have
			$$p((\lambda+\rho)^-_{(0)})^{odd}=(\underbrace{0,1,\dots,0,1}_{2p}).$$
		$
				{\rm GKdim}\:L(\lambda)=2np-2p^2.
		$
		\end{enumerate}	
		In general, when $z=-k-\frac{1}{2}$, we have 
		$$	{\rm GKdim}\:L(\lambda)=2np-2p^2 \text{~if~} k<n-p.$$ And 	$$p((\lambda+\rho)^-_{(0)})=(2^{2p+2k+1-2n},1^{4n-2p-4k-2}) \text{~if } n-p\leq k\leq n-\frac{p+3}{2}.$$
		So $	{\rm GKdim}\:L(\lambda)=2np-2p^2+(k-n+p)(2n-2k-3)+2p-1$.
		In other words, when $z=-(n-p)-k-\frac{1}{2}$ with $0\leq k\leq\frac{p-3}{2}$, we have $$	{\rm GKdim}\:L(\lambda)=2np-2p^2+k(2p-2k-3)+2p-1.$$

	\end{enumerate}
	
	%	Hence $z=1-n+\frac{p}{2}$ is the first reducible point. 
\end{enumerate}

\hspace{1cm}	

Next we consider the case when $p$ is even.
\begin{enumerate}
	\item When $z\in\mathbb{Z}$, $z=1-n+\frac{p}{2}$ is the first recucible point of $M_I(z\xi_p)$ by Proposition \ref{B}. We have the follows.
	\begin{enumerate}
		\item When $z\ge 0$, we will have	
		${\rm GKdim}\:L(\lambda)=0$.
		
		\item When $z=-1$, 			
		we have ${\rm GKdim}\:L(\lambda)=2n-1$ since  (\ref{1}).
		
		\item When $z=1-n+\frac{p}{2}$, we divide the discussion into three cases:
		\begin{enumerate}	
			\item $3p\le 2n-2$. We have
			$$p((\lambda+\rho)^-)^{odd}=(\underbrace{1,2,\dots,1,2}_{p-2},1,1,1,1,\underbrace{0,1,\dots,0,1}_{2n-3p-2}).$$
		$
				{\rm GKdim}\:L(\lambda)=2np-\frac{3}{2}p^2+\frac{p}{2}-1.
		$
			\item $n=p+1$ and $p>2$. We have 
			$$p((\lambda+\rho)^-)^{odd}=(1,2,\underbrace{1,1,\dots,1}_{2n-p-6},\underbrace{0,1,\dots,0,1}_{2p-2n+6}).$$
		$
				{\rm GKdim}\:L(\lambda)=2np-\frac{3}{2}p^2+\frac{1}{2}p-3.
		$
			
			\item $3p\ge 2n$ and $n\neq p+1$. We have
			$$p((\lambda+\rho)^-)^{odd}=(\underbrace{1,2,\dots,1,2}_{p-2},\underbrace{1,1,\dots,1}_{2n-3p+2},\underbrace{0,1,\dots,0,1}_{3p-2n+2}).$$
		$
				{\rm GKdim}\:L(\lambda)=2np-\frac{3}{2}p^2-\frac{3p}{2}+n-\frac{13}{4}.
		$
		\end{enumerate}

		%		To sum up, $z=1-n+\frac{p}{2}$ is a reducible point.
		
		\item When $z=-n+\frac{p}{2}$, we will have
		$$(\lambda+\rho)^-=(\frac{p}{2}-\frac{1}{2},\dots,-\frac{p}{2}+\frac{1}{2},n-p-\frac{1}{2},\dots,\frac{1}{2},-\frac{1}{2},\dots,-n+p+\frac{1}{2},\frac{p}{2}-\frac{1}{2},\dots,-\frac{p}{2}+\frac{1}{2}).$$	
		
		\begin{align*}
			{\rm GKdim}\:L(\lambda)=2np-\frac{3}{2}p^2+\frac{p}{2}=\dim(\mathfrak{u}).
		\end{align*}

		%		So, $z=-n+\frac{p}{2}$ is an irreducible point. 		
		%
		%	
		%	So $z=1-n+\frac{p}{2}$ is the first reducible point.
		
		\item When $p\ge k$, 
		$z=-k$, $1\leq k\leq n-p$, we have	$$p((\lambda+\rho)^-)=(2^{2k},1^{2n-4k}),$$ 
		and	$	{\rm GKdim}\:L(\lambda)=2k(n-k+\frac{1}{2})$.
		
		\item	When $z=-(n-p)-k$, $1\leq k\leq n-p$, we have $$p((\lambda+\rho)^-)=(3^{2k},2^{2n-2p-2k},1^{4p-2k-2n}),$$
		and $	{\rm GKdim}\:L(\lambda)=2np-2p^2+2kp-2k^2+n$.
		
		\item	When $z=-(n-p)-k$, $1+n-p\leq k\leq \frac{p}{2}-1$, we have $$p((\lambda+\rho)^-)=(3^{2n-2p},2^{2k-2n+2p},1^{2p-4k}),$$
		and $	{\rm GKdim}\:L(\lambda)=2np-2p^2+2kp-2k^2+k$.	
		
		\item When $p<k$,	 $z=-k$, $p< k\leq n-p$, we have	$$p((\lambda+\rho)^-)=(2^{2p},1^{2n-4p}),$$ 
		and	$	{\rm GKdim}\:L(\lambda)=p(2n-2p+1)$.
		
		%	When $z=-(n-p)-k$, $ k\ge\frac{p}{2}\rm{~and~}2n\le 3p$, we have $$p((\lambda+\rho)^-)=(3^{2n-2p},2^{3p-2n}),$$
		%	and $	{\rm GKdim}\:L(\lambda)=2np-\frac{3}{2}p^2+\frac{1}{2}p$.
		%	
		%		When $z=-(n-p)-k$, $ k\ge\frac{p}{2}\rm{~and~}2n\ge 3p$, we have $$p((\lambda+\rho)^-)=(3^p,1^{2n-3p}),$$
		%	and $	{\rm GKdim}\:L(\lambda)=2np-\frac{3}{2}p^2+\frac{1}{2}p$.
		%	
		\item When $z<1-n+\frac{p}{2}$, we have  $$	{\rm GKdim}\:L(\lambda)=2np-\frac{3}{2}p^2+\frac{1}{2}p.$$
	\end{enumerate}		
	\item When  $z\notin\frac{1}{2}\mathbb{Z}$,
	$M_{I}(\lambda)$ is irreducible, so we have$$	{\rm GKdim}\:L(\lambda)=2np-\frac{3}{2}p^2+\frac{1}{2}p.$$
	
	\item When  $z\in\frac{1}{2}+\mathbb{Z}$, $z=\frac{3}{2}-n+\frac{p}{2}$ is the first recucible point of $M_I(z\xi_p)$ by Proposition \ref{B}. We have the follows.
	\begin{enumerate}
		\item When $z\ge -\frac{1}{2}$, we will have
		\begin{align}\label{17}
			{\rm GKdim}\:L(\lambda)\notag&=n^2-(1+3+\dots+2p-1)-(n-p)^2\notag\\&=n^2-p^2-n^2+2np-p^2\notag\\&=2np-2p^2.
		\end{align}		
		%		By Lemma \ref{reducible}, $M_{I}(\lambda)$ is reducible.	
		
		%		\item When $z=-\frac{3}{2}$, by (\ref{2}) and (\ref{3}), $M_{I}(\lambda)$ is reducible.	
		%		
		%		\item When $z=\frac{3}{2}-n+\frac{p}{2}$, we will have	
		%		$$\lambda+\rho=(\frac{p}{2}+1,\dots,-\frac{p}{2}+2,n-p-\frac{1}{2},\dots,\frac{1}{2}),$$
		%		$$(\lambda+\rho)^-_{(0)}=(\frac{p}{2}+1,\dots,-\frac{p}{2}+2,\frac{p}{2}-2,\dots,-\frac{p}{2}-1),$$
		%		$$p((\lambda+\rho)^-)^{odd}=(\underbrace{1,1,\dots,1}_{p-3},1,0,1,0,1,0).$$
		%		\begin{align*}
		%		{\rm GKdim}\:L(\lambda)&=n^2-(1+2+3+\dots+p-3+p-1+p+1)-(n-p)^2\\&=n^2-\frac{(p-2)(p-3)}{2}-2p-(n-p)^2\\&=2np-\frac{3}{2}p^2+\frac{p}{2}-3=\dim(\mathfrak{u})-3.
		%		\end{align*}	
		%		So $z=\frac{3}{2}-n+\frac{p}{2}$ is a reducible point.	
		%		
		%		\item When $z=\frac{1}{2}-n+\frac{p}{2}$, we will have
		%		
		%		$$\lambda+\rho=(\frac{p}{2},\dots,-\frac{p}{2}+1,n-p-\frac{1}{2},\dots,\frac{1}{2}),$$
		%		$$(\lambda+\rho)^-_{(0)}=(\frac{p}{2},\dots,-\frac{p}{2}+1,\frac{p}{2}-1,\dots,-\frac{p}{2}),$$
		%		$$p((\lambda+\rho)^-)^{odd}=(\underbrace{1,1,\dots,1}_{p},1,0).$$
		%		\begin{align}\label{4}
		%		{\rm GKdim}\:L(\lambda)\notag&=n^2-(1+2+3+\dots+p-2+p-1)-(n-p)^2\notag\\&=n^2-\frac{p(p-1)}{2}-(n-p)^2\notag\\&=2np-\frac{3}{2}p^2+\frac{p}{2}=\dim(\mathfrak{u}).
		%		\end{align}

		\item	In general, when $z=-k-\frac{1}{2}$, we have 
		$$	{\rm GKdim}\:L(\lambda)=2np-2p^2 \text{~if~} k<n-p.$$   
		\item
		When $z=-(n-p)-k-\frac{1}{2}$, $0\le k\le \frac{p-4}{2}$, we have	$$p((\lambda+\rho)^-_{(0)})=(2^{2p+2k+1-2n},1^{4n-2p-4k-2}).$$
		So $	{\rm GKdim}\:L(\lambda)=2np-2p^2+k(2p-2k-3)+2p-1$.\\
	\end{enumerate}
\end{enumerate}

Finally we consider the case when $p$ is odd and $3p<2n+1$.
\begin{enumerate}
	\item When $z\in\mathbb{Z}$, $z=\frac{3}{2}-n+\frac{p}{2}$ is the first recucible point of $M_I(z\xi_p)$ by Proposition \ref{B}. We have the follows.
	\begin{enumerate}	
		\item When $z\in\mathbb{Z}_{\ge 0}$,	we will have
		
		${\rm GKdim}\:L(\lambda)=0$.	
		
		\item When $z=-1$, we will have
		
		${\rm GKdim}\:L(\lambda)=2n-1$.
		
		\item When $z=\frac{3}{2}-n+\frac{p}{2}$, we will have
		$$\lambda+\rho=(\frac{p}{2}+1,\dots,-\frac{p}{2}+2,n-p-\frac{1}{2},\dots,\frac{1}{2}).$$
		
		%		\item If $p$ is even
		%		$$(\lambda+\rho)^-_{(0)}=(\frac{p}{2}+1,\dots,-\frac{p}{2}+2,\frac{p}{2}-2,\dots,-\frac{p}{2}-1)$$
		%		$$p((\lambda+\rho)^-_{(0)})^{odd}=(\underbrace{1,1,\dots,1}_{p-3},1,0,1,0,1,0)$$	
		%		\begin{align*}
		%		{\rm GKdim}\:L(\lambda)&=n^2-(1+2+3+\dots+p-3+p-1+p+1)-(n-p)^2\\&=n^2-\frac{(p-2)(p-3)}{2}-(n-p)^2-2p\\&=2np-\frac{3}{2}p^2+\frac{p}{2}-3=\dim(\mathfrak{u})-3
		%		\end{align*}	
		
		Since $p$ is odd, we have the follows.
		\begin{enumerate}
			\item When $3p=2n-1$, we will have
			$$p((\lambda+\rho)^-_{(0)})^{odd}=(\underbrace{1,2,\dots,1,2}_{p-3},1,1,1,1,0,1).$$	
			$
				{\rm GKdim}\:L(\lambda)=2np-\frac{3}{2}p^2+\frac{p}{2}-4.
		$
			
			\item When $3p<2n-1$, we will have
			$$p((\lambda+\rho)^-_{(0)})^{odd}=(\underbrace{1,2,\dots,1,2}_{p-3},1,1,1,1,1,1,\underbrace{0,1,\dots,0,1}_{2n-3p-3}).$$	
		$
				{\rm GKdim}\:L(\lambda)=2np-\frac{3}{2}p^2+\frac{p}{2}-3.
		$
		\end{enumerate}

		%	Hence when $z=\frac{3}{2}-n+\frac{p}{2}$, $M_{I}(\lambda)$ will be reducible.
		%	
		\item In general, when $p\ge k$, 
		$z=-k$, $1\le k\le n-p$, we have 
		$$p(\lambda+\rho)^-=(2^{2k},1^{2n-2k}),$$
		and ${\rm GKdim}\:L(\lambda)=k(2n-2p+1).$
		
		\item When $z=-k$, $p< k\le n-p$, we have 
		$$p(\lambda+\rho)^-=(2^{2p},1^{2n-2p}),$$
		and ${\rm GKdim}\:L(\lambda)=p(2n-2p+1).$
		
		\item When $p\ge n-p+k$,
		$z=-(n-p)-k$, $1\le k\le 2p-n$, we have 
		$$p(\lambda+\rho)^-=(3^{2k},2^{2n-2p-2k},1^{4p-2n-2k}),$$
		and ${\rm GKdim}\:L(\lambda)=2np-2p^2-2k^2+2kp+n.$
		
		\item When $p<n-p+k$,
		$z=-(n-p)-k$, $2p-n<k\leq \frac{p-3}{2}$, we have 
		$$p(\lambda+\rho)^-=(3^{2k},2^{2p-4k},1^{-4p+2n+2k}),$$
		and ${\rm GKdim}\:L(\lambda)=2np-2p^2-2k^2+2kp+p-k.$
		
		\item When $z=-(n-p)-k$, $p-3< 2k$, equivalently $z<\frac{p+3}{2}-n$, we have 
		$$p(\lambda+\rho)^-=(3^p,1^{2n-3p}),$$
		and ${\rm GKdim}\:L(\lambda)=2np-\frac{3}{2}p^2+\frac{1}{2}p.$
		
	\end{enumerate}	
	\item When $z\in\frac{1}{2}+\mathbb{Z}$,  $z=1-n+\frac{p}{2}$ is the first recucible point of $M_I(z\xi_p)$ by Proposition \ref{B}. We have the follows. 
	
	\begin{enumerate}
		\item When $z\ge -\frac{1}{2}$, by (\ref{17}) we have
		${\rm GKdim}\:L(\lambda)=2np-2p^2$.

		\item When $z=1-n+\frac{p}{2}$, we will have
		$$(\lambda+\rho)^-_{(0)}=(\frac{p}{2}+\frac{1}{2},\dots,-\frac{p}{2}+\frac{3}{2},\frac{p}{2}-\frac{3}{2},\dots,-\frac{p}{2}-\frac{1}{2}),$$
		$$p((\lambda+\rho)^-_{(0)})^{odd}=(\underbrace{1,1,\dots,1}_{p-2},1,0,1,0).$$	
	$
			{\rm GKdim}\:L(\lambda)=2np-\frac{3}{2}p^2+\frac{p}{2}-1=\dim(\mathfrak{u})-1.
	$
		\item	In general, when $z=-k-\frac{1}{2}$, we have 
		$$	{\rm GKdim}\:L(\lambda)=2np-2p^2 \text{~if~} k<n-p.$$   
		\item
		When $z=-(n-p)-k-\frac{1}{2}$, $0\le k\le \frac{p-3}{2}$, we have	$$p((\lambda+\rho)^-_{(0)})=(2^{2p+2k+1-2n},1^{4n-2p-4k-2}).$$
		So $	{\rm GKdim}\:L(\lambda)=2np-2p^2+k(2p-2k-3)+2p-1$.

		%
		%\item In general, when $z=-k-\frac{1}{2}$, we have 
		%$$	{\rm GKdim}\:L(\lambda)=2np-2p^2+(k-n+p)(2n-2k-3)+2p-1\:\:\rm{if} \:1<n-p\le k-1.$$ 	
		%	

	\end{enumerate}	
	
\end{enumerate}	
%	Hence when $z=1-n+\frac{p}{2}$, $M_{I}(\lambda)$ is reducible.	

\end{proof}

From \cite{EH}, we know that the highest weight module $L(-(n-\frac{3}{2})\xi_1)$ is the unique Wallach representation of $SO(2,2n-1)$, which is  unitary.

For type  $B_n$, we have given the detailed arguments for  the calculations in the case when $2\le p\le n-1$. Now we omit some similar arguments for types $C_n$ and $D_n$.

\begin{prop}
	Let $ \mathfrak{g}=\mathfrak{so}(2n,\mathbb{C}) $ and $L(z\xi_p)$ be  a scalar highest weight module.	Then we have the follows.
\begin{enumerate}
	\item If $p=1$, then
	
	\[	{\rm GKdim}\:L(z\xi_1)=\left\{
	\begin{array}{ll}
	2n-2 &\textnormal{if}\:\:z<2-n \emph{~or~}z\notin \mathbb{Z}\\	 
	
	2n-3 &\textnormal{if}\:\:2-n\le z\le -1\\
	
	0 &\textnormal{if}\:\:z\in \mathbb{Z}_{\geq 0}.\\
	\end{array}	
	\right.
	\]
	
	\item If $2\le p\le n-2$, then
	\begin{enumerate}
		\item When $3p\ge 2n$ and $p$ is even, we have
		\[	{\rm GKdim}\:L(z\xi_p)=\left\{
		\begin{array}{ll}
		0 &\textnormal{if}\:\:z\in \mathbb{Z}_{\geq 0}\\
		2np-\frac{3}{2}p^2-\frac{1}{2}p &\textnormal{if}\:\:z<1+\frac{p}{2}-n \emph{~or~} z\notin \frac{1}{2}\mathbb{Z}\\
		2nk-2k^2-k &\textnormal{if}\:\:z=-k, 1\leq k< n-p\\

		%			2np-\frac{3}{2}p^2-\frac{1}{2}p &\textnormal{if}\:\:z=-k, n-p\le k-\frac{p}{2}\\
		%		   2np-2p^2+k(2p-2k-2)+n-p-1 &\textnormal{if}\:\:z=-(n-p)-k, k<\frac{p}{2}\emph{~and~}0\le k<n-p\\
		
		2np-2p^2+k(2p-2k-2)+n-p-1 &\textnormal{if}\:\:z=p-n-k, 0\le k<n-p\\

		2np-2p^2+k(2p-2k-1) &\textnormal{if}\:\:z=p-n-k, n-p\leq k<\frac{p}{2}\\

		2np-2p^2 &\textnormal{if}\:\:z=-k+\frac{1}{2}, k\leq  n-p\\	 
		
		%			2np-\frac{3}{2}p^2-\frac{1}{2}p &\textnormal{if}\:\:z=p-n-\frac{k}{2}\in\frac{1}{2}+\mathbb{Z}, p\le k+1\\
		%	2np-2p^2+(k+1)(p-1-\frac{k}{2}) &\textnormal{if}\:\:z=p-n-\frac{k}{2}\in\frac{1}{2}+\mathbb{Z}, 1\le k\le p-3\\		 
		
		2np-2p^2+k(2p-2k-5)+2p-3 &\textnormal{if}\:\:z=p-n-k-\frac{1}{2}, 0\le k\le \frac{p-4}{2}.

		\end{array}	
		\right.
		\]
		
		\item When $p$ is odd, we have
		\[\small{	{\rm GKdim}\:L(z\xi_p)=\left\{
			\begin{array}{ll}
			0 &\textnormal{if}\:\:z\in \mathbb{Z}_{\geq 0}\\
			2np-\frac{3}{2}p^2-\frac{1}{2}p &\textnormal{if}\:\:z<\frac{3+p}{2}-n \emph{~or~} z\notin \frac{1}{2}\mathbb{Z}\\
			2nk-2k^2-k &\textnormal{if}\:\:z=-k, k<n-p\emph{~and~}1\le k\le p\\
			2np-2p^2-p &\textnormal{if}\:\:z=-k, p<k<n-p\\
			%		 2np-\frac{3}{2}p^2-\frac{1}{2}p &\textnormal{if}\:\:z=-k, n-p\le k-\frac{p+1}{2}\\
			%		 2np-\frac{3}{2}p^2-\frac{1}{2}p &\textnormal{if}\:\:z=-(n-p)-k, 2p-n<k,  \emph{~and~}0< k<n-p\\
			2np-2p^2+k(2p-2k-3)+p-1 &\textnormal{if}\:\:z=p-n-k, \max\{0,2p-n\}<k\le\frac{p-3}{2}\\
			2np-2p^2+k(2p-2k-2)+n-p-1 &\textnormal{if}\:\:z=p-n-k, k\leq 2p-n\emph{~and~}0\le k<n-p\\
			2np-2p^2+k(2p-2k-1) &\textnormal{if}\:\:z=p-n-k,  n-p\leq k\leq \frac{p-3}{2}\\
			2np-2p^2 &\textnormal{if}\:\:z=-k+\frac{1}{2}, k\leq  n-p\\	 			
			%			2np-\frac{3}{2}p^2-\frac{1}{2}p &\textnormal{if}\:\:z=p-n-\frac{k}{2}\in\frac{1}{2}+\mathbb{Z}, p\le k\\
			%			2np-2p^2+(k+1)(p-1-\frac{k}{2}) &\textnormal{if}\:\:z=p-n-\frac{k}{2}\in\frac{1}{2}+\mathbb{Z}, 1\le k\le p-4\\
			2np-2p^2+k(2p-2k-5)+2p-3 &\textnormal{if}\:\:z=p-n-k-\frac{1}{2}, 0\le k\le \frac{p-5}{2}.\\
				\end{array}	
			\right.}
		\]

		\item When $3p<2n$ and $p$ is even, we have
		\[	{\rm GKdim}\:L(z\xi_p)=\left\{
		\begin{array}{ll}
		0 &\textnormal{if}\:\:z\in \mathbb{Z}_{\geq 0}\\
		2np-\frac{3}{2}p^2-\frac{1}{2}p &\textnormal{if}\:\:z<\frac{3+p}{2}-n \emph{~or~} z\notin \frac{1}{2}\mathbb{Z}\\

		%					2np-2p^2 &\textnormal{if}\:\:z\in p-n+\frac{1}{2}+\mathbb{Z}_{\geq 0}\\	 			
		%					2np-\frac{3}{2}p^2-\frac{1}{2}p &\textnormal{if}\:\:z=p-n-\frac{k}{2}\in\frac{1}{2}+\mathbb{Z}, p\le k+1\\
		%					2np-2p^2+(k+1)(p-1-\frac{k}{2}) &\textnormal{if}\:\:z=p-n-\frac{k}{2}\in\frac{1}{2}+\mathbb{Z}, 1\le k\le p-3\\
		
		2nk-2k^2-k &\textnormal{if}\:\:z=-k, 1\leq k< n-p\emph{~and~}k<p\\
		2np-2p^2-p &\textnormal{if}\:\:z=-k, p\le k< n-p\\		
		%			2np-\frac{3}{2}p^2-\frac{1}{2}p	 &\textnormal{if}\:\:z=-k,  n-p\le k-\frac{1}{2}p\\
		2np-2p^2+k(2p-2k-2)+n-p-1 &\textnormal{if}\:\:z=p-n-k,   0\leq k \leq 2p-n\\	
		2np-2p^2+k(2p-2k-3)+p-1 &\textnormal{if}\:\:z=p-n-k, 2p-n<0\le k\le\frac{p-3}{2}\\
		
		2np-2p^2 &\textnormal{if}\:\:z=-k+\frac{1}{2}, k\leq  n-p\\
		2np-2p^2+k(2p-2k-5)+2p-3 &\textnormal{if}\:\:z=p-n-k-\frac{1}{2}, 0\le k\le \frac{p-4}{2}.\\
		
		\end{array}	
		\right.
		\]

		%				\item if $p$ is odd, then	
		%			\[	{\rm GKdim}\:L(z\xi_1)=\left\{
		%		\begin{array}{ll}
		%			2np-2p^2 &\textnormal{if}\:\:z\in p-n+\frac{1}{2}+\mathbb{Z}_{\geq 0}\\	 			
		%			2np-\frac{3}{2}p^2-\frac{1}{2}p &\textnormal{if}\:\:z=p-n-\frac{k}{2}, p\le k\\
		%			2np-2p^2+(k+1)(p-1-\frac{k}{2}) &\textnormal{if}\:\:z=p-n-\frac{k}{2}, p>k \emph{~and~}1\le k\le p-1\\
		%			0 &\textnormal{if}\:\:z\in \mathbb{Z}_{\geq 0}\\
		%			2nk-2k^2-k &\textnormal{if}\:\:z=-k, n-p\ge k+1\emph{~and~}p\ge k\\
		%		2np-2p^2-p &\textnormal{if}\:\:z=-k, p=3, n-p\ge k+1\emph{~and~}k\ge 3\\
		%		2np-\frac{3}{2}p^2-\frac{1}{2}p &\textnormal{if}\:\:z=-(n-p)-k, p< n-p+k,  \emph{~and~}0< k\le n-p+1\\
		%		2np-\frac{3}{2}p^2-\frac{1}{2}p-1 &\textnormal{if}\:\:z=-(n-p)-k, p< n-p+k,  \emph{~and~}k=0\\
		%		2np-2p^2-2k(k-p+1)+n-p-1 &\textnormal{if}\:\:z=-(n-p)-k, p\ge n-p+k,  k<\frac{p+1}{2}\emph{~and~}0\le k<n-p.\\
		%		\end{array}	
		%		\right.
		%		\]		

	\end{enumerate}

	\item If $p=n-1\:\:or\:\:p=n$, then
	\begin{enumerate}
		\item When $n$ is even, we have				
		\[		{\rm GKdim}\:L(z\xi_p)=\left\{
		\begin{array}{ll}
		\frac{1}{2}n^2-\frac{1}{2}n &\textnormal{if}\:\:z<2-n \emph{~or~}z\notin \mathbb{Z}\\	     	  	   
		k(2n-2k-1) &\textnormal{if}\:\:z=-2k\emph{~or~}-2k+1, 1\le k\le \frac{n-2}{2}\\
		
		0 &\textnormal{if}\:\:z\in \mathbb{Z}_{\geq 0}.\\
		\end{array}	
		\right.
		\]	
		
		\item When $n$ is odd, we have
		\[		{\rm GKdim}\:L(z\xi_p)=\left\{
		\begin{array}{ll}
		\frac{1}{2}n^2-\frac{1}{2}n &\textnormal{if}\:\:z<3-n \emph{~or~}z\notin \mathbb{Z}\\	     	  	   
		k(2n-2k-1) &\textnormal{if}\:\:z=-2k\emph{~or~}-2k+1, 1\le k\le \frac{n-3}{2}\\
		
		0 &\textnormal{if}\:\:z\in \mathbb{Z}_{\geq 0}.\\
		\end{array}	
		\right.
		\]	
	\end{enumerate}		
\end{enumerate}		     
\end{prop}	

\begin{proof}
The details of the arguments for the case $p=1$ can be found in the proof of Proposition \ref{C}. 

Now we consider the case $p=n$. From Proposition \ref{C}, $z=\frac{1}{2}-\frac{n}{2}$ is the first reducible point of $M_I(z\xi_n)$.
\begin{enumerate}	
	\item When $z=\frac{1}{2}-\frac{n}{2}$, we will have
	${\rm GKdim}\:L(\lambda)=\frac{1}{2}n^2+\frac{1}{2}n-1$.
	
	\item When $z=1-\frac{n}{2}$, we will have
	$$p((\lambda+\rho)^-)^{odd}=(\underbrace{1,\dots,1}_{n-3},1,0,1,0,1,0)=(\underbrace{1,\dots,1}_{n-2},0,1,0,1).$$	
	By (\ref{33}), 
	${\rm GKdim}\:L(\lambda)=\frac{1}{2}n^2+\frac{1}{2}n-3$.
	
	\item When $z=\frac{3}{2}-\frac{n}{2}$, we will have
	$$(\lambda+\rho)^-=(\frac{3}{2}+\frac{n}{2},\dots,-\frac{n}{2}+\frac{5}{2},\frac{n}{2}-\frac{5}{2},\dots,-\frac{n}{2}-\frac{3}{2}),$$
	$$p((\lambda+\rho)^-)^{ev}=(\underbrace{1,\dots,1}_{n-4},1,0,1,0,1,0,1,0)=(\underbrace{1,\dots,1}_{n-3},0,1,0,1,0,1).$$
	$
	{\rm GKdim}\:L(\lambda)=\frac{1}{2}n^2+\frac{1}{2}n-6.
$
	
	\item When $z=2-\frac{n}{2}$, we will have
	$$p((\lambda+\rho)^-)^{ev}=(\underbrace{1,\dots,1}_{n-5},\underbrace{1,0,\dots,1,0}_{10})=(\underbrace{1,\dots,1}_{n-4},0,1,0,1,0,1,0,1).$$
$
	{\rm GKdim}\:L(\lambda)=\frac{1}{2}n^2+\frac{1}{2}n-10.
$	
	\item When $z=-2$, we will have
	$$p((\lambda+\rho)^-)^{odd}=(1,1,1,\underbrace{1,0,\dots,1,0}_{2n-6}).$$
	$
	{\rm GKdim}\:L(\lambda)=4n-6.
$
\end{enumerate}	

In general, we  find that  ${\rm GKdim}\:L(\lambda)=	kn-\frac{k(k-1)}{2}$ if $z=-\frac{k}{2}$ for $ 1\le k\le n-1$.

When $z=-\frac{1}{2}+k$, $z=k$ $(k\ge 0)$ and $z\notin \frac{1}{2}\mathbb{Z}$, we have obtained ${\rm GKdim}\:L(\lambda)$ in Proposition \ref{C}.	\\

For the case $2\leq p\leq n-1$, we only consider the case when $3p\leq 2n$ and $p$ is even.
\begin{enumerate}
	\item When $z\in\mathbb{Z}$,  $z=\frac{p}{2}-n$ is the first reducible point of $M_I(z\xi_p)$ by Proposition \ref{C}. We have the follows.

	\begin{enumerate}
		\item When $z > -1$,  we have
		$${\rm GKdim}\:L(\lambda)=0.$$
		
		\item When $z=-1$, we have	
		$${\rm GKdim}\:L(\lambda)=2n-1.$$
		
		\item When $z=\frac{p}{2}-n+1$, we will have
		$$(\lambda+\rho)^-=(\frac{p}{2}+1,\dots,-\frac{p}{2}+2,n-p,\dots,1,-1,\dots,-n+p,\frac{p}{2}-2,\dots,-\frac{p}{2}-1),$$
		and				
		$$	{\rm GKdim}\:L(\lambda)=2np-\frac{3}{2}p^2+\frac{1}{2}p-1.$$	
		
		\item When $z=\frac{p}{2}-n$, we will have
		$$(\lambda+\rho)^-=(\frac{p}{2},\dots,-\frac{p}{2}+1,n-p,\dots,1,-1,\dots,-n+p,\frac{p}{2}-1,\dots,-\frac{p}{2}),$$
		and
		$$	{\rm GKdim}\:L(\lambda) =2np-\frac{3}{2}p^2+\frac{1}{2}p.$$		
	\end{enumerate}
	
	In general, we can have the follows.	
	\begin{enumerate}
		\item When $z=-k$, $n-p\ge k-1$ and $p\ge k>0$, we have
		$$p(\lambda+\rho)^-=(2^{2k-1},1^{2n-4k+2}),$$
		and ${\rm GKdim}\:L(\lambda)=k(2n-2k+1).$	
		
		\item When $z=-k$, $n-p\ge k-1$ and $p<k$, we have
		$$p(\lambda+\rho)^-=(2^{2p},1^{2n-4p}),$$
		and ${\rm GKdim}\:L(\lambda)=p(2n-2p+1).$
		
		\item When $z=-2k$, $n-p\le 2(k-1)$ and $p\ge 2k$, we have
		$$p(\lambda+\rho)^-=(3^{4k-2n+2p-2},2^{4n-4p-4k+3},1^{2p-4k}),$$
		and ${\rm GKdim}\:L(\lambda)=2n(4k+2p-n)-2(2k+p^2)+4k(p+1)-3(n-p)-1.$
		
		If we write $z=-(n-p)-k \in 2\mathbb{Z}$, $2\leq k\leq 2p-n	$, we will have $${\rm GKdim}\:L(\lambda)= 2np-2p^2+k(2p-2k-2)-5(n-p)-1.$$
		
		\item When $z=-2k-1$, $n-p+1\le 2k$ and $p\ge 2k+1$, we have
		$$p(\lambda+\rho)^-=(3^{4k-2n+2p},2^{4n-4p-4k+1},1^{2p-4k-2}),$$
		and ${\rm GKdim}\:L(\lambda)=(2k+p)(4n-4k-2p)+(p-1)(4k+1)-n(2n-1).$
		
		If we write $z=-(n-p)-k \in 1+2\mathbb{Z}$, $2\leq k\leq 2p-n	$, we will have $${\rm GKdim}\:L(\lambda)= 2np-2p^2+k(2p-2k+2)-(n-p)-1.$$
		
		\item When $z=-(n-p)-k$, $2p-n<k\leq \frac{1}{2}p$, we have
		$$p(\lambda+\rho)^-=(3^{2k-2},2^{2p-4k+4},1^{2n-4p+2k-2}),$$
		and ${\rm GKdim}\:L(\lambda)=2np-2p^2+(2k-1)(p-k)+2k-1=2np-2p^2+k(2p-2k+3)-p-1.$

		\item When $z=-(n-p)-k$, $k\ge \frac{1}{2}p+1$, we have
		$$p(\lambda+\rho)^-=(3^p,1^{2n-3p}),$$
		and ${\rm GKdim}\:L(\lambda)=2np-\frac{3}{2}p^2+\frac{1}{2}p.$

	\end{enumerate}

	\item When $z\notin\frac{1}{2}\mathbb{Z}$, we have	
	$$(\lambda+\rho)^-_{(0)}=(n-p-1,\dots,0,0,\dots,-n+p+1)\text{~and~} (\lambda+\rho)_{(z)}=(z+n-1,\dots,z+n-p)\in[\lambda]_3.$$	
	\begin{align*}
	{\rm GKdim}\:L(\lambda) &=n^2-n-(1+\dots+p-1)-(2+\dots+2n-2p-2)\notag\\&=n^2-n-\frac{p(p-1)}2-2\cdot\frac{(n-p)(n-p-1)}{2}\notag\\&=2np-\frac{3}{2}p^2+\frac{1}{2}p=\dim(\mathfrak{u}).	
	\end{align*}
	
	\item When  $z\in\frac{1}{2}+\mathbb{Z}$, $z=\frac{p}{2}-n+\frac{1}{2}$ is the first reducible point of $M_I(z\xi_p)$ by Proposition \ref{C}. We have the follows. 
	
	\begin{enumerate}
		\item	
		When $z\ge -\frac{3}{2}$, we have	
		$${\rm GKdim}\:L(\lambda)=p(2n-2p+1).$$

		\item When $z=\frac{p}{2}-n+\frac{1}{2}$, we will have
		$$(\lambda+\rho)^-=(\frac{p}{2}+\frac{1}{2},\dots,-\frac{p}{2}+\frac{3}{2},n-p,\dots,1,-1,\dots,-n+p,\frac{p}{2}-\frac{3}{2},\dots,-\frac{p}{2}-\frac{1}{2}).$$
		So
		$$	{\rm GKdim}\:L(\lambda) =2np-\frac{3}{2}p^2+\frac{1}{2}p-1.$$

		%			\item When $z=\frac{p}{2}-n-\frac{1}{2}$, we will have	
		%			$$(\lambda+\rho)^-=(\frac{p}{2}-\frac{1}{2},\dots,-\frac{p}{2}+\frac{1}{2},n-p,\dots,1,-1,\dots,-n+p,\frac{p}{2}-\frac{1}{2},\dots,-\frac{p}{2}+\frac{1}{2}).$$
		%And		
		%		
		%			$$	{\rm GKdim}\:L(\lambda) =2np-\frac{3}{2}p^2+\frac{1}{2}p.$$
	\end{enumerate}		
	
	In general, we have the follows.
	\begin{enumerate}
		\item When $z=-k-\frac{1}{2}$, $k\leq n-p$, we have
		$${\rm GKdim}\:L(\lambda) =p(2n-2p+1).$$
		\item When $z=-k-\frac{1}{2}$, $k\ge n-p+1$ and $n\le\frac{1}{2}p+k $, we have
		$${\rm GKdim}\:L(\lambda) =2np-\frac{3}{2}p^2+\frac{1}{2}p.$$

		\item When $z=-(n-p)-k-\frac{1}{2}$, $1\leq k \le \frac{p-2}{2}$, we have
		$$p(\lambda+\rho)^-_{(\frac{1}{2})}=(2^{2k-2n+2p},1^{4n-4k-2p}),$$
		and ${\rm GKdim}\:L(\lambda) =2np-2p^2+(p-k)(2k+1)=2np-2p^2+k(2p-2k-1)+p.$

	\end{enumerate}	
\end{enumerate}

\end{proof}

From \cite{EH}, we know that the highest weight module $L(-\frac{k}{2}\xi_n)$ is the $k$-th Wallach representation of $Sp(n,\mathbb{R})$ for $1
\leq k \leq n-1$, which is  unitary.

\begin{prop}
		Let $ \mathfrak{g}=\mathfrak{so}(2n,\mathbb{C}) $ and $L(z\xi_p)$ be  a scalar highest weight module.	Then we have the follows.
	\begin{enumerate}
		\item If $p=1$, then
		
		\[	{\rm GKdim}\:L(z\xi_1)=\left\{
		\begin{array}{ll}
		2n-2 &\textnormal{if}\:\:z<2-n \emph{~or~}z\notin \mathbb{Z}\\	 
		
		2n-3 &\textnormal{if}\:\:2-n\le z\le -1\\
		
		0 &\textnormal{if}\:\:z\in \mathbb{Z}_{\geq 0}.\\
		\end{array}	
		\right.
		\]
		
		\item If $2\le p\le n-2$, then
		\begin{enumerate}
			\item When $3p\ge 2n$ and $p$ is even, we have
			\[	{\rm GKdim}\:L(z\xi_p)=\left\{
			\begin{array}{ll}
			0 &\textnormal{if}\:\:z\in \mathbb{Z}_{\geq 0}\\
			2np-\frac{3}{2}p^2-\frac{1}{2}p &\textnormal{if}\:\:z<1+\frac{p}{2}-n \emph{~or~} z\notin \frac{1}{2}\mathbb{Z}\\
			2nk-2k^2-k &\textnormal{if}\:\:z=-k, 1\leq k< n-p\\

			%			2np-\frac{3}{2}p^2-\frac{1}{2}p &\textnormal{if}\:\:z=-k, n-p\le k-\frac{p}{2}\\
			%		   2np-2p^2+k(2p-2k-2)+n-p-1 &\textnormal{if}\:\:z=-(n-p)-k, k<\frac{p}{2}\emph{~and~}0\le k<n-p\\
			
			2np-2p^2+k(2p-2k-2)+n-p-1 &\textnormal{if}\:\:z=p-n-k, 0\le k<n-p\\

			2np-2p^2+k(2p-2k-1) &\textnormal{if}\:\:z=p-n-k, n-p\leq k<\frac{p}{2}\\

			2np-2p^2 &\textnormal{if}\:\:z=-k+\frac{1}{2}, k\leq  n-p\\	 
			
			%			2np-\frac{3}{2}p^2-\frac{1}{2}p &\textnormal{if}\:\:z=p-n-\frac{k}{2}\in\frac{1}{2}+\mathbb{Z}, p\le k+1\\
			%	2np-2p^2+(k+1)(p-1-\frac{k}{2}) &\textnormal{if}\:\:z=p-n-\frac{k}{2}\in\frac{1}{2}+\mathbb{Z}, 1\le k\le p-3\\		 
			
			2np-2p^2+k(2p-2k-5)+2p-3 &\textnormal{if}\:\:z=p-n-k-\frac{1}{2}, 0\le k\le \frac{p-4}{2}.

			\end{array}	
			\right.
			\]
			
			\item When $p$ is odd, we have
			\[\small{	{\rm GKdim}\:L(z\xi_p)=\left\{
				\begin{array}{ll}
				0 &\textnormal{if}\:\:z\in \mathbb{Z}_{\geq 0}\\
				2np-\frac{3}{2}p^2-\frac{1}{2}p &\textnormal{if}\:\:z<\frac{3+p}{2}-n \emph{~or~} z\notin \frac{1}{2}\mathbb{Z}\\

				2nk-2k^2-k &\textnormal{if}\:\:z=-k, k<n-p\emph{~and~}1\le k\le p\\
				2np-2p^2-p &\textnormal{if}\:\:z=-k, p<k<n-p\\
				%		 2np-\frac{3}{2}p^2-\frac{1}{2}p &\textnormal{if}\:\:z=-k, n-p\le k-\frac{p+1}{2}\\
				%		 2np-\frac{3}{2}p^2-\frac{1}{2}p &\textnormal{if}\:\:z=-(n-p)-k, 2p-n<k,  \emph{~and~}0< k<n-p\\
				2np-2p^2+k(2p-2k-3)+p-1 &\textnormal{if}\:\:z=p-n-k, \max\{0,2p-n\}<k\le\frac{p-3}{2}\\
				2np-2p^2+k(2p-2k-2)+n-p-1 &\textnormal{if}\:\:z=p-n-k, k\leq 2p-n\emph{~and~}0\le k<n-p\\
				2np-2p^2+k(2p-2k-1) &\textnormal{if}\:\:z=p-n-k,  n-p\leq k\leq \frac{p-3}{2}\\
				2np-2p^2 &\textnormal{if}\:\:z=-k+\frac{1}{2}, k\leq  n-p\\	 			
				%			2np-\frac{3}{2}p^2-\frac{1}{2}p &\textnormal{if}\:\:z=p-n-\frac{k}{2}\in\frac{1}{2}+\mathbb{Z}, p\le k\\
				%			2np-2p^2+(k+1)(p-1-\frac{k}{2}) &\textnormal{if}\:\:z=p-n-\frac{k}{2}\in\frac{1}{2}+\mathbb{Z}, 1\le k\le p-4\\
				2np-2p^2+k(2p-2k-5)+2p-3 &\textnormal{if}\:\:z=p-n-k-\frac{1}{2}, 0\le k\le \frac{p-5}{2}.\\
				
				\end{array}	
				\right.}
			\]

			\item When $3p<2n$ and $p$ is even, we have
			\[	{\rm GKdim}\:L(z\xi_p)=\left\{
			\begin{array}{ll}
			0 &\textnormal{if}\:\:z\in \mathbb{Z}_{\geq 0}\\
			2np-\frac{3}{2}p^2-\frac{1}{2}p &\textnormal{if}\:\:z<\frac{3+p}{2}-n \emph{~or~} z\notin \frac{1}{2}\mathbb{Z}\\

			%					2np-2p^2 &\textnormal{if}\:\:z\in p-n+\frac{1}{2}+\mathbb{Z}_{\geq 0}\\	 			
			%					2np-\frac{3}{2}p^2-\frac{1}{2}p &\textnormal{if}\:\:z=p-n-\frac{k}{2}\in\frac{1}{2}+\mathbb{Z}, p\le k+1\\
			%					2np-2p^2+(k+1)(p-1-\frac{k}{2}) &\textnormal{if}\:\:z=p-n-\frac{k}{2}\in\frac{1}{2}+\mathbb{Z}, 1\le k\le p-3\\
			
			2nk-2k^2-k &\textnormal{if}\:\:z=-k, 1\leq k< n-p\emph{~and~}k<p\\
			2np-2p^2-p &\textnormal{if}\:\:z=-k, p\le k< n-p\\		
			%			2np-\frac{3}{2}p^2-\frac{1}{2}p	 &\textnormal{if}\:\:z=-k,  n-p\le k-\frac{1}{2}p\\
			2np-2p^2+k(2p-2k-2)+n-p-1 &\textnormal{if}\:\:z=p-n-k,   0\leq k \leq 2p-n\\	
			2np-2p^2+k(2p-2k-3)+p-1 &\textnormal{if}\:\:z=p-n-k, 2p-n<0\le k\le\frac{p-3}{2}\\
			
			2np-2p^2 &\textnormal{if}\:\:z=-k+\frac{1}{2}, k\leq  n-p\\
			2np-2p^2+k(2p-2k-5)+2p-3 &\textnormal{if}\:\:z=p-n-k-\frac{1}{2}, 0\le k\le \frac{p-4}{2}.\\
			
			\end{array}	
			\right.
			\]

			%				\item if $p$ is odd, then	
			%			\[	{\rm GKdim}\:L(z\xi_1)=\left\{
			%		\begin{array}{ll}
			%			2np-2p^2 &\textnormal{if}\:\:z\in p-n+\frac{1}{2}+\mathbb{Z}_{\geq 0}\\	 			
			%			2np-\frac{3}{2}p^2-\frac{1}{2}p &\textnormal{if}\:\:z=p-n-\frac{k}{2}, p\le k\\
			%			2np-2p^2+(k+1)(p-1-\frac{k}{2}) &\textnormal{if}\:\:z=p-n-\frac{k}{2}, p>k \emph{~and~}1\le k\le p-1\\
			%			0 &\textnormal{if}\:\:z\in \mathbb{Z}_{\geq 0}\\
			%			2nk-2k^2-k &\textnormal{if}\:\:z=-k, n-p\ge k+1\emph{~and~}p\ge k\\
			%		2np-2p^2-p &\textnormal{if}\:\:z=-k, p=3, n-p\ge k+1\emph{~and~}k\ge 3\\
			%		2np-\frac{3}{2}p^2-\frac{1}{2}p &\textnormal{if}\:\:z=-(n-p)-k, p< n-p+k,  \emph{~and~}0< k\le n-p+1\\
			%		2np-\frac{3}{2}p^2-\frac{1}{2}p-1 &\textnormal{if}\:\:z=-(n-p)-k, p< n-p+k,  \emph{~and~}k=0\\
			%		2np-2p^2-2k(k-p+1)+n-p-1 &\textnormal{if}\:\:z=-(n-p)-k, p\ge n-p+k,  k<\frac{p+1}{2}\emph{~and~}0\le k<n-p.\\
			%		\end{array}	
			%		\right.
			%		\]		

		\end{enumerate}

		\item If $p=n-1\:\:or\:\:p=n$, then
		\begin{enumerate}
			\item When $n$ is even, we have				
			\[		{\rm GKdim}\:L(z\xi_p)=\left\{
			\begin{array}{ll}
			\frac{1}{2}n^2-\frac{1}{2}n &\textnormal{if}\:\:z<2-n \emph{~or~}z\notin \mathbb{Z}\\	     	  	   
			k(2n-2k-1) &\textnormal{if}\:\:z=-2k\emph{~or~}-2k+1, 1\le k\le \frac{n-2}{2}\\
			
			0 &\textnormal{if}\:\:z\in \mathbb{Z}_{\geq 0}.\\
			\end{array}	
			\right.
			\]	
			
			\item When $n$ is odd, we have
			\[		{\rm GKdim}\:L(z\xi_p)=\left\{
			\begin{array}{ll}
			\frac{1}{2}n^2-\frac{1}{2}n &\textnormal{if}\:\:z<3-n \emph{~or~}z\notin \mathbb{Z}\\	     	  	   
			k(2n-2k-1) &\textnormal{if}\:\:z=-2k\emph{~or~}-2k+1, 1\le k\le \frac{n-3}{2}\\
			
			0 &\textnormal{if}\:\:z\in \mathbb{Z}_{\geq 0}.\\
			\end{array}	
			\right.
			\]	
		\end{enumerate}		
	\end{enumerate}		     
\end{prop}	
\begin{proof}
	First we consider the case when $p=1$. From Proposition \ref{D}, $M_I(\lambda)$ is irreducible when $z<2-n$ and $z\notin \mathbb{Z}$. So we only consider the case when $z\in \mathbb{Z}$.
\begin{enumerate}
	\item When $z=2-n$, we will have
	$$(\lambda+\rho)^-=(1,n-2,\dots,0,0,\dots,2-n,-1),$$
	$$p((\lambda+\rho)^-)^{odd}=(1,1,1,\underbrace{0,1,\dots,0,1}_{2n-6}).$$
	$
	{\rm GKdim}\:L(\lambda)=n^2-n-1-2\cdot\frac{(n-1)(n-2)}{2}=2n-3.
	$
	
	\item When $z=3-n$, we will have
	$$(\lambda+\rho)^-=(2,n-2,\dots,0,0,\dots,2-n,-2),$$
	$$p((\lambda+\rho)^-)^{odd}=(1,1,1,\underbrace{0,1,\dots,0,1}_{2n-6}).$$	
$
	{\rm GKdim}\:L(\lambda)=n^2-n-1-2\cdot\frac{(n-1)(n-2)}{2}=2n-3.
$	
	
	\item When $z=-1$, similarly we will have
	
	${\rm GKdim}\:L(\lambda)=2n-3$.
	
	\item When $z\ge 0$, We can easily get ${\rm GKdim}\:L(\lambda)=0$.
\end{enumerate}	

Finally, combining with Proposition \ref{D} we can get the formula of  ${\rm GKdim}\:L(\lambda)$ when $p=1$. \\

Now we consider the case when $p=n-1$ and $p=n$.	
\begin{enumerate}
	\item When $p=n-1$ and $n$ is even, we have the follows.
	\begin{enumerate}
		\item When $z=2-n$, we will have
		${\rm GKdim}\:L(\lambda)=\frac{1}{2}n^2-\frac{1}{2}n-1$ from the proof of Proposition \ref{D}.
		
		\item When $z=3-n$, we will have
		$$(\lambda+\rho)^-=(\frac{n}{2}+\frac{1}{2},\dots,\frac{3}{2}-\frac{1}{2}n,\frac{1}{2}n-\frac{3}{2},\dots,-\frac{1}{2}n-\frac{1}{2}),$$
		$$p((\lambda+\rho)^-)^{ev}=(\underbrace{1,\dots,1}_{n-2},1,0,1,0).$$
So	$
		{\rm GKdim}\:L(\lambda)=\frac{1}{2}n^2-\frac{1}{2}n-1.
	$
		\item When $z=4-n$, we will have
		$$(\lambda+\rho)^-=(\frac{n}{2}+1,\dots,2-\frac{1}{2}n,\frac{1}{2}n-2,\dots,-\frac{1}{2}n-1),$$
		$$p((\lambda+\rho)^-)^{ev}=(\underbrace{1,\dots,1}_{n-3},0,1,0,1,0,1).$$
So \begin{equation}\label{16}
{\rm GKdim}\:L(\lambda)=\frac{1}{2}n^2-\frac{1}{2}n-6.
\end{equation}

		\item When $z=5-n$, we will have
		$$(\lambda+\rho)^-=(\frac{n}{2}+\frac{3}{2},\dots,\frac{5}{2}-\frac{1}{2}n,\frac{1}{2}n-\frac{5}{2},\dots,-\frac{1}{2}n-\frac{3}{2}),$$
		\begin{align*}
		p((\lambda+\rho)^-)^{ev}&=(\underbrace{1,\dots,1}_{n-4},\underbrace{1,0,\dots,1,0}_{8})\\&=(\underbrace{1,\dots,1}_{n-3},0,1,0,1,0,1).
		\end{align*}
		By (\ref{16}), ${\rm GKdim}\:L(\lambda)=\frac{1}{2}n^2-\frac{1}{2}n-6$.
		
		\item When $z=-2$ and $z=-1$, we will have
		
		${\rm GKdim}\:L(\lambda)=2n-3$.
		
		\item When $z\ge 0$, we will have
		
		${\rm GKdim}\:L(\lambda)=0$.
	\end{enumerate}
	
	From the above discussion, we can find the formula of ${\rm GKdim}\:L(\lambda)$.
	
	\item When $p=n-1$ and $n$ is odd.
	
	The process is similar to the case when $n$ is even. We omit this part of the process.
	
	\item When $p=n$.
	
	The process is similar to the case when $p=n-1$, whether $n$ is odd or even. We omit this part of the process.
\end{enumerate}

If $2\le p\le n-2$, we only consider the case when $p$ is odd.

% When $z\in\mathbb{Z}$,  $z=\frac{p}{2}-n+\frac{3}{2}$ is the first reducible point of $M_I(z\xi_p)$ by Proposition \ref{C}. We have the follows.

\begin{enumerate}
	\item  When $z\in\mathbb{Z}$,  $z=\frac{p}{2}-n+\frac{3}{2}$ is the first reducible point of $M_I(z\xi_p)$ by Proposition \ref{D}. We have the follows.
	\begin{enumerate}
		\item When $z\leq 0$,  we have
		${\rm GKdim}\:L(\lambda)=0.$
		
		\item When $z=-1$, we have			
		${\rm GKdim}\:L(\lambda)=2n-3.$
		
		\item When $z=\frac{p}{2}-n+\frac{3}{2}$, we will have
		$$(\lambda+\rho)^-=(\frac{p}{2}+\frac{1}{2},\dots,-\frac{p}{2}+\frac{1}{2},n-p-1,\dots,0,0,\dots,-n+p+1,\frac{p}{2}-\frac{1}{2},\dots,-\frac{p}{2}-\frac{1}{2}).$$		
		We divide the discussion into three cases:    
		\begin{enumerate}
			\item $3p>2n$. We have
			$$ p((\lambda+\rho)^-)^{ev}=(2,\underbrace{1,2,\dots,1,2}_{2n-2p-2},\underbrace{1,\dots,1}_{3p-2n-1},1,0,1).$$
			
			By (\ref{8}), 
			$
			{\rm GKdim}\:L(\lambda)\notag=2np-\frac{3}{2}p^2-\frac{1}{2}p-1.	
			$

			\item $2n-5<3p<2n$. We have 
			$$ p((\lambda+\rho)^-)^{ev}=(2,\underbrace{1,2,\dots,1,2}_{p-3},\underbrace{1,\dots,1}_{2n-3p+1},\underbrace{0,1,\dots,0,1}_{3p-2n+3}).$$
		So	$
			{\rm GKdim}\:L(\lambda) =2pn-\frac{3}{2}p^2-2p+n-\frac{5}{2}.
		$	
			
			\item $3p\le 2n-5$. We have 
			$$ p((\lambda+\rho)^-)^{ev}=(2,\underbrace{1,2,\dots,1,2}_{p-3},1,1,1,1,\underbrace{0,1,\dots,0,1}_{2n-3p-3}).$$	
		So
		$	{\rm GKdim}\:L(\lambda) =2pn-\frac{3}{2}p^2-\frac{1}{2}p-1.
		$	
		\end{enumerate}
	\end{enumerate}	
	In general, we can have the follows.
	\begin{enumerate}
		\item When $z=-k$, $n-p\ge k+1$ and $1\le k\le p$, we have
		$$p(\lambda+\rho)^-=(2^{2k+1},1^{2n-4k-2}),$$ 	
		and ${\rm GKdim}\:L(\lambda) =k(2n-2k-1).$	
		
		\item When $z=-k$, $n-p\ge k+1$ and $p<k$, we have
		$$p(\lambda+\rho)^-=(2^{2p+1},1^{2n-4k-2}),$$ 	
		and ${\rm GKdim}\:L(\lambda) =p(2n-2p-1).$		
		
		%\item When $z=-k$, $n-p\le k-\frac{p+1}{2}$, we have
		%$$p(\lambda+\rho)^-=(4,3^{2n-2p-2},1^{3p-2n+1}),$$ 	
		%and ${\rm GKdim}\:L(\lambda) =2np-\frac{3}{2}p^2-\frac{1}{2}p.$	
		
		\item When $z=-(n-p)-k$, $0\le k<n-p$ and $p\ge n-p+k$, we have
		$$p(\lambda+\rho)^-=(4,3^{2k},2^{2n-2p-2k-2},1^{4p-2k-2n}),$$ 	
		and ${\rm GKdim}\:L(\lambda) =2np-2p^2-2k(k-p+1)+n-p-1.$	
		
		\item When $z=-(n-p)-k$, $ \frac{p-3}{2}\ge k\ge n-p$, we have
		$$p(\lambda+\rho)^-=(4,3^{2n-2p-2},2^{2k+2p-2n+2},1^{2p-4k-2}),$$ 	
		and ${\rm GKdim}\:L(\lambda) =2np-2p^2+k(2p-2k-1).$	
		
		%\item When $z=-(n-p)-k$, $0< k< n-p$ and $2p-n<k$, we have
		%$$p(\lambda+\rho)^-=(4,3^{p-1},1^{2n-3p-1}),$$ 	
		%and ${\rm GKdim}\:L(\lambda) =2np-\frac{3}{2}p^2-\frac{1}{2}p.$	
		
		\item When $z=-(n-p)-k$, $\max\{0,2p-n\}<k\le \frac{p-3}{2}$, we have
		$$p(\lambda+\rho)^-=(4,3^{2k},2^{2p-4k-2},1^{2n--4p+2k}),$$ 	
		and ${\rm GKdim}\:L(\lambda) =2np-2p^2+k(2p-2k-3)+p-1.$		
	\end{enumerate}	
	\item When $z\notin\frac{1}{2}\mathbb{Z}$, we have
	$${\rm GKdim}\:L(\lambda) =2np-\frac{3}{2}p^2-\frac{1}{2}p.$$	
	
	\item When $z\in\frac{1}{2}+\mathbb{Z}$, $z=\frac{p}{2}-n+2$ is the first reducible point of $M_I(z\xi_p)$ and we have follows.
	\begin{enumerate}
		\item When $z\ge p-n$, we have
		$${\rm GKdim}\:L(\lambda) =2np-2p^2.$$
		
		\item When $z=\frac{p}{2}-n+2$, we have
		$$(\lambda+\rho)^-_{(\frac{1}{2})}=(\frac{p}{2}+1,\cdots,-\frac{p}{2}+2,\frac{p}{2}-2,\cdots,-\frac{p}{2}-1).$$
		So $${\rm GKdim}\:L(\lambda) =2np-\frac{3}{2}p^2-\frac{1}{2}p-3.$$	
	\end{enumerate}
	In general, we have the follows.
	\begin{enumerate}
		%	\item When $z=-(n-p)-\frac{k}{2}\in\frac{1}{2}+\mathbb{Z}$, $p\le k$, we have
		%	$${\rm GKdim}\:L(\lambda) =2np-\frac{3}{2}p^2-\frac{1}{2}p.$$
		%	
		%	\item When  $z=-(n-p)-\frac{k}{2}\in\frac{1}{2}+\mathbb{Z}$, $1\le k\le p-1$, we have
		%	$$(\lambda+\rho)^-_{(\frac{1}{2})}=(2^{k+1},1^{2p-2k-2}),$$
		%	and ${\rm GKdim}\:L(\lambda) =2np-2p^2+(k+1)(p-1-\frac{k}{2}).$
		\item When $z=-k+\frac{1}{2}$, $k\leq  n-p$, we have 
		${\rm GKdim}\:L(\lambda) =2np-2p^2$.
		
		\item When $z=p-n-k-\frac{1}{2}$, $0\le k\le \frac{p-5}{2}$, we have ${\rm GKdim}\:L(\lambda) =	2np-2p^2+k(2p-2k-5)+2p-3 $.		 
		
	\end{enumerate}

	%		At this point, we have completed the proof of all propositions.	
	\end{enumerate}	
	
\end{proof}
From \cite{EH}, we know that the highest weight module $L(-(n-2)\xi_1)$ is the unique Wallach representation of $SO(2,2n-2)$, which is  unitary. $L(-2k\xi_n)$ is the $k$-th Wallach representation of $SO^*(2n)$ for $1\leq k\leq  [\frac{n}{2}]-1$, which is  unitary.

The Gelfand-Kirillov dimensions of Wallach representations are also uniformly listed in \cite[Proposition 4.1]{BX}.

\end{document}